\documentclass[11pt, notitlepage]{article}
\makeatletter\@addtoreset{section}{part}\makeatother%
\usepackage{amssymb}
\usepackage[margin=1.5cm]{geometry}
\usepackage{amsmath}
\usepackage{amsthm}
\usepackage{bbm}
\usepackage{amsfonts}
\usepackage{mathtools}
\usepackage{graphicx} 
\usepackage{subcaption} 
\usepackage{relsize}
\usepackage{titling}
\usepackage[]{algorithm2e}
\usepackage{enumitem}
\usepackage{hyperref}

\usepackage{framed}
\usepackage{appendix}
\appendixtitleon
\appendixtitletocon
\newcommand{\subtitle}[1]{%
  \posttitle{%
    \par\end{center}
    \begin{center}\large#1\end{center}
    \vskip0.5em}%
}

\usepackage[dvipsnames]{xcolor}

\newcommand{\xddots}{%
  \raise 4pt \hbox {.}
  \mkern 6mu
  \raise 1pt \hbox {.}
  \mkern 6mu
  \raise -2pt \hbox {.}
}
\usepackage{hyperref}
\hypersetup{
colorlinks=true,
citecolor=OliveGreen, 
linkcolor=blue 
}

\title{Bilinear Coagulation Equations}

\author{Daniel Heydecker\thanks{University of Cambridge, dh489@cam.ac.uk. This work was supported by the UK Engineering and Physical Sciences Research Council (EPSRC) grant EP/L016516/1 for the University of Cambridge Centre for Doctoral Training, the Cambridge Centre for Analysis}, Robert I. A. Patterson\thanks{robert.patterson@wias-berlin.de. This research was supported by the Deutsche Forschungsgemeinschaft (DFG) grant CRC 1114 ``Scaling Cascades in Complex Systems'', Project C08.}}
\date{\today}
\newcommand{\EE}{\ensuremath{\mathbb{E}}}
\newcommand{\NN}{\ensuremath{\mathbb{N}}}
\newcommand{\PP}{\ensuremath{\mathbb{P}}}
\newcommand{\dd}{\ensuremath{\mathrm{d}}}

\newcommand{\abs}[1]{\left\lvert{#1}\right\rvert}

\newcommand{\supnorm}[1]{\left\lVert{#1}\right\rVert_\infty}

\newtheorem{thm}{Theorem}
\newtheorem{theorem}{Theorem}[section]
\newtheorem{lem}{Lemma}[section]

\newtheorem{cor}[lem]{Corollary}
\newtheorem{rmk}[lem]{Remark}
\newtheorem{remark}[theorem]{Remark}
\newtheorem{defn}{Definition}
\newtheorem{definition}[theorem]{Definition}
\newtheorem{hyp}{Assumption}

\begin{document}

\maketitle
\begin{abstract} We consider coagulation equations of Smoluchowski or Flory type where the total merge rate has a bilinear form $\pi(y)\cdot A\pi(x)$ for a vector of conserved quantities $\pi$, generalising the multiplicative kernel. For these kernels, a gelation transition occurs at a finite time $t_\mathrm{g}\in (0,\infty)$, which can be given exactly in terms of an eigenvalue problem in finite dimensions. We prove a hydrodynamic limit for a stochastic coagulant, including a corresponding phase transition for the largest particle, and exploit a coupling to random graphs to extend analysis of the limiting process beyond the gelation time.

\end{abstract}

\section{Introduction and Main Results}
Smolouchowski \cite{vS16} introduced the basic mathematical model for coagulating particles, giving an ordinary differential equation describing the distribution of particle masses, which arises from considering a microscopic particle system. The physics of the underlying system enters into the model through a choice of interaction kernel $K(x,y)$, describing the speed of the coagulation $x, y\mapsto x+y$; the particular case $K(x,y)=xy$ is known as the \emph{multiplicative} kernel, and is particularly well-studied. It is well-known that the resulting Smolouchowski equation corresponds to the distribution of cluster sizes in the Erd\H{o}s-R\'eyni random graphs $\mathcal{G}(N, t/N)$ in the $N\rightarrow \infty$ limit, which, together with the simplicity of the kernel, allows a fairly complete analysis of this case.   \medskip \\ In many physical situations, the rate of coagulation will depend on more than only the mass of the particles, and so it is desirable to generalise Smoluchowski's equation. In particular, we note the works \cite{N99, N00} which allow coagulation in more than one possible way, and where the total rate of coagulation between particles of types $x,y$ can be bounded in terms of a function $\varphi$ which is conserved as particles coagulation. \medskip \\ There are two natural ways to frame the study of coagulation: one can either start from an interacting particle system, where existence and uniqueness is elementary, but where characterising the many particle limit may require substantial effort, or one can work directly with the mean-field Smolouchowski and Flory equations, for which existence and uniqueness require more consideration, and can fail in some cases \cite{N99}. Relatedly, there are several different ways in which one can characterise gelation. At the level of the particle system, one can study the phase transition where the size of the largest particle goes from size $\ll N$ to a size comparable to $N$ \cite{L78}. At the level of the limiting equation, gelation refers to the point where the solution to the Smolouchowski or Flory equation $(\mu_t)_{t\ge 0}$ fails to conserve the total particle mass, which is known to be related \cite{L78,N00} to the divergence of the second moment of the particle mass. \medskip \\ We will study coagulation systems where, as above, coagulation can occur in several possible ways, and where the internal structure of particles can evolve, in a mass-preserving way, without coagulation. The important hypothesis is the  \emph{bilinear} structure: we ask that the total mass of the kernel $\overline{K}(x,y)$ can be expressed in the form $\pi(y)\cdot A \pi(x)$, for a fixed matrix $A$ and a vector $\pi$ of conserved quantities. In this case, the limiting Flory equation is well-posed, globally in time, and the stochastic particle system couples exactly to a class of random graphs introduced by \cite{BJR07} which generalise the Erd\H{o}s-R\'eyni graphs. We analyse the limiting equation in Theorem \ref{thrm: Smoluchowski equation} and prove a law of large numbers for the stochastic particle systems in Theorem \ref{thrm: convergence of stochastic coagulent}, together demonstrating that the three effects described above all occur at the same time $t_\mathrm{g}$.

\subsection{\textbf{Definitions}}

As mentioned above, our analysis rests on the \emph{bilinear} form of the total rate $\overline{K}$, which allows us to connect the Smoluchowski equation to random graphs in Section \ref{sec: coupling_to_random_graph}. The following definition makes this precise.
\begin{defn} \label{def: BCS}  A \emph{Bilinear Coagulation System} is a 5-tuple $(S, R, \pi,K, J)$ consisting of a complete metric space $S$, a continuous involution $R$ on $S$, a finite collection of continuous maps $\pi=(\pi_i)_{0\le i\le n+m}$, $n\ge 1, m\ge 0$ from $S$ to $\mathbb{R}$, and nonnegative kernels $K, J$ on $S\times S\times S$ and $S\times S$ respectively, such that the following hold.
\begin{enumerate}[label=\roman{*}).]
\item For all $0\le i\le n+m$ and all $x,y\in S$, \begin{equation} \pi_i=\pi_i(x)+\pi_i(y) \hspace{1cm}K(x,y,\cdot)+J(x, \cdot)\text{- almost everywhere}. \end{equation} \item For $1\le i\le n$, the map $\pi_i:S\rightarrow \mathbb{R}$ takes only nonnegative values, and $\pi_0$ takes values in the positive integers $\mathbb{N}.$
\item The involution $R$ satisfies  \begin{equation} \pi_i \circ R=\begin{cases} \pi_i & 0\le i\le n; \\ -\pi_i, & n+1\le i \le n+m \end{cases}  \end{equation} and, for all $x,y\in S$, \begin{equation} K(Rx,Ry,\cdot)=R_\#K(x,y,\cdot); \hspace{1cm} J(Rx,\cdot)=R_\#J(x,\cdot). \end{equation} Here, and throughout, the subscript $_\#$ denotes the pushforward of a measure.
\item There exists a constant $C$ such that, for all $x\in S$, \begin{equation} \sum_{i= n+1}^m \pi_i(x)^2\le C \varphi(x)^2 \end{equation} where $\varphi(x)=\sum_{i=0}^n \pi_i(x)$. Moreover, the sublevel sets $S_\xi=\{x\in S: \varphi(x)\le \xi\}$ are compact, for all $\xi\in [0,\infty)$.
\item For all $x,y\in S$, the total rate $\overline{K}(x,y)=K(x,y,S)$ may be expressed as
\begin{equation}  \label{eq: overline K}  \overline{K}(x,y)=\sum_{1\le i,j\le n+m}a_{ij}\pi_i(x)\pi_j(y)
\end{equation}  for a fixed $(n+m)\times (n+m)$ symmetric real matrix $A=(a_{ij})_{1\le i,j\le n+m}$. Moreover, the matrix $A$ is of the block-diagonal form \begin{equation} A=\begin{pmatrix} A^+ & 0 \\ 0 & A^\mathrm{par} \end{pmatrix} \end{equation} where $A^+, A^\mathrm{par}$ are $n\times n$ and $m\times m$ square matrices respectively, and all entries of $A^+$ are nonnegative. Finally, for all $1\le i\le n+m$, there exists $1\le j\le n+m$ such that $a_{ij}>0$, so that no row or column of $A$ vanishes. For $J$, we ask that the total rate $\overline{J}(x)=J(x,S)$ satisfies $\sup_x\frac{\overline{J}(x)}{\varphi(x)}<\infty.$
\item For $f\in C_c(S)$, the maps \begin{equation} Kf: S\times S\rightarrow \mathbb{R}, \hspace{1cm} (x,y)\mapsto \int_S f(z)K(x,y,dz); \end{equation} \begin{equation} Jf: S\rightarrow \mathbb{R}, \hspace{1cm} x\mapsto \int_S f(z)J(x,dz) \end{equation} are continuous.

\end{enumerate} \end{defn} \begin{rmk} We think of $\pi_0(x)$ as counting the number of particles at time $0$ which have been absorbed into $x$. As a result, we will ask in (A5.) below that our initial measure $\mu_0$ is supported on $\{\pi_0\}=1$, and $\pi_0$ artificially introduces monodisperse initial conditions. \medskip \\ If we are given a space $S$ equipped only with $\pi_1,...,\pi_{n+m}$, we can replace $S$ by $\NN\times S$, and setting $\pi_0(a,x)=a, \pi_i(a,x)=\pi_i(x), i=1,...,n+m, (a,x)\in \NN\times S$. In this way, and since $\pi_0$ does not enter the total rate $\overline{K}(x,y)$, the artificial requirements on $\pi_0$ above do not restrict the physics of the coagulation system.    \end{rmk} 
\paragraph{Stochastic Particle Systems.} With the setting defined above, we can introduce the interacting particle systems under consideration. \medskip \\  We study a system of coagulating particles $(x^N_j(t): j\le l^N(t))$, and the associated empirical measure
\begin{equation}\label{eq: sc1}\mu^N_t = \frac{1}{N}\sum_{j=1}^{l^N(t)} \delta_{x^N_j(t)}\end{equation} with the following dynamical rules. \begin{enumerate}[label=\roman*).] \item
The rate at which unordered pairs of particles $\left\{ x,y \right\}$ in $S$ merge to form a new particle in $A \subset S$ is $2K(x,y, A)/N$. \item A particle of type $x$ evolves can a particle of type $y\in A \subset S$ with a total rate $J(x,A)$. \end{enumerate} This is a generalisation of a Marcus--Lushnikov coagulation process \cite{L78} on $S$, which we will refer to as the \emph{stochastic coagulant}.
Note that a $1/N$ scaling of the pair interaction rate is used, which ensures that each molecule has a total evolution rate of order $1$.
Dividing jump rates by $N$ is equivalent to accelerating time by the same factor and this alternative formulation means that the jump rates in the definition of the ``stochastic coalescent'' in  \cite{A99} as well as of the ``stochastic $K$-coagulant'' in \cite{N00} omit the $1/N$ from the rates and rescale time when taking the $N\rightarrow \infty$ limit. 
\paragraph{Limiting kinetic equations.}
We now consider various forms of the limiting Smoluchowski equation.  Define a drift operator $L$, by specifying for all  $f\in C_c(S)$, 
 \begin{equation}\begin{split} \label{eq: drift wo gel}
    \langle f,L(\mu)\rangle=&\frac{1}{2}\int_{S^3}\{f(z)-f(x)-f(y)\}K(x,y,dz)\mu(dx)\mu(dy)\\& \hspace{1cm}+\int_{S^2}\{f(y)-f(x)\}J(x,dy). \end{split}
\end{equation}
The weak form of the Smolochowski equation for a process of measures $(\mu_t)_{t<T}$ on $S$ is to ask that
\begin{equation}
    \tag{Sm}\label{eq: E}  \forall f\in C_c(S),\hspace{0.1cm}  t<T,\hspace{1cm}\langle f,\mu_t \rangle =  \langle f, \mu_0\rangle  +\int_0^t \langle f,L(\mu_s)\rangle  ds.
\end{equation}  The equation (\ref{eq: E}) captures the effects of coagulations between finite clusters. However, as discussed above, we wish to include the possibility of a macroscopic component, which we term \emph{gel}. To include this effect, we modify the drift operator by specifying, for $f\in C_c(S)$, \begin{equation} \langle f,L_\mathrm{g}(\mu_t)\rangle =\langle f, L(\mu_t)\rangle -\int_{S}f(x)\overline{K}(x,y)\mu_t(dx)(\mu_0-\mu_t)(dy). \end{equation}The weak form of the Flory equation is to ask, similarly, \begin{equation} \tag{Fl} \label{eq: E+G}
    \forall f\in C_c(S),\hspace{0.1cm}  t<T,\hspace{1cm}\langle f,\mu_t\rangle = \langle f,\mu_0\rangle + \int_0^t \langle f,L_\mathrm{g}(\mu_s)\rangle ds.
\end{equation} Here, the additional term comes into play only after $\mu_t$ ceases to conserve the quantities $\langle \pi_i,\mu_t\rangle, 1\le i\le n+m$, and the extra term represents the interaction with the gel. This generalises the Smoluchowski coagulation equations \cite{vS16} in a way analogous to Flory \cite{ZS80}, and we use the term `$K$-coagulant' for a solution to (\ref{eq: E+G}), following \cite{N00}. \medskip \\  Precise conditions on measurability and integrability required to interpret these equations concretely are given in Appendix \ref{sec: requirements}.
\medskip \\ We write \begin{equation}\label{eq: gel data} \begin{split}
   g_t&=(M_t, E_t,P_t) =\langle \pi, \mu_0-\mu_t\rangle \\ &= \left(\langle \pi_i, \mu_0-\mu_t\rangle \right)_{i=0}^{n+m} \end{split}
\end{equation} for the gel data, where $M_t, E_t, P_t$ are the $0^\text{th}$, $1^\text{st}-n^\text{th}$, and $(n+1)^\text{th}-(n+m)^\text{th}$ coordinates, respectively. Following remarks in \cite{N00}, one may show that if $\mu_t$ is a solution to (\ref{eq: E+G}), then the maps $t\mapsto \langle \pi_i, \mu_t\rangle, i\le n$ are non-increasing, which guarantees that $M_t, E_t\ge 0$. We write  $S^\Pi$ for the state space of gel data, given by \begin{equation}
    S^\Pi=\NN \times \mathbb{R}^n \times \mathbb{R}^m
\end{equation} and use the same notation $\pi_i, 0\le i\le n+m$ for the projections onto the factors. When $x\in S$ and $g\in S^\Pi$, we use $\overline{K}(x,g)$ for the rate of absorption, given by (\ref{eq: overline K}) with the new meanings of $\pi_i(g).$ We will also write $\varphi$ for the linear combination $\varphi=\sum_{i\le n} \pi_i$, defined on both $S$ and $ S^\Pi$.

\begin{defn}[Conservative Solutions] Let $S$ be a bilinear coagulation system. We say that a solution $(\mu_t)_{t<T}$ to either \eqref{eq: E} or \eqref{eq: E+G} is \emph{conservative} if all the functions $t\mapsto \langle \pi_i, \mu_t\rangle, 0\le i\le n+m$ are constant on $[0,T)$. \end{defn}
 Thus, any solution to (\ref{eq: E}) or (\ref{eq: E+G}) is conservative up to some time $0\le t_\mathrm{g}\leq \infty$, and non-conservative thereafter.
 
 We will usually impose symmetry requirements (A1.) on the initial data which guarantee that $\langle \pi_i, \mu_t\rangle=0$ for all $t$, for all $i=n+1,..,n+m$.  As noted above, the functions $t\mapsto \langle \pi_i, \mu_t\rangle, i\le n$ are non-increasing, whenever $(\mu_t)_{t<T}$ is a local solution to either equation. Therefore, under hypothesis (A1.), a solution $(\mu_t)_{t<T}$ to either equation is conservative if, and only if, the map $t\mapsto \langle \varphi, \mu_t\rangle$ is constant on $[0,T)$.
 
Let $\mathcal{M}=\mathcal{M}_{\le 1}(S)$ be the space of measures on $S$ with total mass at most 1. We equip $\mathcal{M}$ with the \emph{vague} topology $\mathcal{F}(\mathcal{M}, C_c(S))$ induced by continuous, compactly supported functions on $S$, and fix a complete metric $d$ compatible with this topology. 
\subsection{\textbf{Statement of Results}}\label{sec: results}

We will make the following hypotheses on the initial data $\mu_0$. 
\begin{hyp}\label{hyp: A} We will ask that the initial data $\mu_0$ is a sub-probability measure on a bilinear coagulation space $S$, satisfying the following hypotheses.\begin{enumerate}[label=(A\arabic*.)] 
\item The measure $\mu_0$ is even under the transformation $R$: $R_\#\mu_0=\mu_0.$
\item For all $i\le n$, we have $\langle \pi_i^3, \mu_0\rangle <\infty.$
\item  The set $\{\pi_i: 1\le i\le n\}$ is linearly independent in the space $L^2(m)$. In particular, none of the functions $\pi_i: 1\le i\le n$ are $0$ $\mu_0$-almost everywhere. \item The kernel $K$ is $\mu_0$-irreducible: if $A\subset S$ is such that, for all $x\in A$ and $y\in A^\mathrm{c}$, $\overline{K}(x,y)=0$, then either $\mu_0(A)=0$ or $\mu_0(A^\mathrm{c})=0.$  Moreover, $\mu_0$ is not a point mass. \item The initial data $\mu_0$ is supported on $\{x\in S: \pi_0(x)=1\}.$ \end{enumerate} \end{hyp}

We summarise our results on the analysis of the Flory equation (\ref{eq: E+G}) as follows.
\begin{thm}\label{thrm: Smoluchowski equation}
Let $S$ be a $\pi_0$-bilinear coagulation system, and let $\mu_0$ be a sub-probability measure on $S$ satisfying Assumption \ref{hyp: A}. Then the equation (\ref{eq: E+G}) has a unique solution $(\mu_t)_{t\geq 0}$ starting at $\mu_0$; we write $g_t=(M_t, E_t, P_t)$ for the gel data defined in (\ref{eq: gel data}). This solution has the following properties.
\paragraph{\textbf{1. Phase Transition.}} Let $t_\mathrm{g}$ be the first time at which the solution $\mu_t$ fails to be conservative, that is:
\begin{equation} \begin{split} t_\mathrm{g}&:=\inf\{t\ge 0: \langle \pi_i, \mu_t\rangle \neq \langle \pi_i,\mu_0\rangle \text{ for some }0\le i\le n+m\}=\inf\{t\ge 0: \langle \varphi, \mu_t\rangle < \langle \varphi, \mu_0\rangle \}.
\end{split} \end{equation}
Then
$t_\mathrm{g}\in (0,\infty)$, and can be given explicitly in terms of the moments of $\mu_0$ as \begin{equation}\label{eq: closed form for tg}
       t_\mathrm{g}= \mathfrak{r}(\Lambda(\mu_0))^{-1}; \hspace{1cm} \Lambda(\mu_0)_{ij}=\langle (A\pi)_i\pi_j, \mu_0\rangle, \hspace{1cm}1\le i, j\le n
   \end{equation}
where $\mathfrak{r}(\cdot)$ denotes the spectral radius of a matrix.
\paragraph{\textbf{2. Behaviour of the Second Moment.}} Consider the second moments
\begin{equation} \mathcal{Q}(t)=(\langle \pi_i\pi_j, \mu_t\rangle)_{i,j=0}^n;\hspace{1cm}\mathcal{E}(t)=\langle \varphi^2, \mu_t\rangle.
\end{equation}
Then \begin{enumerate}[label=\roman{*}).]
    \item $\mathcal{Q}(t)$ is finite and continuous, and so locally bounded, on $[0, \infty)\setminus\{t_\mathrm{g}\}.$ 

    \item On $[0, t_\mathrm{g})$, each moment $\mathcal{Q}_{ij}$ is monotonically increasing, as is $\mathcal{E}$.
    
    \item At the gelation time, $\mathcal{E}(t_\mathrm{g})=\infty$, and $\mathcal{E}(t)\rightarrow \infty$ as $t\rightarrow t_\mathrm{g}.$ 
\end{enumerate}

\paragraph{\textbf{3. Representation of Gel Data.}} For each $t\ge 0$, there exists a unique maximal $n$-tuple $c_t=(c^i_t)_{i=1}^n \ge 0$ such that, for all $x\in S$, \begin{equation}\label{eq: NLFP 1} \sum_{i=1}^n c^i_t \pi_i(x)=2t \int_{S} \left(1-\exp\left(-\sum_{i=1}^n c^i_t \pi_i(y)\right)\right)\overline{K}(x,y)\mu_0(dy). \end{equation} $c_t$ undergoes a phase transition at time $t_\mathrm{g}$: if $t\le t_\mathrm{g}$, then $c_t=0$, and if $t>t_\mathrm{g}$ then at least one component of $c_t$ is strictly positive. Moreover, the map $t\mapsto c_t$ is continuous. \medskip \\  The gel data are given in terms of $c_t$ by \begin{equation}\label{eq: formula for M, E_0}
    g^i_t = \int_{S}\pi_i(x)\left(1-\exp\left(-\sum_{j=1}^n c^j_t\pi_j(x)\right)\right)  \mu_0(dx).
\end{equation} Therefore, if $t>t_\mathrm{g}$ then $M_t>0$, and $E_t>0$ componentwise. Moreover, the map $t\mapsto g_t$ is continuous, and $g_{t_\mathrm{g}}=0$.
\paragraph{\textbf{4. Gel Dynamics.}} The map $t\mapsto g_t$ is differentiable on $t\in(t_\mathrm{g}, \infty)$, and
\begin{equation}
    \frac{d}{dt}g^i_t=\sum_{j,k=1}^n \langle  \pi_i\pi_j, \mu_t\rangle a_{jk}g^k_t.
\end{equation}
\paragraph{\textbf{5. Order of the Phase Transition, and the Size-Biasing Effect.}} The map $t\mapsto c_t$ is right-differentiable at $t_\mathrm{g}$, and as a consequence, the phase transition is first order; that is, the right-derivatives of the gel data $g^i_t, i=0,1,...,n$ exist and are strictly positive at $t_\mathrm{g}$. Moreover, there exist $\theta_i \ge 0, i=1,..,n$, such that $\sum_i \theta_i=1$ and such that \begin{equation} \label{eq: size bias} \sum_{i=1}^n \theta_i (g'_{t_\mathrm{g}+})_i \ge \left( \frac{\sum_{i=1}^n \langle \theta_i \pi_i, \mu_0\rangle }{\langle \pi_0, \mu_0\rangle}\right) (g'_{t_\mathrm{g}+})_0. \end{equation} We call this a \emph{size-biasing} effect: the average of the linear combination $\sum_{i}\theta_i\pi_i$ over particles in the early gel is at least the average over all particles. Let us define also the total interaction rate, which will quantify the inhomogeneity of the initial data $\mu_0$: \begin{equation} s(x)=\int_S \overline{K}(x,y)\mu_0(dy).\end{equation} If $s$ is not constant $\mu_0$-almost everywhere, then $\theta_i$ can be chosen so that the inequality in (\ref{eq: size bias}) is strict.

 \end{thm}
We also prove the following theorem, which is a law of large numbers result for the coagulating particle system $(x^N_j(t): j\le l^N(t))$ . Firstly, following ideas of \cite{N00}, we show that the empirical measure $\mu^N_t$ converges to the limiting solution $(\mu_t)_{t\ge 0}$ in the vague topology, uniformly in time. The second part of the result is that the \emph{stochastic gel} $g^N_t=N^{-1}\pi(x^N_1(t))$ itself satisfies a law of large numbers, converging to the true gel $g_t$ as $N \rightarrow \infty$, where we order the particles so that $x^N_1$ is the largest particle by $\pi_0$.

We make the following hypotheses for the law of large numbers. These are naturally satisfied when, for example, the initial particles $(x^N_i(0): 1\le i\le l^N(0))$ are sampled as a Poisson random measure with intensity $N\mu_0$. However, it is useful for some intermediate results to give these results in the more general form used here.  \begin{hyp}\label{hyp: B} Let $\mu^N_0$ be the initial data  the stochastic coagulant, and let $\mu_0$ be the initial data of the limiting Flory equation. \begin{enumerate}[label=(B\arabic*.)]\item  As $N\rightarrow \infty$, the initial measures $\mu^N_0=\frac1N\sum_{i\le l^N(0)}\delta_{x_i}$ converge in probability to $\mu_0$ under the vague topology, that is: \begin{equation}d(\mu^N_0, \mu_0)\rightarrow 0 \qquad\text{in probability}. \end{equation} Moreover, $\mu^N_0$ is supported on the set $\{\pi_0=1\}.$
\item We also have the convergence \begin{equation} \langle \pi_i, \mu^N_0\rangle \rightarrow \langle \pi_i, \mu_0\rangle \hspace{1cm} \text{in probability}\end{equation} for all $0\le i\le n+m$, and the uniform integrability \begin{equation} \sup_{N\ge 1} \EE\langle \varphi^2, \mu^N_0\rangle <\infty;\hspace{1cm} \sup_{N\ge 1} \EE\left[\langle \varphi^2, \mu^N_0\rangle 1\left(\langle \varphi^2, \mu^N_0\rangle \ge M\right)\right] \rightarrow 0 \text{ as }M\rightarrow \infty. \end{equation}\end{enumerate} \end{hyp}
\begin{thm} \label{thrm: convergence of stochastic coagulent} Let $\mu_0$ be a sub-probability measure on $S$ satisfying Assumption \ref{hyp: A}, and let $(\mu_t)_{t\ge 0}, (g_t)_{t\ge 0}$ be the associated solution to (\ref{eq: E+G}) and corresponding gel. For $N\ge 1$, let $\mu^N_t$ be the stochastic coagulant with initial data satisfying Assumption \ref{hyp: B}, and write $(x^N_j(t): j\le l^N(t))$ for the particles of the stochastic system, sorted in decreasing order of $\pi_0$. Let $g^N_t=N^{-1}(\pi_i(x^N_1(t)))_{i=0}^n$ be the data of the largest particle in the stochastic system, normalised by $N^{-1}$. Then we have the convergence \begin{equation} \label{eq: convergence of stochastic system}
    \sup_{t\ge 0}\hspace{0.1cm}\left(d\left(\mu^N_t,\mu_t\right)+\left|g^N_t-g_t\right|\right)\rightarrow 0
\end{equation} in probability. In particular, we have the following phase transition: \begin{enumerate}[label=\roman{*}).]
    \item If $t\le t_\mathrm{g}$, then the largest particle has gel data of the order $o_\mathrm{p}(N)$;
    \item If $t>t_\mathrm{g}$, the largest particle has gel data of the order $\Theta_\mathrm{p}(N)$.
\end{enumerate} 

Moreover, if $\xi_N$ is any sequence with $\xi_N\rightarrow \infty$ and $\frac{\xi_N}{N}\rightarrow 0$, then we may define $\widetilde{g}^N_t$ by summing the data of all particles $x^N_j(t)$ with $\pi_0(x^N_j(t)) \ge \xi_N$, and normalising by $N$. Then the same result holds when we replace $g^N_t$ by $\widetilde{g}^N_t$ in (\ref{eq: convergence of stochastic system}).
\end{thm} 

Here, and throughout, we use the notation $o_\mathrm{p}(\cdot), \mathcal{O}_\mathrm{p}(\cdot), \Theta_\mathrm{p}(\cdot)$ for the probabilistic equivalents of $o(\cdot), \mathcal{O}(\cdot), \Theta(\cdot)$, and say that an event\footnote{or, more formally, a sequence of events indexed by $N$} holds \emph{with high probability} if relevant probabilities converge to $1$ as $N\rightarrow \infty$. Precise definitions can be found in \cite{JLR}.

\subsection{\textbf{Plan of the Paper.}} Our programme will be as follows. \begin{enumerate} \item In the remainder of this section, we will discuss other works on coagulating particle systems in the literature, and how they relate to our results. \item In Section \ref{sec:SE}, we will prove that the limiting equation (\ref{eq: E+G}) has unique, globally defined solutions, based on a truncation argument from \cite{N99,N00}.
\item In Section \ref{sec: csc}, we prove an initial result, Lemma \ref{lemma: local uniform convergence of stochastic coagulent}, on the convergence of the stochastic coagulant, using the ideas of \cite[Theorem 4.1]{N00}. This will later be used to prove later points of Theorem \ref{thrm: Smoluchowski equation} based on probabilistic arguments for the empirical measures $\mu^N_t$, and the random graphs $G^N_t$ introduced in Section \ref{sec: coupling_to_random_graph}.
\item In Section \ref{sec: coupling_to_random_graph}, we show how the stochastic coagulant can be coupled to a family of inhomogenous random graphs defined in \cite{BJR07}. Key results for these graphs are recalled in Appendix \ref{sec: IRG}. The critical time $t_\mathrm{c}$ for these graphs may be found exactly, leading to the explicit expression in Theorem \ref{thrm: Smoluchowski equation}. \item A weakness of the preceding sections is that, a priori, the critical time $t_\mathrm{c}$ for the graph processes may differ from the gelation time $t_\mathrm{g}$; in Section \ref{sec: ECT}, we show that this cannot happen. This is based on a preliminary version of Theorem \ref{thrm: convergence of stochastic coagulent}, which shows convergence of $(\mu^N_t, g^N_t)$ at a single fixed time $t\ge 0$. 
\item Section \ref{sec: finiteness of second moment} is dedicated to a proof of item 2 of Theorem \ref{thrm: Smoluchowski equation}, concerning the second moments $\mathcal{Q}_{ij}(t)=\langle \pi_i\pi_j, \mu_t\rangle, \mathcal{E}(t)=\langle \varphi^2, \mu_t\rangle$. The statements about the subcritical and critical cases $t<t_\mathrm{g}, t=t_\mathrm{g}$ follow general ideas in \cite{N99,N00}, while the statement about the supercritical case $t>t_\mathrm{g}$ uses additional ideas from the theory of random graphs.
\item Section \ref{sec: gel dynamics} uses the ideas of previous sections to prove items 3 and 4 of Theorem \ref{thrm: Smoluchowski equation}, concerning the gel data $g_t$ beyond the critical point. \item Section \ref{sec: uniform convergence} uses the analysis of the gel to extend Lemma \ref{lemma: local uniform convergence of stochastic coagulent} to show that convergence is uniform in time.
\item Section \ref{sec: BNCP} proves item 5 of Theorem \ref{thrm: Smoluchowski equation}, concerning the behaviour near the critical point. This completes the proof of this theorem. \item To finish the proof of Theorem \ref{thrm: convergence of stochastic coagulent}, we revisit the ideas of Section \ref{sec: ECT} to prove convergence of the stochastic gel $g^N_t, \widetilde{g}^N_t$, uniformly in time. This is the focus of Section \ref{sec: COG}, and builds further on ideas of previous sections. \end{enumerate}

\subsection{Literature Review}
The original equation introduced by Smoluchowski considers the case of coagulating particles, whose only property is a mass belonging to $\NN$, and where the coaguation $x, y\mapsto x+y$ has a general rate $\overline{K}(x,y)$. In this case, identifying measures $\mu \in \mathcal{M}_{\le 1}(\NN)$ with a summable sequence, the equation analagous to (\ref{eq: E}) reads 
\begin{equation}\label{e:introsmol}
  \frac{\dd}{\dd t}\mu_t(x) = \frac12 \sum_{y < x}\overline{K}(y,x-y)\mu_t(y)\mu_t(x-y) \dd y - \mu_t(x)\sum_{y=1}^\infty \overline{K}(x,y) \mu_t(y) \dd y,
\end{equation}
For an extensive review the reader is referred to \cite{A99}.
The case $\overline{K}(x,y):= xy$ is known as the \emph{multiplicative} coagulation kernel and in this case with $\mu_0=\delta_1$, the solution of (\ref{e:introsmol}) exhibits gelation at $t_\mathrm{g}=1$.

The existence and value of the gelation time has been studied for a range of $\overline{K}$.
For particles with integer masses and $\epsilon (xy)^\alpha \leq \overline{K}(x,y) \leq M xy,\ M\in\mathbb{R}_+, \alpha\in(\frac12, 1)$ Jeon \cite{J98} proved the existence of a gelation phase transition and provided an upper bound on the gelation time.
\medskip \\ Norris \cite{N99,N00} introduced a more general form, analagous to (\ref{eq: E}) on a general space $S$, allowing particles with internal structure and where, for any pair of particles, there are multiple possible coagulation products, in the case $\overline{K}(x,y) \leq C \varphi(x) \varphi(y)$ for some function $\varphi$ growing no more than linearly in particle mass, a step that is important for the present work.
A lower bound for the gelation time was proved in \cite{N99} and an upper bound was added under appropriate assumptions in \cite{N00}; however, these bounds do not coincide in general. Normand \cite{Nm09} obtained explicit results concerning the blowup of a second moment for a sexed model which gives a lower bound on the gelation time, and in a later work \cite{Nm11} finds explicit expressions for the gelation time for a selection of models with arms.  Let us also mention the more recent work of Merle and Normand, who show results similar to Theorems \ref{thrm: Smoluchowski equation}, \ref{thrm: convergence of stochastic coagulent} for the multiplicative coagulant, but where particles become inert when they reach a scale $\alpha(N)\ll N$; particles above this size play the role of the gel. Consequently, ours is one of the first models for which the gelation time can be found exactly; moreover, several aspects of our analysis extend what was previously known about the Smoluchowski equation, using the connection to random graphs \cite{BJR07}.

The study of gelation as the formation of a very large connected structure by joining basic building blocks goes back at least to Flory \cite{Flo41} whose motivation was hydrocarbon polymerisation in the manufacture of plastics.
Flory understood polymerisation as the formation of a random graph, rather than in terms of coagulation, and was aware of a sharp phase transition at the emergence of a giant connected structure, which he termed `gel'.
A rigorous proof of the random graph phase transition was provided by Erd\H{o}s and R\'enyi \cite{ER60}.
The existence of a phase transition corresponding to the formation of a giant particle, which corresponds to the phase transition in Theorem \ref{thrm: convergence of stochastic coagulent}, was first discussed by Lushnikov \cite{L78}, who uses this to explain the explosion of the second moment, corresponding to item 2 of Theorem \ref{thrm: Smoluchowski equation}, in the particular case of the multiplicative kernel.
The first connection between random graph and particle approaches appears in \cite{BP91}, where the phase transition is proved for the particle coagulation process and an interpretation as a new proof for a phase transition in the Erd\H{o}s-R\'enyi random graph is noted; this is also discussed in the survey article \cite{A99}.
We extend this connection, and show that the bilinear form of the merger rate allows us to couple the stochastic coagulant process to \emph{inhomogeneous} random graphs as considered by \cite{BJR07}.

Our original motivation was to study a concept of interaction clusters introduced by Gabrielov et al.~\cite{GKSZ08} in the context of the billiard model for an ideal gas.
The distribution of the sizes of the interaction clusters is formally derived in \cite{PSW16} in terms of the solution of the Boltzmann equation.
Reducing to the case of cutoff Maxwell molecules for the spatially homogeneous Boltzmann equation, the phase transition observed in \cite{GKSZ08} can be identified precisely and the cluster size distributions observed to match those arising from the Smoluchowski coagulation equation with product kernel \cite{L78,A99,PSW16}. Heuristically, when a collision occurs, the corresponding clusters merge, which may be represented as a coagulation event at the level of interaction clusters. In \cite{PSW17} the clusters were studied for the Kac process, which is a stochastic approximation to the billiard model with elastic collisions, and the restriction to Maxwell molecules was lifted.
This allowed a general collision rate including the hard sphere case and it was formally shown in a large particle number limit that the distribution of the cluster sizes converges to a version of the Smoluchowski coagulation equation with a time-dependent product kernel.
In the Kac model where the rate of collision between two molecules with velocities $v,w$ is proportional to $\abs{v-w}^2=|v|^2+|w|^2-2v\cdot w$, a sum over particles in a cluster shows that the total merge rate depends on the mass, momentum and energy of the two clusters. Moreover, since collisions are elastic, these quantities add when two clusters merge, and are unchanged when a cluster undergoes an internal collision.
This quadratic collision rate is of significant interest \cite{Lu,PSW17,Villani}, although it does not have a natural physical interpretation.
The explicit representation of the critical times in the present work enable us to verify the conjecture that the phase transition occurs strictly before the mean free time \cite{PSW17}.
\color{black}
\section{\textbf{Well-Posedness of the Limiting Equation}}\label{sec:SE}  This chapter is dedicated to a first analysis of the Smoluchowski equations (\ref{eq: E}, \ref{eq: E+G}), following Norris \cite{N99,N00}. Our goal in this section is to prove the following lemma on the well-posedness of (\ref{eq: E+G}).

\begin{lem}\label{lemma: E and U} For any measure $\mu_0 \in \mathcal{M}$ satisfying (A1.), the equation with gel (\ref{eq: E+G}) has a unique global solution $(\mu_t)_{t\geq 0}$ starting at $\mu_0$. Moreover, $P_t=0$ for all $t\ge 0$. \end{lem}
\begin{cor}\label{cor: maximal conservative solutions} Suppose $(\mu'_t)_{t<T}$ is a conservative local solution to the equation without gel, (\ref{eq: E}), starting at $\mu_0$. Then $\mu_t=\mu'_t$ for all $t<T$, and $T<t_\mathrm{g}$. Hence, (\ref{eq: E}) has a unique maximal conservative solution, given by $(\mu_t)_{t<t_\mathrm{g}}$.
\end{cor}

Our proof of Lemma~\ref{lemma: E and U} is an adaptation of the arguments in \cite[Section 2]{N99} and \cite[Section 2]{N00} and is based on a truncation argument. Recalling that $\varphi=\sum_{i=0}^n \pi_i$, we see that  $\overline{K}(x,y)\leq \Delta\varphi(x)\varphi(y)$ for some $\Delta=\Delta(A).$ For all $\xi>0$, we define the truncated particle space \begin{equation} \label{d:trunc_space} S_\xi=\{x\in S: \varphi(x) \le \xi\}. \end{equation}

We consider the following `truncation at level $\xi$': in the empirical measure, we track only those particles inside $S_\xi$, and consider all other particles to belong to a `truncated gel'. Although the particles in the truncated gel affect the dynamics in $S_\xi$, these contributions depend only on the total data $g^\xi$ of the truncated gel, due to the bilinear form of the kernel. This leads to an ordinary differential equation with Lipschitz coefficients in an infinite dimensional space. \medskip \\ We formalise this intuition as follows.  For a measure $\mu^\xi$ supported on $S_\xi$ and $g^\xi\in S_\mathrm{g}$, we define a signed measure $L^\xi_\mathrm{g}(\mu^\xi, g^\xi)$ on $S_\xi$ by specifying, for all $f\in C_c(S)$, \begin{equation}\begin{split}\label{eq: truncated drift 1} &\left\langle f, L^\xi_\mathrm{g}(\mu^\xi, g^\xi)\right\rangle \\[1ex] &\hspace{1cm}=\frac{1}{2}\int_{S_\xi^2}[f(x+y)1[\varphi(x+y)\le \xi]-f(x)-f(y)]\overline{K}(x,y)\mu^\xi(dx)\mu^\xi(dy)\\&\hspace{2.5cm}+\int_{S_\xi}(f(y)-f(x))J(x,dy)\mu^\xi(dx) -\int_{S_\xi}f(x)\overline{K}(x,g^\xi)\mu^\xi(dx). \end{split}\end{equation} This corresponds to the dynamics of particles inside $S_\xi$. The rate of change of the truncated gel data is given by \begin{equation} \begin{split}\label{eq: truncated drift 2} R^\xi_\mathrm{g}(\mu^\xi,g^\xi)&=\frac{1}{2}\int_{S_\xi^2} \pi(x+y)1[\varphi(x+y)> \xi]\overline{K}(x,y)\mu^\xi(dx)\mu^\xi(dy) \\ &\hspace{1cm}+\int_{S_\xi}\pi(x)\overline{K}(x,g^\xi)\mu^\xi(dx).\end{split}\end{equation}We now seek measures $\mu^\xi_t$ supported on $S_\xi$ and gel data $g^\xi_t=(M^\xi_t, E^\xi_t,P^\xi_t)\in S_\mathrm{g}$ such that, for all bounded measurable $f$ on $S_\xi$, \begin{equation} \tag{Fl$|^1_\mathrm{\xi}$} \label{eq:rE1}  \langle f, \mu^\xi_t\rangle = \langle f, \mu^\xi_0\rangle +\int_0^t\left\langle f, L^\xi_\mathrm{g}(\mu^\xi_s, g^\xi_s)\right\rangle ds;\end{equation}
\begin{equation} \label{eq: rE2} \tag{Fl$|^2_\mathrm{\xi}$}
g^\xi_t=g^\xi_0+\int_0^tR^\xi_\mathrm{g}(\mu^{\xi}_s,g^\xi_s)ds.
\end{equation} 
We will use the following existence and uniqueness result for the restricted dynamics (\ref{eq:rE1}, \ref{eq: rE2}).
\begin{lem}\label{lemma: E and U of Restricted}[Existence and Uniqueness of Restricted Dynamics]\label{lemma: restricted dynamics} Suppose $\mu^\xi_0$ is a finite measure on $S_\xi$ which satisfies ({A1}.), and $g^\xi_0 \in S_\mathrm{g}$ satisfies $\pi_i(g^\xi_0)=0$ for all $i>n$. Then there exists a unique map $(\mu^\xi_t, g^\xi_t)$ on $[0, \infty)$, which solves the restricted dynamics (\ref{eq:rE1}, \ref{eq: rE2}). Moreover, for all $t\ge 0$, $\mu^\xi_t$ is a positive, finite measure on $S_\xi$, $P^\xi_t=0$ for all times $t\ge 0$, and $g^\xi_t \in S_\mathrm{g}$. 
\end{lem}
\begin{proof}[Sketch Proof of Lemma \ref{lemma: restricted dynamics}] This may be proved by a trivial modification of the arguments in \cite[Proposition 2.2]{N99}. We define Picard iterates $((\mu^{(\xi,n)}_t, g^{(\xi,n)}_t): n\ge 0, t\ge 0)$ by \begin{align} (\mu^{(\xi,0)}_t, g^{(\xi,0)}_t)&=(\mu^\xi_0, g^\xi_0);\\ \left(\mu^{(\xi,n+1)}_t, g^{(\xi,n+1)}_t\right)&=(\mu^\xi_0,g^\xi_0)+\int_0^t (L^\xi_\mathrm{g}, R^\xi_\mathrm{g})\left(\mu^{(n,\xi)}_s, g^{(n,\xi)}_s\right) ds. \end{align} One then uses bilinear continuity arguments in total variation norm $\|\cdot\|$ to show that, given a bound $\langle \varphi, \mu^\xi_0\rangle +\varphi(g^\xi_0)\le C$, there is a positive time $T=T(\xi,C)>0$ such that the Picard iterates $(\mu^{(\xi,n)}_t)_{t\le T}$ converge uniformly in total variation on $[0,T]$, and that the limit $\mu^\xi_t$ solves (\ref{eq:rE1}, \ref{eq: rE2}), possibly allowing $\mu^\xi_t$ to be a signed measure. This argument also implies that the solution is unique on this interval. Now, we note that the quantity $\langle \varphi, \mu^\xi_t\rangle +\varphi(g^\xi_t)$ is constant in time, and therefore this construction can be repeated on $[T, 2T]$, $[2T, 3T]$, etc, which proves global existence and uniqueness. Finally, an integrating factor is introduced to argue that $\mu_t$ is a positive measure.
In our case, it is also straightforward to see that the gel data $M^\xi_t, E^\xi_t\ge 0$, and that $P^\xi_t=0$, thanks to the symmetry ({A1}.).
\end{proof}

\begin{proof}[Proof of Lemma \ref{lemma: E and U}] 
We first show existence. For all $\xi<\infty$, we let $(\mu^\xi_t, g^\xi_t)$ be the solution to the dynamics (\ref{eq:rE1}, \ref{eq: rE2}) restricted to $S_\xi$, with initial data \begin{equation} \mu^\xi_0(dx)=1[x\in S_\xi]\hspace{0.1cm}\mu_0(dx);\hspace{1cm} g^\xi_0 = \int_{x\not\in S_\xi} \pi(x)\mu_0(dx). \end{equation}  Observe that, if $\xi<\xi'$, then $\widetilde{\mu}^\xi_t, \widetilde{g}^\xi_t$ given by \begin{equation} \widetilde{\mu}^\xi_t(dx)=1_{x\in S_\xi}\hspace{0.1cm}\mu^{\xi'}_t(dx);\hspace{1cm} \widetilde{g}^\xi_t=g^{\xi'}_t+\int_{x\in S_{\xi'}\setminus S_\xi} \pi(x) \mu^{\xi'}_t(dx)\end{equation} solve the dymanics (\ref{eq:rE1},\ref{eq: rE2}) with the same initial data $\mu^\xi_0, g^\xi_0$. From uniqueness in Lemma \ref{lemma: restricted dynamics}, it follows that $\widetilde{\mu}^\xi_t=\mu^\xi_t; \widetilde{g}^\xi_t=g^\xi_t$. This shows that the measures $\mu^\xi_t$ are increasing in $\xi$, while the gel data $M^\xi_t, E^\xi_t$ are decreasing, and $P^\xi_t$ is identically $0$, by symmetry ({A1}.). Therefore, the limits \begin{equation} \mu_t=\lim_{\xi\uparrow \infty} \mu^\xi_t; \hspace{1cm} M_t=\lim_{\xi\rightarrow \infty} M^\xi_t; \hspace{1cm} E_t=\lim_{\xi\rightarrow \infty} E^\xi_t \end{equation} exist in the sense of monotone limits; one can then check that $\mu_t$ and $g_t=(M_t,E_t,0)$ satisfy the full equation (\ref{eq: E+G}), with initial values $\mu_0$ and $g_0=0.$

To see uniqueness, let $\mu_t$ be the solution constructed above and write $g_t=(M_t,E_t, P_t)$ for the data of the gel. Let $\widetilde{\mu}_t$ be any solution to (\ref{eq: E+G}) starting at $\mu_0$, and let $\widetilde{g}_t=(\tilde{M}_t,\tilde{E}_t,\tilde{P}_t)$ be the associated data of the gel. For all $\xi<\infty$, it is simple to verify that \begin{equation} \widetilde{\mu}^\xi_t(dx)=1_{x\in S_\xi}\hspace{0.1cm} \widetilde{\mu}_t(dx);\hspace{1cm} \widetilde{g}^\xi_t=\widetilde{g}_t+\int_{S^\mathrm{c}_\xi} \pi(x) \widetilde{\mu}_t(dx) \end{equation} is a solution to the dynamics (\ref{eq:rE1}, \ref{eq: rE2}) on $S_\xi$. By uniqueness in Lemma \ref{lemma: restricted dynamics}, it follows that $\widetilde{\mu}^\xi_t=\mu^\xi_t$, and taking monotone limits, we see that $\widetilde{\mu}_t=\lim_{\xi\rightarrow \infty} \widetilde{\mu}^\xi_t=\lim_{\xi\rightarrow \infty} {\mu}^\xi_t=\mu_t$. The argument for $\widetilde{g}$ is identical.
\end{proof}

\section{\textbf{Convergence of the Stochastic Coagulant}}
\label{sec: csc}

We now turn to a preliminary version of Theorem \ref{thrm: convergence of stochastic coagulent}.
In this section, we will outline the proof of the convergence of the stochastic coagulant $\mu^N_t$ to a solution $\mu_t$ of (\ref{eq: E+G}), locally uniformly in time. Most of the arguments are well-known for the Smoluchowski equation \cite{N99,N00}, and for brevity, we will sketch the proof with an indication of the nontrivial technical details. Throughout, we fix $\mu_0$ satisfying Assumption \ref{hyp: A}, and $\mu^N_t$ with initial data $\mu^N_0$ satisfying Assumption \ref{hyp: B}.  Our result is as follows.
\begin{lem}\label{lemma: local uniform convergence of stochastic coagulent}
Suppose $\mu_0$ satisfies Assumption \ref{hyp: A}, and let $(\mu_t)_{t\ge 0}$ be the solution to (\ref{eq: E+G}) starting at $\mu_0$. Let $(\mu^N_t)_{t\ge 0}$ be stochastic coalescents with initial data $\mu^N_0$ satisfying Assumption \ref{hyp: B}. Then we have the local uniform convergence
\begin{equation}
\forall t_\mathrm{f}\ge 0 \hspace{0.5cm} \sup_{t\le t_\mathrm{f}} \hspace{0.1cm} d(\mu^N_t, \mu_t)\rightarrow 0 \text{ in probability} 
\end{equation} where recall that $d$ is a complete metric inducing the vague topology.
\end{lem} \begin{rmk}
We will later upgrade the \emph{local} uniform convergence to \emph{full} uniform convergence in Lemma \ref{lemma: uniform convergence of coagulant}. We also remark that this does not immediately imply the convergence of the gel terms in Theorem \ref{thrm: convergence of stochastic coagulent}, as the test functions involved are  neither compactly supported nor even bounded. This will be dealt with in Sections \ref{sec: ECT}, \ref{sec: COG}, where the proofs build on this result.
\end{rmk}
\begin{proof} The proof follows the well known method of proving tightness and identifying possible limit paths: 
Firstly, the jump rates can bounded, uniformly in time, in terms of the initial second moment $\langle \varphi^2, \mu^N_0\rangle$ and, thanks to (B2.), these are stochastically bounded: $\langle \varphi^2, \mu^N_0\rangle \in \mathcal{O}_\mathrm{p}(1)$ as $N\rightarrow \infty$. As a result, it follows that for all  $t_\mathrm{f}\ge 0$, the processes $(\mu^N_t)_{0\le t\le t_\mathrm{f}}$ are tight in the Skorohod topology of $\mathbb{D}([0,t_\mathrm{f}],(\mathcal{M},d))$.  \medskip \\ Next, we wish to argue that if $\overline{\mu}$ is any subsequential limit point, then $\overline{\mu}$ coincides with the solution $\mu_t$ to (\ref{eq: E+G}). For this stage, we show that for certain well-chosen $\xi>0$, the pair \begin{equation} \mu^{N,\xi}_t=\mu^N_t1_{S_\xi}; \hspace{1cm} g^{N,\xi}_t=\langle \pi, \mu^N_t-\mu^{N,\xi}_t\rangle  \end{equation} converge to a pair $\overline{\mu}^\xi_t=\overline{\mu}_t1_{S_\xi}, \overline{g}^\xi_t$ which solve the restricted evolution equations (\ref{eq:rE1},  \ref{eq: rE2}), started at \begin{equation} \overline{\mu}^\xi_0=\mu_01_{S_\xi}; \hspace{1cm} \overline{g}^\xi_0=\int_{x\not \in S_\xi} x\mu_0(dx). \end{equation} In order to prove this convergence, we will need a pair of regularity conditions (C1-C2.) which will be displayed below. These allow us to obtain vague convergence of $\mu^{N,\xi}_t$, despite the discontinuity of the cutoff $S_\xi$.  Moreover, one can show that these conditions are satisfied for almost all $\xi>0$.  \begin{enumerate}[label=\roman{*}).]\item Almost surely, for  almost all $t\le t_\mathrm{f}$, \begin{equation} \label{eq: reg} \tag{{C1.}}
   \overline{\mu}_t\left(\left\{x\colon \varphi(x) = \xi\right\}\right) + 
    \overline{\mu}_t \otimes \overline{\mu}_t \left(\left\{(x,y) \colon \varphi(x+y) = \xi\right\}\right) = 0;
\end{equation} \item This also holds for $t=0$. That is, almost surely, \begin{equation} \label{eq: reg2} \tag{{C2.}}
   \overline{\mu}_0\left(\left\{x\colon \varphi(x) = \xi\right\}\right) + 
    \overline{\mu}_0 \otimes \overline{\mu}_0 \left(\left\{(x,y) \colon \varphi(x+y) = \xi\right\}\right) = 0. 
\end{equation} \end{enumerate} Thanks to the construction of solutions to the global equation (\ref{eq: E+G}) in Lemma \ref{lemma: E and U}, we know that for all such $\xi$, $\overline{\mu}_t1_{S_\xi}$ coincides with $\mu_t1_{S_\xi}$. Finally, we take a limit of such $\xi \uparrow \infty$, to conclude the equality $\overline{\mu}_t=\mu_t, t\le t_\mathrm{f}$.  Since the limit process $(\mu_t)_{0\le t\le t_\mathrm{f}}$ is continuous in the vague topology $(\mathcal{M},d)$, it follows that we may upgrade from Skorohod to uniform convergence: \begin{equation} \sup_{0\le t\le t_\mathrm{f}}\hspace{0.1cm}d\left(\mu^N_t, \mu_t\right) \rightarrow 0 \hspace{0.5cm} \text{in probability}  \end{equation} as claimed.  
\end{proof}

\section{\textbf{Coupling of the Stochastic Coagulant to Random Graphs}} \label{sec: coupling_to_random_graph}
In this section, we will show that the stochastic coagulant defined in the introduction may be coupled to a \emph{dynamic} version of the random graphs $\mathcal{G}^\mathcal{V}(N,tk)$ considered in \cite{BJR07}. This allows us to apply some results of that paper, which we summarise in Appendix \ref{sec: IRG}, to analyse the stochastic coagulant process and the limit equation.
\begin{defn}\label{def: GNT}[Dynamic Inhomogenous Random Graphs] Fix a measure $\mu_0$ satisfying Assumption \ref{hyp: A}. Let $\mathbf{x}_N=(x_i, i=1,2,...,l^N)$ be a collection random points in $S$ of potentially random length $l^N$, and sample $\tau_e \sim \text{Exponential}(1)$, independently of each other, for $e=(ij), 1\le i,j\le l^N$, and independently of $\mathbf{x}_N$. We define the kernel \begin{equation} k(v,w)=2\overline{K}\left(x,y\right) \end{equation} where the right-hand side is the total mass of the interaction kernel $\overline{K}(x,y)=K(x,y,S)$. We form the random graphs $(G^N_t)_{t \ge 0}$ on $\{1,2,...,l^N\}$ by including the edge $e=(ij)$ if \begin{equation}
    t\ge \frac{N \tau_e}{k(x_i,x_j)}.
\end{equation} We write $G^N_t \sim \mathcal{G}(\mathbf{x}_N, \frac{tK}{N})$ for the distribution of $G^N_t$, for a single fixed $t\ge 0$. We say that $G^N_0$ satisfy Assumption \ref{hyp: B} for $\mu_0$ if the same is true of the empirical measures $\mu^N_0=N^{-1}\sum_{i\le l^N}\delta_{x_i}$. We emphasise that the $x_i$ do \emph{not} change during the dynamics. \end{defn}

This has the following immediate consequences. Firstly, the conditions in Assumption \ref{hyp: B} guarantee that $ \mathcal{V}=(S, \mu_0, (\mathbf{x}_N)_{N\ge 1})$, is a generalised vertex space in the sense of \cite{BJR07}, which is recalled in Definition \ref{def: Generalised vertex space}, and $k$ is an irreducible kernel as described in Definition \ref{def: kernel}, thanks to (A4.). Using both parts of (B2.), one can also show that the kernel $k$ is \emph{graphical} in the sense of Definition \ref{def: graphical kernel}. \medskip \\ 
For all times $t$, $G^N_t$ is an instance of the inhomogeneous random graph from Definition \ref{definition of GN}. Moreover, the process $(G^N_t)_{t\ge 0}$ is increasing, and is a Markov process, by the memoryless property of the exponential variables $\tau_e$.  We write $T$ for the convolution operator \begin{equation} \label{eq: T} (Tf)(x)=\int_{S}f(y)k(x,y)\mu_0(dy) \end{equation} and $\|T\|$ for the associated operator norm in $L^2(\mu_0)$. \medskip \\ We write also $t_\mathrm{c}=\|T\|^{-1}$. The following is the basic statement of a phase transition for $G^N_t$, which follows from Theorem \ref{thrm: RG1}. \begin{lem}\label{lemma: basic phase transition} Let $\mu_0$ satisfy Assumption \ref{hyp: A}, and let $G^N_t$ be the random graphs constructed above, such that $G^N_0$ satify Assumption \ref{hyp: B}. Write $C_1(G^N_t)$ for the size of the largest component of $G^N_t$. Then we have the following phase transition: \begin{enumerate}[label=\roman*).] \item If $t\le t_\mathrm{c}$, then $N^{-1}C_1(G^N_t)\rightarrow 0$ in probability. \item If $t>t_\mathrm{c}$, then there exists $c=c(t)$ such that, with high probability, $C_1(G^N_t)\ge cN.$ \end{enumerate}  \end{lem}

We write $\mathcal{C}_1(G),...\mathcal{C}_j(G),...$ for the connected components, which we also call \emph{clusters}, of $G$, in decreasing order of size, allowing $\mathcal{C}_j=\emptyset$ if $G$ has fewer than $j$ components and $C_j(G)$ for the number of vertices in $\mathcal{C}_j(G)$.
For a cluster $\mathcal{C}$ of the graph $G^N_t$, we will write 
\begin{equation}
  M(\mathcal{C})=\sum_{i\in \mathcal{C}}\pi_0(x_i),\quad E(\mathcal{C})= \left(\sum_{i\in \mathcal{C}} \pi_j(x_i)\right)_{j=1}^{n},\quad  P(\mathcal{C})= \left(\sum_{i\in \mathcal{C}} \pi_j(x_i)\right)_{j=n+1}^{n+m}\end{equation}
for the unnormalised data, and
\begin{equation}\label{eq: cluster quantities}
\pi(\mathcal{C})=\sum_{i\in \mathcal{C}} \pi(x_i)=\left(M(\mathcal{C}),E(\mathcal{C}),P(\mathcal{C})\right),\quad
    \varphi(\mathcal{C})=\sum_{i\in \mathcal{C}} \sum_{j=0}^n \pi_j(x_i). \end{equation}
We write $\delta(\mathcal{C})$ for the point mass $\delta(\mathcal{C})=\delta_{\pi(\mathcal{C})}$, and $\pi_\star(G^N_t)$ for the normalised empirical measure \begin{equation} \pi_\star(G^N_t)=\frac{1}{N}\sum_\text{Clusters} \delta(\mathcal{C})\end{equation} where the sum is over all clusters $\mathcal{C}$ of $G^N_t$. This is connected to the stochastic coagulants as follows:

\begin{lem}[Coupling of Random Graphs and Stochastic Coagulants]\label{lemma: coupling} Fix points $\mathbf{x}_N=(x_1,...,x_{l^N(0)})$ in $S$, and let $(G^N_t)_{t\ge 0}$ be the random graph process described in Definition \ref{def: GNT} for this choice of vertex data. Consider also a stochastic coagulant $(\mu^N_t)_{t\ge 0}$ started from $\mu^N_0=\frac{1}{N}\sum_{i\le l^N(0)} \delta_{x_i}$. Then the processes $\pi_\star(G^N_t)$ and $\pi_\#\mu^N_t$ are equal in law.
\end{lem}

\begin{rmk}\label{rmk: why random graphs}This is the key result which makes much of our analysis possible. Many of the remaining points of Theorem \ref{thrm: Smoluchowski equation} above concern only the moments $\langle \pi_i, \mu_t\rangle$, $\langle \varphi^2, \mu_t\rangle$ which depend on $\mu_t$ only through the pushforward $\pi_\#\mu_t.$ By applying Lemma \ref{lemma: local uniform convergence of stochastic coagulent} in the space $S^\Pi$, we can use the pushforwards $\pi_\#\mu^N_t$ as stochastic proxies to $\pi_\#\mu_t$, and thanks to Lemma~\ref{lemma: coupling}, the measures $\pi_\#\mu^N_t$ can be realised as $\pi_\star(G^N_t)$ for a random graph process $G^N_t$. In this way, we can apply results from the theory of random graphs \cite{BJR07} to deduce results about solutions $(\mu_t)$ to the Smoluchowski equation (\ref{eq: E+G}). \end{rmk}
\begin{proof}[Sketch of proof of Lemma~\ref{lemma: coupling}]
Let us fix $\mathbf{x}_N$. Firstly, both processes are Markov: for $\pi_\#\mu^N_t$, the follows because the total rate (\ref{eq: overline K}) depends only on $\pi(x), \pi(y)$, and similarly for $\pi_\star(G^N_t)$. One may also verify that the two processes undergo the same transitions at the same rates, again thanks to (\ref{eq: overline K}), and that the total rate is bounded in terms of $\mathbf{x}_N$. The boundedness of the total rate implies the uniquness in law for the corresponding Markov generator, which concludes the proof.    \end{proof}

Combining this with the approximation result Lemma \ref{lemma: local uniform convergence of stochastic coagulent} for the stochastic coagulant, we may connect the random graph process to the limit equation as follows.
\begin{lem}[Convergence of the Random Graphs]\label{lemma: convergence of random graphs} Let $\mu_0$ be a measure on $S$ satisfying Assumption \ref{hyp: A}, and let $(G^N_t)_{t\ge 0}$ be the random graph processes constructed above with initial data $\mathbf{x}_N=(x_1,...x_{l^N})$ which satisfies Assumption \ref{hyp: B}. Let $(\mu_t)_{t\ge 0}$ be the solution to the Smoluchowski Equation (\ref{eq: E+G}) starting at $\mu_0$; then we have the local uniform convergence \begin{equation}\sup_{t\le t_\mathrm{f}} \hspace{0.1cm}d_\Pi(\pi_\star(G^N_t), \pi_\#\mu_t) \rightarrow 0 \end{equation} in probability, for all $t_\mathrm{f}<\infty,$  where $d_\Pi$ is a metric for the vague topology $\mathcal{F}(\mathcal{M}_{\le 1}(S^\Pi), C_c(S^\Pi))$. \end{lem} 

We can also compute the critical time associated to $G^N_t$ explicitly:
\begin{lem}[Computation of critical time]\label{lemma: computation of tcrit} Let $\mu_0$ be a measure satisfying Assumption \ref{hyp: A}, and let $G^N_t$ be random graphs satisfying Assumption \ref{hyp: B}. Then the convolution operator $T$ constructed above is a bounded linear map on $L^2(\mu_0)$ and the inverse of the critical time for the graph phase transition, $t_\mathrm{c}^{-1}$, is the largest eigenvalue of the $n\times n$ matrix $\Lambda(\mu_0)$ given by $\Lambda(\mu_0)_{ij}=\langle (A\pi)_i\pi_j, \mu_0\rangle$. In particular, $t_\mathrm{c}\in (0,\infty).$
\end{lem}
\begin{rmk} This is exactly the form claimed for $t_\mathrm{g}$ in Theorem \ref{thrm: Smoluchowski equation}. However, we have not yet established that $t_\mathrm{c}=t_\mathrm{g}$; this is the content of Lemma \ref{lemma: connect critical times}. \end{rmk}
\begin{proof}[Proof of Lemma~\ref{lemma: computation of tcrit}]
Firstly, by (A2.), it is easy to see that $k\in L^2(S\times S, \mu_0\times\mu_0)$, and so, by Lemma \ref{lemma: spectrum of T}, $\|T\|=t_\mathrm{c}^{-1}$ is the largest eigenvalue of $T$; its eigenspace is one-dimensional and consists of functions that are single signed, $\mu_0$- almost everywhere.
Since $0< \|T\|< \infty$ we have $0 < t_\mathrm{c} < \infty$.

In order to reduce from the operator $T$ to the matrix $\Lambda(\mu_0)$ we construct a basis $\{e_i\}_{i\geq 1}$ of $L^2(\mu_0)$ such that
\begin{equation}
       e_i(x)=\pi_i(x), \hspace{0.2cm}i=1,2,..n+m
\end{equation} and, for $i>n+m$, $e_i$ is orthogonal to $E=\text{Span}(e_1,...,e_{n+m})$.
Note that $\pi_0$ plays no special role in the basis, because it does not appear in the rate $\overline{K}$.
We also write $E_+=\text{Span}(e_1,...e_n)$ and $E_\mathrm{Sym}=\text{Span}(e_{n+1},....,e_{n+m})$.
   By expanding the total rate $\overline{K}(x,y)$, we see that, for all $f\in L^2(m),$ \begin{equation} \label{eq: expansion of Tf}
       (Tf)(x)  =2\sum_{i,j=1}^{n+m} a_{ij}\langle f, \pi_i\rangle_{L^2(\mu_0)} \pi_j(x)
   \end{equation} 
 where $\langle\cdot,\cdot\rangle_{L^2(\mu_0)}$ denotes the $L^2(\mu_0)$ inner product.
Therefore, $T$ maps into the subspace $E$, and is 0 on its orthogonal complement.
We further note that the subspaces $E_+, E_\mathrm{sym}$ are orthogonal, and are invariant under $T$.
Therefore, the eigenspace $E^\lambda$ corresponding to $\lambda=t_\mathrm{c}^{-1}$ is a direct sum $E_+^\lambda \oplus E_\mathrm{sym}^\lambda$ of eigenspaces contained within $E_+, E_\mathrm{sym}$. \medskip \\ Since $E^\lambda$ is one-dimensional, one summand must be trivial, and so either $E^\lambda = E^\lambda_+ \subseteq E_+$, or $E^\lambda\subseteq E_\mathrm{sym}$. To exclude the second possibility, we note that any $f\in E_\mathrm{sym}$ satisfies $f(Rx)=-f(x)$ for all $x$ by Definition~\ref{def: BCS}, while eigenfunctions of $T$ are single-signed $\mu_0$-almost everywhere. It therefore follows that $E^\lambda \subseteq E_+$ and that $t_\mathrm{c}^{-1}$ is the largest eigenvalue of $T|_{E_+}$.

The result is now immediate since (\ref{eq: expansion of Tf}) shows that $\Lambda(\mu_0)$ is the matrix representation of $T|_{E_+}$ respect to the basis introduced above.
\end{proof}

We also define $\rho_t$ as the survival function from Lemma \ref{lemma: survival function}, given by the maximal solution to \begin{equation} \label{eq: nonlinear fixed point equation} \rho_t(x)=1-\exp\left(-t(T\rho_t)(x)\right). \end{equation}  We note, for future use, the following properties where $k$ is the kernel given above.
\begin{lem}\label{lemma: form of rho-t}
    The survival function $\rho_t(v)=\rho(tk,x)$ takes the form \begin{equation}
        \rho_t(x)=1-\exp\left(-\sum_{i=1}^n c^i_t\pi_i(x)\right)
    \end{equation} for some $c^i_t \ge 0$. Moreover, the functions $t\mapsto c^i_t$ are continuous.
\end{lem} This proves the first two assertions of item 4 of Theorem \ref{thrm: Smoluchowski equation}.
\begin{proof} Using the symmetry $k(Rx,Ry)=k(x,y)$ and Assumption (A1.), it is simple to verify that $\tilde{\rho}(x):=\rho_t(Rx)$ also satisfies the fixed point equation (\ref{eq: nonlinear fixed point equation}). By maximality of $\rho_t$, we must have $\rho_t(Rx)\le \rho_t(x)$ for all $x\in S$, which implies that $\rho_t$ is even under $R$. \medskip \\ Using the identification of the range of $T$ as in Lemma \ref{lemma: computation of tcrit}, we see that there exist $c^i_t: 1\le i\le n+m$ such that \begin{equation}
    t(T\rho_t)(x)=\sum_{i=1}^{n+m} c^i_t \pi_i(x)
\end{equation}and expanding $k$ as in (\ref{eq: expansion of Tf}), the coefficients are given explicitly by \begin{equation}\label{eq: exp for ct}
c^i_t=2t\sum_{j=1}^n a_{ij}\langle \pi_j\rho_t, \mu_0\rangle.
\end{equation} Since $\rho_t$ is even, we have $c^i_t=0$ for $i>n$, and since $\rho_t\ge 0$, $c^i_t\ge 0$ for $i=1,...,n$. Using (\ref{eq: nonlinear fixed point equation}) again, we obtain the claimed representation  \begin{equation}
   \rho_t(x)=1-\exp\left(-\sum_{i=1}^n c^i_t\pi_i(x)\right).
\end{equation} The continuity follows by applying dominated convergence to (\ref{eq: exp for ct}), and using the continuity of $\rho_t$ established in Theorem \ref{thrm: continuity of rho}. \end{proof} 

\section{\textbf{Equality of the Critical Times}} \label{sec: ECT}

In this section, we will prove that the critical time $t_\mathrm{c}$ for the graph process, introduced in Section \ref{sec: coupling_to_random_graph}, coincides with the gelation time for the limiting equation, defined in Section \ref{sec:SE} as the time at which mass and energy begin to escape to infinity. 
\begin{lem}\label{lemma: connect critical times} Let $\mu_0$ be a measure on $S$ satisfying Assumption \ref{hyp: A}. Let $(\mu_t)_{t\ge 0}$ be the solution to (\ref{eq: E+G}) starting at $\mu_0$, with associated data $M_t, E_t$ of the gel; recall that $t_\mathrm{g}$ is defined by \begin{equation}
    t_\mathrm{g}:=\inf\{t\ge 0: \left<\varphi, \mu_t\right> < \left<\varphi, \mu_0\right>\} 
    = \inf\{t\ge 0: g_t\neq 0\}.
\end{equation} Let $(G^N_t)$ be the random graph processes constructed above, and suppose that Assumption \ref{hyp: B} holds for $G^N_0, \mu_0$. Then the critical time $t_\mathrm{c}$ for the graph transition process coincides with the gelation time $t_\mathrm{g}$.
\end{lem} 
The following is a straightforward corollary. \begin{cor}\label{corr: actual expression for tg} Let $\mu_0$ satisfy Assumption \ref{hyp: A}, and let $(\mu_t)_{t\ge 0}$ be the solution to (\ref{eq: E+G}) starting at $\mu_0$, with gelation at $t_\mathrm{g}$. Then $t_\mathrm{g}$ is given explicitly by (\ref{eq: closed form for tg}). \end{cor} \begin{proof}[Proof of Corollary \ref{corr: actual expression for tg}] Let us form $\mathbf{x}_N$ by sampling points as a Poisson random measure with intensity $N\mu_0$. It is immediate that the resulting data $\mathbf{x}_N$ satisfies Assumption \ref{hyp: B} for the measure $\mu_0$, and the critical time $t_\mathrm{c}$ of the associated graphs $G^N_t$ is given by the claimed expression (\ref{eq: closed form for tg}). From the previous lemma, it now follows that the gelation time $t_\mathrm{g}=t_\mathrm{c}$, which proves the claimed result. \end{proof} 
The proof of Lemma \ref{lemma: connect critical times} is based on the following weak version of the convergence of the gel in Theorem \ref{thrm: convergence of stochastic coagulent}, which will be revisited in Section \ref{sec: COG} to establish uniform convergence.
\begin{lem} \label{lemma: WCOG} Let $(\mu_t)_{t\ge 0}, M_t, E_t$ be as in Lemma~\ref{lemma: connect critical times} and $G^N_t$ be as in the proof of Corollary~\ref{corr: actual expression for tg}. Fix $t>0$, and write $g^N_t$ for the scaled mass and energy of the largest particle in $G^N_t$, as in Section \ref{sec: coupling_to_random_graph}: \begin{equation} g^N_t=\frac{1}{N}\hspace{0.1cm}\pi(\mathcal{C}_1(G^N_t))=\left(\frac{1}{N}\sum_{i \in \mathcal{C}_1(G^N_t)} \pi_j(x_i)\right)_{j=0}^{n+m}=(M^N_t, E^N_t, P^N_t). \end{equation} Then $M^N_t\rightarrow M_t$ and $E^N_t\rightarrow E_t$ in probability.
\end{lem}
We first show that Lemma \ref{lemma: WCOG} implies Lemma \ref{lemma: connect critical times}; the remainder of this section is dedicated to the proof of Lemma \ref{lemma: WCOG}.
\begin{proof}[Proof of Lemma \ref{lemma: connect critical times}] Let us assume, for the moment, that Lemma \ref{lemma: WCOG} holds. Throughout, let $(x_i)_{i=1}^{l^N}$ be the vertex data of the random graph process, which we recall are independent of time.

Firstly, suppose for a contradiction that $t_\mathrm{g}< t_\mathrm{c}$. Then $\varphi(g_{t_\mathrm{c}})>0$, but we bound \begin{equation}\label{eq: use of CS 0} \varphi(g^N_{t_\mathrm{c}}) \le \left(\frac{1}{N}C_1(G^N_{t_\mathrm{c}})\right)^\frac{1}{2}\left(\frac{1}{N}\sum_{i=1}^{l^N}\varphi(x_i)^2\right)^\frac{1}{2}.\end{equation} The first term converges to $0$ in probability, by definition of the phase transition in Theorem \ref{thrm: RG1}, and the second term is bounded in $L^2$ by hypothesis (B2.). This implies that $\varphi(g^N_{t_\mathrm{c}})\rightarrow 0$ in probability, which contradicts Lemma \ref{lemma: WCOG}; we must therefore have that $t_\mathrm{g}\ge t_\mathrm{c}.$ \medskip \\ \medskip \\ Conversely, if $t< t_\mathrm{g}$, then $M_t=0$ by definition. Now, the convergence \begin{equation} \frac{1}{N}C_1(G^N_t) = M^N_t\rightarrow 0\end{equation} in probability implies that the largest cluster is of the order $o_\mathrm{p}(N)$, which is only possible if $t\le t_\mathrm{c}$ by Lemma \ref{lemma: basic phase transition}. Since $t<t_\mathrm{g}$ was arbitrary, we must have $t_\mathrm{g}\le t_\mathrm{c}$, and together with the previous argument, we have shown that $t_\mathrm{g}=t_\mathrm{c}$ as claimed. \end{proof} 
 The proof of Lemma \ref{lemma: WCOG} is based on the following argument. We know, from Theorem \ref{thrm: RG2}, that any `mesoscopic' clusters contain negligable mass; thanks to the integrability assumption (A2.), the same is true for the energy. Therefore, almost all mass and energy either belongs to the `microscopic' scale, whose convergence is quantified by Lemma \ref{lemma: local uniform convergence of stochastic coagulent}, or the giant component, whose convergence is the subject of interest here. Therefore, with a suitable approximation argument, the claimed convergence will follow from the quoted results.  \medskip\\ We begin with a preparatory lemma; throughout, we will assume the notation of Lemma \ref{lemma: WCOG}.
For the proof of of Lemma \ref{lemma: connect critical times}, and later Theorem \ref{thrm: convergence of stochastic coagulent}, we will wish to study the convergence of integrals $\langle \varphi f, \mu^N_t\rangle$, for bounded continuous functions $f$ with non-compact support. However, the convergence result Lemma \ref{lemma: local uniform convergence of stochastic coagulent} only gives us information when the support of $f$ is compact. Our second preparatory lemma allows us to approximate the integrals $\langle \varphi f, \mu^N_t\rangle$ for functions $f$ whose support is bounded in the $\pi_0$-direction.

\begin{lem}[A step towards uniform integrability]\label{lemma: STUI}
Let $\mu_0$, $(\mu^N_t)_{t\ge 0}$ be as in the previous lemma. Then, for every $r>0$,
\begin{equation}
    \beta(r,\eta):= \sup_{N\geq 1}
      \mathbb{E}\left[\sup_{t\geq 0}\bigg\langle \varphi 1[\varphi(x)>\eta, \pi_0(x)\leq r], \mu^N_t\bigg\rangle \right]
    \rightarrow 0\hspace{1cm} \text{as }\eta\rightarrow \infty.
\end{equation} 
\end{lem}
\begin{proof}
 We note that $\langle \varphi 1[\varphi(x)>\eta, \pi_0(x)\le r], \mu^N_t\rangle $ depends on $\mu^N_t$ only through the pushforward $\pi_\# \mu^N_t$, since the integrand only depends on the values of $\pi$ at the different particles.
From Lemma \ref{lemma: coupling}, we can find random graphs $G^N_t$, such that $\mathbf{x}_N$ is an enumeration of the atoms of $\mu^N_0$ and  $\pi_\star(G^N_t)=\pi_\#\mu^N_t$ for all times $t$.
With this coupling, we express the integral as follows:
\begin{equation}\begin{split} 
\langle \varphi 1[\varphi(x)>\eta, \pi_0(x)\le r], \mu^N_t\rangle &= \frac{1}{N}\sum_{\text{Clusters }\mathcal{C}\subset G^N_t}\varphi(\mathcal{C})1[\varphi(\mathcal{C})>\eta, \pi_0(\mathcal{C})\le r] \\[1ex]
& = \frac{1}{N}\sum_{j=1}^{l^N(t)}\sum_{i \in \mathcal{C}_j(G_t^N)} \varphi(x_i)1\left[\varphi(\mathcal{C}_j(G^N_t))>\eta, \pi_0(\mathcal{C}_j(G^N_t))\le r\right] .   \end{split}\end{equation}
Using Cauchy-Schwarz, we bound
\begin{equation} \begin{split}
 &\sup_{t\geq 0}\hspace{0.1cm}\bigg\langle \varphi 1[\varphi(x)>\eta, \pi_0(x)\leq r], \mu^N_t\bigg\rangle
  \\[1ex] & \hspace{0.5cm}
\leq
  \left(\frac1N\sum_{j=1}^{l^N(0)}\varphi(x_j)^2 \right)^{\frac12}
  \left(\sup_{t\geq 0}\frac1N \sum_{j=1}^{l^N(t)}\sum_{i \in \mathcal{C}_j(G_t^N)}    1\left[\varphi(\mathcal{C}_j(G_t^N)) > \eta, C_j(G_t^N)\leq r\right]\right)^{\frac12} \\[1ex] & \hspace{0.5cm}
=
  \left(\frac1N\sum_{i=1}^{l^N(0)} \varphi(x_i)^2 \right)^{\frac12}
  \left(\sup_{t\geq 0}\hspace{0.1cm}\bigg\langle \pi_0 1[\varphi(x)>\eta, \pi_0(x)\leq r], \mu^N_t\bigg\rangle\right)^{\frac12}.
\end{split}\end{equation}
As remarked in Definition \ref{def: GNT}, the data $x_i$ associated with the graph nodes are constant in time, so the first factor is independent of $t\ge 0$, and is bounded in $L^2$ by the second assertion of (B2.).
Therefore, it is sufficient to prove the claim with $\varphi$ replaced by $\pi_0$.

Now we note that with probability one
\begin{equation*}
  \sup_{t\geq 0}\,\left\langle \pi_0 1[\varphi(x)>\eta, \pi_0(x)\leq r], \mu^N_t\right\rangle \leq
  r \sup_{t\geq 0}\,\left\langle 1[\varphi(x)>\eta], \mu^N_t\right\rangle \leq
  \frac{r}{\eta} \sup_{t\geq 0}\,\left\langle \varphi, \mu^N_t\right\rangle = \frac{r}{\eta}\left\langle \varphi, \mu^N_0\right\rangle
\end{equation*}
and the result follows from (B2.). 
\end{proof}
Using the preparatory lemma developed above, we now prove Lemma \ref{lemma: WCOG}. 
\begin{proof}[Proof of Lemma \ref{lemma: WCOG}]   Throughout, we let $(\mu^N_t)_{t\geq 0}$ be a stochastic coagulant coupled to a random graph process $(G^N_t)_{t\geq 0}$, as described in Section \ref{sec: coupling_to_random_graph} with vertex data $\mathbf{x}_N=(x_i)_{i=1}^{l^N(0)}$; thanks to this construction, $M^N_t$ is exactly the size of the largest cluster in $G^N_t$, and $E^N_t$ are the sums \begin{equation} E^N_t=\left(N^{-1}\sum_{j\in \mathcal{C}_1(G^N_t)} \pi_i(x_j)\right)_{i=1}^n. \end{equation}   The case $t=0$ is trivial, and can be omitted. We deal first with the $0^\text{th}$ coordinate $M^N_t$; the cases for the $1^\text{st},...,n^\text{th}$ coordinates $E^N_t$ are entirely analagous. \medskip \\ Fix $t> 0$, and let $\xi_N$ be a sequence, to be constructed later, such that \begin{equation}\label{eq: choice of xiN for WCOG}
       \xi_N\rightarrow \infty; \hspace{1cm} \frac{\xi_N}{N}\rightarrow 0.
   \end{equation}  We now construct `bump functions' as follows.  Let $\eta_r\rightarrow \infty$ be a sequence growing sufficiently fast that, in the notation of Lemma \ref{lemma: STUI}, $\beta(r, \eta_r)\rightarrow 0$, and let
 \begin{equation}
       S_{(r)} := \{x \in S: \pi_0(x)< r,  \varphi(x)\leq \eta_r\}.
 \end{equation}
 Let $\widetilde{h}_r$ be the indicator $\widetilde{h}_r=1[\pi_0(x)< r]$, and construct a continuous, compactly supported function $\widetilde{f}_r$ such that
 \begin{equation}
      0\leq \widetilde{f}_r\leq 1;\hspace{1cm} \widetilde{f}_r=1 \hspace{0.1cm} \text{ on } S_{(r)};\hspace{1cm} \widetilde{f}_r(x)=0 \hspace{0.1cm} \text{ if } \pi_0(x)\ge r.
 \end{equation}
 The final condition is compatible with continuity because $\pi_0:S\rightarrow \mathbb{N}$ is continuous and integer valued. We define $f_N=\widetilde{f}_{\xi_N}$ and $h_N=\widetilde{h}_{\xi_N}$.  We now decompose the difference $M^N_t-M_t:$ \begin{equation}\label{eq: decomposition of erorr in WCOG}\begin{split} M^N_t-M_t &= \underbrace{(\langle \pi_0, \mu_t\rangle -\langle \pi_0 f_N, \mu_t\rangle)}_{:=\mathcal{T}^1_N} + \underbrace{\langle \pi_0 f_N, \mu_t-\mu^N_t\rangle}_{:=\mathcal{T}^2_N} \\[1ex]&\hspace{2cm}+ \underbrace{\langle \pi_0 (f_N-h_N), \mu^N_t\rangle}_{:=\mathcal{T}^3_N} +\underbrace{
   \langle \pi_0 h_N, \mu^N_t\rangle - (\langle \pi_0, \mu^N_0\rangle-M^N_t)}_{:=\mathcal{T}^4_N}
   \\[1ex]&\hspace{3cm}+ \underbrace{\langle \pi_0, \mu^N_0-\mu_0\rangle}_{:=\mathcal{T}^5_N} .\end{split} \end{equation} where we recall that $M_t=\langle \pi_0, \mu_0-\mu_t\rangle$. We now estimate the errors $\mathcal{T}^i_N$, $i=1,3,4,5;$ the remaining term $\mathcal{T}^2_N$ will be dealt with separately, and requires careful construction of the sequence $\xi_N$. \paragraph{1. Estimate on $\mathcal{T}^1_N$.} Let $z_N$ be the lower bound $z_N=1_{S_{(\xi_N)}}$, so that $z_N \le f_N \le 1$. As $N\rightarrow \infty$, $\pi_0 z_N \uparrow \pi_0$, and so by monotone convergence, $
       \langle \pi_0 z_N, \mu_t\rangle \uparrow \langle \pi_0, \mu_t\rangle$. This implies that the (nonrandom) error $\mathcal{T}^1_N \rightarrow 0$.
\paragraph{2. Estimate on $\mathcal{T}^3_N$.} From the definitions of $f_N, h_N$, we observe that \begin{equation}
       |\mathcal{T}^3_N(t)|=\langle \pi_0(h_N-f_N), \mu^N_t\rangle \le  \langle \pi_0 1[\pi_0(x)<\xi_N, \varphi(x)>\eta_{\xi_N}], \mu^N_t\rangle.
   \end{equation} Therefore, in the notation of Lemma \ref{lemma: STUI}, $\mathbb{E}\left[\sup_{t\ge 0}|\mathcal{T}^3_N(t)|\right] \leq \beta(\xi_N, \eta_{\xi_N})$.
   By construction of $\eta_r$, and since $\xi_N \rightarrow \infty$, it follows that $\mathbb{E}[ \sup_{t\ge 0}|\mathcal{T}^3_N(t)|] \rightarrow 0,$ which implies convergence to $0$ in probability.
\paragraph{3. Estimate on $\mathcal{T}^4_N$.} Recalling that $h_N(x)=1[\pi_0(x)<\xi_N]$ and using the coupling to random graphs, we have the equality \begin{equation}
           \begin{split}
               \langle \pi_0 h_N, \mu^N_t\rangle & = \langle \pi_0, \mu^N_0\rangle-M^N_t1\left[M^N_t\ge \frac{\xi_N}{N}\right]-\frac{1}{N}\sum_{j\ge 2:C_j(G^N_t)\ge \xi_N}\hspace{0.1cm}\sum_{i\in C_j(G^N_t)}\pi_0(x_i) 
           \end{split} 
       \end{equation} which gives the equality \begin{equation}\label{eq: form of T4n} \mathcal{T}^4_N=-M^N_t1\left(M^N_t\le \frac{\xi_N}{N}\right)-\frac{1}{N}\sum_{j\ge 2: C_j(G^N_t)\ge \xi_N} \pi_0(\mathcal{C}_j(G^N_t)). \end{equation}  Using Cauchy-Schwarz, we bound \begin{equation}\begin{split}\label{eq: bound on T4} \abs{\mathcal{T}^4_N(t)} & \le \left(\frac{1}{N}\sum_{j\ge 2: C_j(G^N_t)\ge \xi_N} C_j(G^N_t)\right)^\frac{1}{2}\left(\frac{1}{N}\sum_{i=1}^{l^N(0)} \varphi(x_i)^2\right)^\frac{1}{2} +\frac{\xi_N}{N}. \end{split}\end{equation} The first term converges to $0$ in probability by Theorem \ref{thrm: RG2} and (B2.), and the second converges to $0$ since $\xi_N\ll N$. Together, these imply that $\mathcal{T}^4_N(t)\rightarrow 0$ in probability.

 \paragraph{4. Estimate on $\mathcal{T}^5_N$.} Using the first part of (B2.), we have the convergence in distribution \begin{equation} \langle \pi_0, \mu^N_0\rangle \rightarrow \langle \pi_0, \mu_0\rangle \end{equation} which implies that $\mathcal{T}^5_N\rightarrow 0$ in probability as desired. 
\paragraph{5. Construction of $\xi_N$, and convergence of $\mathcal{T}^2_N$.} It remains to show how a sequence $\xi_N$ can be constructed such that $\mathcal{T}^2_N \rightarrow 0$ in probability and such that (\ref{eq: choice of xiN for WCOG}) holds. Recalling the definition of $\tilde{f}_r$ above, let $A^1_{r,N}$ be the events $    A^1_{r,N}=\{ |\langle \varphi \widetilde{f}_r, \mu^N_t-\mu_t\rangle|<\frac{1}{r}\}$;
 as $N\rightarrow \infty$ with $r$ fixed, both $\mathbb{P}(A^1_{r,N})\rightarrow 1$, by Lemma \ref{lemma: local uniform convergence of stochastic coagulent}. We now define $N_r$ inductively for $r\geq 1$ by setting $N_1=1$, and letting $N_{r+1}$ be the minimal $N>\max(N_r, (r+1)^2)$ such that, for all $N'\ge N$, $\mathbb{P}(A^1_{r+1,N'})>\frac{r}{r+1}. $
 Now, we set $\xi_N=r$ for $N\in [N_r, N_{r+1}).$ It follows that $\xi_N \rightarrow \infty$ and $\xi_N\leq \sqrt{N}\ll N$, and
 \begin{equation}
       \mathbb{P}\left(C_1(G^N_t))\geq \xi_N\right)\ge 1-\frac{1}{\xi_N} \rightarrow 1. 
 \end{equation} Therefore, $\xi_N$ satisfies the requirements (\ref{eq: choice of xiN for WCOG}) above. Moreover, \begin{equation}
       \mathbb{P}\left(|\mathcal{T}^2_N| <\frac{1}{\xi_N}\right) \ge \mathbb{P}\left(A^1_{\xi_N,N}\right) > 1-\frac{1}{\xi_N}\rightarrow 1
 \end{equation}
 and so, with this choice of $\xi_N$, $\mathcal{T}^2_N \rightarrow 0$ in probability. Since we have now dealt with every term appearing in the decomposition (\ref{eq: decomposition of erorr in WCOG}), it follows that $M^N_t\rightarrow M_t$ in probability, as claimed. \medskip \\ The arguments for the $1^\text{st}-n^\text{th}$ components $E^N_t$ are identical to those above, using the same bound (\ref{eq: bound on T4}) on $\mathcal{T}^4_N$. \end{proof} 
We also note, for future use, an important corollary of this argument.
\begin{cor}\label{corr: gel at tgel} At the instant of gelation, the gel is negligible: $g_{t_\mathrm{g}}=0$.  \end{cor} \begin{proof} For the $0^\text{th}-n^\text{th}$ components, this follows from the critical case of Theorem~\ref{thrm: RG1} exactly as in (\ref{eq: use of CS 0}). The remaining $m$ components $g^i_t, i>n$ are identically $0$ by the symmetry (A1.), as in Lemma \ref{lemma: E and U}. \end{proof}
\section{\textbf{Behaviour of the Second Moments}}
\label{sec: finiteness of second moment} In this section, we consider part 2 of Theorem \ref{thrm: Smoluchowski equation}, concerning the behaviour of the second moments $\mathcal{Q}(t)_{ij}=\langle \pi_i\pi_j, \mu_t\rangle, 0\le i, j\le n$ and $\mathcal{E}(t)=\langle \varphi^2, \mu_t\rangle$. Following \cite{L78,N00}, one might expect that the gelation time $t_\mathrm{g}$ corresponds to a divergence of $\mathcal{E}(t)$ as $t\uparrow t_\mathrm{g}$; by an approximation argument, we will show that this is indeed the case. We also introduce a \emph{duality argument}, corresponding to Theorem \ref{thrm: coupling supercritical and subcritical}, which allows us to prove that $\mathcal{E}$ is finite on $(t_\mathrm{g}, \infty)$. The final assertion follows from the fact that $g_{t_\mathrm{g}}=0$, which is the content of Corollary \ref{corr: gel at tgel}.
\subsection{\textbf{Subcritical Regime}} We first deal with the subcritical regime $[0, t_\mathrm{g})$, to show that the second moments $\mathcal{Q}_{ij}(t), \mathcal{E}(t)$ are finite and increasing on this interval, and that $t_\mathrm{g}$ is exactly the first time at which $\mathcal{E}$ diverges.
\begin{lem}\label{lemma: second moment before tgel} Suppose $\mu_0$ satisfies Assumption \ref{hyp: A}, and let $(\mu_t)_{t\ge 0}$ be the corresponding solution to (\ref{eq: E+G}). The second moments $\mathcal{Q}(t)_{ij}=\langle \pi_i\pi_j, \mu_t\rangle, 0\le i, j\le n,\hspace{0.1cm} \mathcal{E}(t)=\langle \varphi^2, \mu_t\rangle$ are finite, continuous and increasing on $[0, t_\mathrm{g})$, and $\mathcal{E}(t)=\langle \varphi^2, \mu_t\rangle$ increases to infinity as $t\uparrow t_\mathrm{g}$, where $t_\mathrm{g}$ is the associated gelation time. \end{lem} 

The ideas of this argument follow \cite{N00}, where there is a similar result for \emph{approximately multiplicative} kernels, for which the total rate $\overline{K}(x,y)$ is bounded above \emph{and below} by nonzero multiples of $\widetilde{\varphi}(x)\widetilde{\varphi}(y)$, where $\widetilde{\varphi}$ is a mass function playing the same r\^ole as our $\varphi$. Unfortunately, this cannot be applied directly, for two reasons. \begin{enumerate}[label=\roman{*}).]
    \item Firstly, the total rate in (\ref{eq: overline K}) contains the terms $a_{ij}\pi_i(x)\pi_j(y), n\le i,j\le n+m$ of indefinite sign.
    \item Secondly, the remaining combination of $\pi_i, 1\le i\le n$ is not \emph{a priori} of approximately multiplicative form: particles where some $\pi_i$ are small, and others large, will in general prevent such a bound from holding.
\end{enumerate} Our strategy will be as follows. \begin{enumerate}
    \item Firstly, we will show that if $(\mu_t)_{t\ge 0}$ solves (\ref{eq: E+G}), then the pushforward measures $(\pi_\#\mu_t)_{t<t_\mathrm{g}(\mu_0)}$ solve a modified equation (\ref{eq: mE+G}) on the simpler space $S^\Pi=\NN\times [0,\infty)^n\times \mathbb{R}^m$, with a reduced kernel $K^{\Pi, \mathrm{m}}$. This allows us to eliminate the terms of indefinite sign mentioned above. This new equation has unique solutions, and so $\nu_t=\pi_\#\mu_t$ is the unique solution starting at $\nu_0=\pi_\#\mu_0$; in particular, the second moments $\langle \varphi^2, \nu_t\rangle$, $\langle \varphi^2, \mu_t\rangle$ coincide, and gelation takes place at the same time $t_\mathrm{g}(\mu_0)=t_\mathrm{g}(\nu_0)$. Therefore, we can prove the desired result working solely at the level of (\ref{eq: mE+G}).
    \item Thanks to results of Norris \cite[Theorem 2.1]{N00}, if $(\nu_t)_{t\ge 0}$ is a solution to (\ref{eq: mE+G}) with $\langle \varphi^2, \nu_0\rangle <\infty$, then there exists $t_\mathrm{e}=t_\mathrm{e}(\nu_0)>0$ such that $\langle \varphi^2, \nu_t\rangle$ is locally integrable on $[0,t_\mathrm{e})$ and such that $\langle \varphi^2, \nu_t\rangle \uparrow \infty$ as $t\uparrow t_\mathrm{e}$. 
    \item We introduce a truncated state space $S^\Pi_\epsilon$, which excludes particles where any $\pi_i/\pi_0, 1\le i\le n$ is either very large or very small, and construct new initial data $\nu^\epsilon_0$ which are supported in this space. In this context, the kernel $K^{\Pi, \mathrm{m}}$ is approximately multiplicative, and so \cite[Theorem 2.2]{N00} guarantees that the solutions $(\nu^\epsilon_t)_{t\ge 0}$ undergo gelation at exactly the blow-up time $t_\mathrm{e}(\nu^\epsilon_0)$.
    \item We argue, from the characterisation of the gelation time in Section \ref{sec: coupling_to_random_graph}, that our construction gives an approximation of the gelation times: $t_\mathrm{g}(\nu^\epsilon_0)\rightarrow t_\mathrm{g}(\nu_0)$. We will argue, based on a system of ordinary differential equations for the moments $\langle \pi_i\pi_j, \nu_t\rangle = \langle \pi_i\pi_j, \mu_t\rangle$, that the blowup time is also continuous: $t_\mathrm{e}(\nu^\epsilon_0)\rightarrow t_\mathrm{e}(\nu_0)$. Together with the previous points, this proves the claimed result.
\end{enumerate} We begin by introducing the modified equation. \begin{lem}\label{lemma: modified equation} Let $K^{\Pi, \mathrm{m}}$ be the kernel on $S^\Pi=\mathbb{N}\times[0,\infty)^{n}\times \mathbb{R}^m$ given by \begin{equation} K^{\Pi, \mathrm{m}}(p,q,dr)=\left(\sum_{i,j=1}^na_{ij}p_iq_j\right)\delta_{p+q}(dr). \end{equation} Consider the corresponding equation incorporating gel, for measures on $S^\Pi$, which we write as \begin{equation} \tag{m$\Pi$Fl}\label{eq: mE+G}
    \nu_t=\nu_0+\int_0^t L^\mathrm{m}_\mathrm{g}(\nu_s)ds.
\end{equation} Let $\mu_0$ be a measure on $S$ satisfying Assumption \ref{hyp: A}, and let $(\mu_t)_{t\ge 0}$ be the corresponding solution to (\ref{eq: E+G}). Then the pushforward measures $\nu_t=\pi_\#\mu_t$ are the unique solution to (\ref{eq: mE+G}) starting at $\nu_0=\pi_\#\mu_0$.\end{lem} \begin{rmk} Under the new kernel $K^{\Pi,\mathrm{m}}$, the quantities $\pi_i$ are still conserved for $0\le i\le n$, but not for $n+1\le i\le n+m$. However, since we seek to analyse $\langle \varphi^2, \mu_t\rangle, \varphi=\sum_{0\le i\le n} \pi_i$, we will not need any conservation properties of $\pi_i$ for $i>n$ in this section.   \end{rmk}  \begin{proof}[Sketch Proof of Lemma \ref{lemma: modified equation}]  Much of the proof consists of algebraic manipulations, using the definitions and hypotheses in Definition \ref{def: BCS}. In the interest of brevity, such manipulations will omitted. \medskip \\ Let us first consider the reflected measures $R_\# \mu_t =\mu_t\circ R^{-1}$ on $S$. By (A1.), $R_\#\mu_0=\mu_0$, and using part iii) of Definition \ref{def: BCS}, one can show that for all $t\ge 0$, all finite measures $\mu$ on $S$ and all bounded, measurable functions $f$ on $S$, $\langle f\circ R, L(\mu)\rangle =\langle f, L(R_\#\mu)\rangle$. From this, and performing a similar manipulation for the gel term, it follows that $(R_\#\mu_t)_{t\ge 0}$ also solves the equation (\ref{eq: E+G}) with the same initial data which implies, by uniqueness, we must have $\mu_t=R_\#\mu_t$ for all $t\ge 0$. Using this, one can now similarly prove that, for all $t$ and $f$ as above, \begin{equation} \begin{split} \label{eq: symmetry under R}
 \langle f, L(\mu_t)\rangle &=  \int_{S^3} (f(z)-f(x)-f(y))K(Rx,y,dz)\mu_t(dx)\mu_t(dy) \\
 &=  \int_{S^3} (f(z)-f(x)-f(y))K(x,Ry,dz)\mu_t(dx)\mu_t(dy).
\end{split} \end{equation} Taking a linear combination, and again performing a similar manipulation for the gel term, it follows that $\mu_t$ solves the equation analagous to (\ref{eq: E+G}) for the symmetrised kernel
 \begin{equation}\label{eq: modified K} 
 \begin{split}
 K^\mathrm{Sym}(x,y,dz)
 =\frac{1}{4}K(Rx, y, dz)+\frac{1}{2}K(x,y,dz)+\frac{1}{4}K(x,Ry, dz). \end{split} 
\end{equation} Since the coagulation rate in $K^\mathrm{Sym}$ only depends on $\pi(x), \pi(y)$, one can verify that the pushforward measures  $\pi_\#\mu_t$ on $S^\Pi$ solve the projected equation (\ref{eq: mE+G}) as claimed. \end{proof} We now turn to the second point, which concerns the moment behaviour of the solutions to (\ref{eq: mE+G}). The following result follows from ideas of \cite{N00}, which we will briefly sketch. \begin{lem}\label{lemma: explosition} Let $\nu_0$ be a measure on $S^\Pi$ satisfying Assumption \ref{hyp: A}, and let $(\nu_t)_{t\ge 0}$ be the corresponding solution to (\ref{eq: mE+G}). Then there exists $t_\mathrm{e}=t_\mathrm{e}(\nu_0)>0$ such that $t\mapsto \langle \varphi^2, \nu_t\rangle$ is finite and increasing on $[0,t_\mathrm{e})$, and $\langle \varphi^2,\nu_t\rangle \uparrow \infty$ as $t\uparrow t_\mathrm{e}$. Moreover, $(\nu_t)_{t<t_\mathrm{e}}$ is conservative, and so $t_\mathrm{e}(\nu_0)\le t_\mathrm{g}(\nu_0)$. \end{lem}  The subscript $_\mathrm{e}$ here denotes `explosion': $t_\mathrm{e}$ is the first time the second moment diverges to $\infty$. \begin{proof}[Sketch Proof of Lemma \ref{lemma: explosition}] This argument applies different results from \cite{N00} to our case. We say that a local solution $(\nu_t)_{t<T}$ to (\ref{eq: mE+G}) is \emph{strong} if the map $t\mapsto \langle \varphi^2, \nu_t\rangle $ is integrable on compact subsets of $[0,T)$. Applying the results of \cite[Theorem 2.1]{N00}, there exists a unique maximal strong solution $(\nu'_t)_{t<t_\mathrm{e}(\nu_0)}$ to (\ref{eq: mE+G}), which is conservative and that $t_\mathrm{e}(\nu_0)\ge C\langle \varphi^2, \nu_0\rangle^{-1}$ for some constant $C$ depending on $A$. \medskip \\ We next apply Corollary \ref{cor: maximal conservative solutions} to see that this solution must be an initial segment of $(\nu_t)_{t<t_\mathrm{g}(\nu_0)}$: that is, $t_\mathrm{e}(\nu_0)\le t_\mathrm{g}(\nu_0)$, and $\nu'_t=\nu_t$ for all $t\le t_\mathrm{e}(\nu_0)$. Therefore, the results of \cite{N00} will apply to our process $(\nu_t)_{t<t_\mathrm{e}(\nu_0)}$. \medskip \\   Since $(\nu_t)_{t<t_\mathrm{e}(\nu_0)}$ is conservative, we follow the ideas of \cite[Proposition 2.7]{N00}, to obtain the integral relations,
  for all $t<t_\mathrm{e}$ and $0\le i, j\le n$, \begin{equation} \label{eq: ODE1} \langle  \pi_i\pi_j, \nu_t\rangle =\langle  \pi_i\pi_j, \nu_0\rangle +2\int_0^t \sum_{k,l=1}^n \langle \pi_i\pi_k,\nu_s\rangle a_{kl}\langle \pi_l\pi_j,\nu_s\rangle ds.\end{equation} These immediately imply that $\langle \varphi^2, \nu_t\rangle $ is bounded on compact subsets of $[0, t_\mathrm{e})$, and in particular does not diverge before $t_\mathrm{e}$. Moreover, since all terms on the right-hand side are nonnegative, these relations imply that all moments $\langle \pi_i\pi_j, \nu_t\rangle $ and $\langle \varphi^2, \nu_t\rangle$ are increasing on $[0,t_\mathrm{e}).$ \medskip \\ Finally, we show that $\langle \varphi^2, \nu_t\rangle$ diverges near $t_\mathrm{e}(\nu_0)$. This follows from the time-of-existence estimate quoted above: for $t<t_\mathrm{e}$, the unique maximal strong solution starting at $\nu_t$ is precisely $(\nu_{s+t})_{s<t_\mathrm{e}-t}$, and so for some $C=C(A)<\infty$, \begin{equation} t_\mathrm{e}-t\ge C\langle \varphi^2, \nu_t\rangle^{-1}. \end{equation} This rearranges to show that $\langle \varphi^2, \nu_t\rangle \ge C(t_\mathrm{e}-t)^{-1}$ which diverges as $t\uparrow t_\mathrm{e}$, as claimed.   \end{proof} In order to obtain the full connection of the explosion and gelation times, we modify the setting to exclude the problematic particles identified above. Let \begin{equation}
    S^\Pi_\epsilon= \{p\in S^\Pi: \epsilon \pi_0(p) \le \pi_i(p) \le (\epsilon^{-1}+\epsilon)\pi_0(p)\text{ for all } 1\le i\le n\}.
\end{equation} Note that this state space is preserved under the kernel $K^{\Pi,\mathrm{m}}$. Moreover, on the reduced state space $S^\Pi_\epsilon$, the modified kernel $K^{\Pi, \mathrm{m}}$ is \emph{approximately multiplicative} \cite{N00} in the sense that, for some $\delta_\epsilon>0$ and $\Delta_\epsilon<\infty$, we have \begin{equation}
    \delta_\epsilon\hspace{0.1cm}\varphi(p)\varphi(q) \leq \overline{K^{\Pi, \mathrm{m}}}(p,q) \leq  \Delta_\epsilon\hspace{0.1cm}\varphi(p)\varphi(q)
\end{equation}for all $p,q \in S^\Pi_\epsilon$. \medskip \\ We now construct approximations $\nu^\epsilon_0$ to $\nu_0$ which are supported on $S^\Pi_\epsilon$. Let us fix $\mu_0$ satisfying Assumption \ref{hyp: A} and $\nu_0=\pi_\#\mu_0$; for any $\epsilon>0$, let $\nu_0^\epsilon$ be given by specifying, for all bounded measurable functions $h$ on $S^\Pi$, \begin{equation}\begin{split} \label{eq: shifted in data} \int_{S^\Pi}h(p)\nu^\epsilon_0(dp)&\\& \hspace{-1cm}=\int_{S^\Pi} h(p_0, p_1+\epsilon,....p_n+\epsilon,p_{n+1},....p_{n+m})1[p_i\le \epsilon^{-1} \text{ for all }1\le i\le n]\nu_0(dp). \end{split}\end{equation} In this way, we shift $\nu_0$ slightly away from the axes, while also truncating when any $\pi_i$ becomes large. It follows, from existence and uniqueness, that the solution $(\nu^\epsilon_t)_{t\ge 0}$ to (\ref{eq: mE+G}) starting at $\nu^\epsilon_0$ is supported on $S^\Pi_\epsilon$ for all $t\ge 0.$ We can now apply \cite[Theorem 2.2]{N00} to obtain the connection between gelation and explosion for these solutions: \begin{lem}\label{lemma: blowup and gelation} Let $(\nu^\epsilon_t)_{t\ge 0}$ be the solution to (\ref{eq: mE+G}) starting at the measure $\nu^\epsilon_0$ constructed above. Let $t_\mathrm{e}(\nu^\epsilon_0)$ be the explosion time of the second moment, as above, and $t_\mathrm{g}(\nu^\epsilon_0)$ the first time that $\nu^\epsilon_t$ fails to be conservative. Then $t_\mathrm{e}(\nu^\epsilon_0)=t_\mathrm{g}(\nu^\epsilon_0)$.\end{lem} This then connects the gelation phenomenon to the blowup of the second moment, as desired, but only for the special case of the truncated and shifted initial distribution. We now seek to remove this restriction to obtain the result for the original measures $\mu_0, \nu_0.$ To do this, we will show that $t_\mathrm{g}(\nu^\epsilon_0)\rightarrow t_\mathrm{g}(\nu_0)$ and $t_\mathrm{e}(\nu^\epsilon_0)\rightarrow t_\mathrm{e}(\nu_0)$ as we take $\epsilon\downarrow 0.$

\begin{lem}[Convergence of Gelation Times]\label{lemma: convergence of gelation times} Let $\nu_0, \nu^\epsilon_0$ be the measures constructed above, and $t_\mathrm{g}(\nu_0), t_\mathrm{g}(\nu^\epsilon_0)$ the corresponding gelation times. Then, as $\epsilon\downarrow 0$, $t_\mathrm{g}(\nu^\epsilon_0)\rightarrow t_\mathrm{g}(\nu_0)$.\end{lem}
\begin{proof} First, we recall that $\pi_1,...\pi_n$ are linearly independent in $L^2(\mu_0)$, and hence in $L^2(\nu_0)$, by hypothesis. Using the convergence $\langle \pi_i \pi_j, \nu^\epsilon_0\rangle \rightarrow \langle \pi_i\pi_j, \nu_0\rangle$, it follows that for $\epsilon>0$ small enough, and any $a_i$ with $\sum_i |a_i|=1$, we have $\langle (\sum_i a_i \pi_i)^2, \nu^\epsilon_0\rangle >0$. This, in turn, guarantees that $\pi_1,...\pi_n$ are linearly independent in $L^2(\nu^\epsilon_0)$, for all $\epsilon>0$ small enough.

We can now apply the explicit characterisation of $t_\mathrm{g}$ obtained in Lemma \ref{lemma: computation of tcrit} for the measures $\nu^\epsilon_0$: \begin{equation} t_\mathrm{g}(\nu^\epsilon_0)=\lambda_1(\Lambda(\nu^\epsilon_0))^{-1} \end{equation}  where $\Lambda(\nu^\epsilon_0)$ is the matrix $\Lambda(\nu^\epsilon_0)_{ij}=\sum_{k=1}^{n+m}\langle \pi_i\pi_k, \nu^\epsilon_0\rangle a_{kj}$ and $\lambda_1(\cdot)$ denotes the largest eigenvalue of a matrix. Moreover, as $\epsilon\downarrow 0$, the coefficients of the matrices $\Lambda(\nu^\epsilon_0)$ converge to the analagous matrix $\Lambda(\nu_0)$ for the measure $\nu_0$. \medskip \\ It is well-known, following for instance from \cite{zedek}, that as the coefficients of a matrix vary continuously, so to do the associated eigenvalues, meaning that \begin{equation} \lambda_1(\Lambda(\nu^\epsilon_0)) \rightarrow \lambda_1(\Lambda(\nu_0)) \end{equation} as $\epsilon\downarrow 0$. Combining this with the characterisation of $t_\mathrm{g}$ above, it follows that \begin{equation}\begin{split} t_\mathrm{g}(\nu^\epsilon_0)=\lambda_1(\Lambda(\nu^\epsilon_0))^{-1}& \rightarrow \lambda_1(\Lambda(\nu_0))^{-1}\\ & =t_\mathrm{g}(\nu_0) \end{split}\end{equation} as desired.
\end{proof}

Finally, we show the same result for the explosion times. Thanks to Lemma \ref{lemma: explosition} and (\ref{eq: ODE1}), the matrix of second moments $q_{ij}(t)=\langle \pi_i\pi_j, \nu_t\rangle, 1\le i, j\le n$ satisfies a closed differential equation, with locally Lipschtiz coefficients, on $[0,t_\mathrm{e})$. We will now show that $t_\mathrm{e}$ is exactly the time of existence of a solution started at $q_0$. \begin{lem}\label{lem: link explosion time and existence time} Consider the ordinary differential equations \begin{equation}\label{eq: ODEQ1} \tag{Q1} \dot{q}_t=b(q_t); \hspace{1cm} b(q)=2qA^+q, \hspace{1cm} q\in \mathrm{Mat}_n(\mathbb{R});\end{equation} \begin{equation}\label{eq: ODEQ2} \tag{Q2} \dot{z}_t=w(q_t)z_t, \hspace{1cm} w: \mathrm{Mat}_n(\mathbb{R})\rightarrow \mathrm{Mat}_{n+1}(\mathbb{R})\text{ linear}; \hspace{1cm} z\in \mathbb{R}^{n+1}.\end{equation}   Then, for all $(z_0,q_0)\in \mathbb{R}^{n+1}\times \mathrm{Mat}_n(\mathbb{R})$, there exists a unique maximal solution $\chi(t,z_0, q_0),\psi(t, q_0)$ starting at $(z_0,q_0)$, defined until the time $\zeta(q_0)$ where (\ref{eq: ODEQ1}) blows up. \medskip \\ Then, for any measure $\nu_0$ on $S^\Pi$, the time of existence is exactly the explosion time: \begin{equation} t_\mathrm{e}(\nu_0)=\zeta(q_0), \qquad (q_0)_{ij}=\langle \pi_i \pi_j, \nu_0\rangle, 1\le i, j\le n. \end{equation}  \end{lem} \begin{proof} Firstly, it is straightforward to verify that $q_t$ does not depend on the initial data $z_0$, since (\ref{eq: ODEQ1}) only depends on $q$; in particular, the blowup time $\zeta$ is a function only of $q_0$. It is also straightforward to verify that (\ref{eq: ODEQ2}) cannot blow up before $\zeta(q_0)$, since on compact subsets $[0,t]\subset [0, \zeta(q_0))$, the coefficients of (\ref{eq: ODEQ2}) are Lipschitz, uniformly in time. As a result, the time of existence for the pair (\ref{eq: ODEQ1}, \ref{eq: ODEQ2}) is exactly the time of existence $\zeta(q_0)$, as claimed.  \medskip \\ To link the explosion times $t_\mathrm{e}$ and the time of existence $\zeta(q_0)$, the equations (\ref{eq: ODE1}) show that the matrix $q_{ij}(t)=\langle \pi_i \pi_j, \nu_t\rangle, 1\le i,j\le n$ and the vector $z_t=\langle \pi_0\pi_i, \nu_t\rangle, 0\le i\le n$ solve the system (\ref{eq: ODEQ1}, \ref{eq: ODEQ2}) on $0\le t< t_\mathrm{e}(\nu_0)$, which implies that $t_\mathrm{e}(\nu_0)\le \zeta(q_0)$.  For the converse, for $t<t_\mathrm{e}$, we have the equality \begin{equation}  \chi(z_0, q_0)_0+\sum_{i=1}^n \psi(q_0, t)_{ii} = \langle \pi_0^2, \nu_t\rangle +\sum_{i=1}^n \langle \pi_i^2, \nu_t\rangle\end{equation} where the initial data are \begin{equation} q_0=(\langle \pi_i\pi_j, \nu_0\rangle)_{i,j=1}^n \hspace{1cm} z_0=(\langle \pi_i\pi_0, \nu_0\rangle)_{0\le i\le n}. \end{equation} The left hand side is bounded on compact subsets of $[0, \zeta(q_0))$ and the right-hand side dominates $\langle \varphi^2, \nu_t\rangle$ up to a constant $C$, which leads to a contradiction if we assume that $t_\mathrm{e}<\zeta(q_0)$, since $\langle \varphi^2, \nu_t\rangle\uparrow \infty$ as $t\uparrow t_\mathrm{e}(\nu_0)$.  We therefore have $\zeta(q_0)\le t_\mathrm{e}(\nu_0)$ which proves the equality desired.    \end{proof} We will now analyse the pair of equations presented above. This will prove the desired continuity of $t_\mathrm{e}$, and some points which will be helpful for later reference. \begin{lem}\label{lemma: ODE considerations} Consider the differential equations (\ref{eq: ODEQ1}, \ref{eq: ODEQ2}) in the previous lemma, and the sets \begin{equation} E=\mathrm{Mat}_n([0,\infty)); \hspace{1cm} E_\delta=\{q\in E: \forall i, q_{ii}> \delta\}; \hspace{1cm}E^\circ=\cup_{\delta>0} E_\delta;\end{equation} \begin{equation} E_\mathrm{cs}=\{q\in E: \text{ for all }i, j\le n \text{ and } t< \zeta(q), \text{ } \psi(q,t)_{ij}^2\le \psi(q,t)_{ii}\psi(q,t)_{jj}\}. \end{equation}  Then, if $q_0 \in E_\delta$, $(\psi(q_0,t))_{t<\zeta(q_0)} \subset E_\delta$, and similarly if $q_0\in E_\mathrm{cs}$, then $(\psi(q_0,t))_{t<\zeta(q_0)}\subset E_\mathrm{cs}$. We have the following properties: \begin{enumerate}[label=\roman{*}).] \item Let $J_\epsilon$ be the set \begin{equation} J_\epsilon  =\{q \in E: \hspace{0.2cm} \zeta(q)\ge\epsilon\}.\end{equation} Then for all $\epsilon, \delta>0$, the set $J_\epsilon\cap E_\delta\cap E_{\mathrm{cs}}$ is bounded. \item Suppose $q^\epsilon_0 \in E_\mathrm{cs}, \epsilon>0$ and $q^\epsilon_0 \rightarrow q_0 \in E_\mathrm{cs}\cap E^\circ$. Then $\zeta(q^\epsilon_0)\rightarrow \zeta(q_0).$ \item Suppose $I\subset \mathbb{R}_+$ is an open interval, and the map $(z_0,q_0): I\rightarrow \mathbb{R}^{n+1}\times (E_\mathrm{cs}\cap E^\circ)$ is continuous, and such that $t<\zeta(q_0(t))$ for all $t\ge 0.$ Then the maps $t\mapsto \psi(q_0(t), t)$ and $t\mapsto \chi(z_0(t), q_0(t), t)$  are continuous on $I$. 
 \end{enumerate} \end{lem} 

\begin{proof} 
\begin{enumerate}[label=\roman{*}).]
    \item Let us first fix $q\in E$. First of all write $a_\star=\min\{a_{ij}: a_{ij}>0\}$ and let $i,j$ be such that $a_{ij}>0$. We now estimate \begin{equation} \frac{d}{dt}\psi(t,q)_{ij}\ge 2 a_{ij}\psi(t,q)_{ij}^2\ge 2a_\star q_{ij}^2. \end{equation} This differential inequality may be integrated to obtain \begin{equation} \label{eq: bound to apply} \psi(t,q)_{ij}\ge \frac{q_{ij}}{1-2ta_\star q_{ij}}. \end{equation} In particular, this gives the upper bound $\zeta(q)\le (2a_\star q_{ij})^{-1}$, which implies the claimed boundedness of $J_\epsilon$ in the $(i,j)^\mathrm{th}$ coordinate whenever $a_{ij}>0$. \medskip \\ We will now extend this boundedness to all $n^2$ coordinates when we restrict to  $q\in E_\delta \cap J_{\epsilon}\cap E_\mathrm{cs}$. Let $M$ be the maximum diagonal entry of $q$: \begin{equation} M=\max_{1\le i\le n} q_{ii} \end{equation} and fix $i$ where this maximum is attained; by hypothesis on $A$, there exists $j\le n$ such that $a_{ij}\ge a_\star>0$. It is straightforward to see that the derivative $\frac{d}{dt}\psi(q,t)_{ij}$ is increasing along the solution, which implies the estimate \begin{equation} \psi\left(\frac{\epsilon}{2},q\right)_{ij}\ge \frac{\epsilon}{2} b(q)_{ij}=\epsilon \sum_{k,l\le n}q_{ik}a_{kl}q_{lj} \ge \epsilon  q_{ii}a_{ij}q_{jj} \ge \epsilon \delta a_\star M.\end{equation}By hypothesis, $\zeta(q)\ge \epsilon$, so $\zeta(\psi(\frac{\epsilon}{2},q))\ge \frac{\epsilon}{2}$. Applying the bound on $\zeta$ above, we find that \begin{equation} \frac{\epsilon}{2} \le \frac{1}{2a_\star^2\epsilon \delta M}.\end{equation} Finally, since we chose $q\in E_\mathrm{cs}$, we have the uniform bound  \begin{equation} \max_{ij} q_{ij} \le M\le (a_\star^2 \epsilon^2 \delta)^{-1}. \end{equation} 
    \item The lower semicontinuity of explosion times is standard, and follows from the continuous dependence on the initial data. Therefore, it is sufficient to prove that $\limsup_{\epsilon\rightarrow 0} \zeta(q^\epsilon)\le \zeta(q).$  \medskip \\ Suppose, for a contradiction, that for some $\eta>0$, we have $\limsup_{\epsilon\rightarrow 0} \zeta(q^\epsilon)>\zeta(q)+\eta$; by passing to a subsequence, we may assume that $\zeta(q^\epsilon)>\tau+\eta$ for all $\epsilon$, where we write $\tau=\zeta(q)$. Moreover, since $q_0^\epsilon \in E_\mathrm{cs}$ and $q^\epsilon\rightarrow q \in E^\circ$, we may assume that $q^\epsilon, q \in E_\delta \cap E_\mathrm{cs}$ for all $\epsilon$, for some $\delta>0$, which implies that $\psi(q^\epsilon,t)\in E_\delta\cap E_\mathrm{cs}$ for all $t<\zeta(q^\epsilon)$ and all $\epsilon>0$.\medskip\\  Now, if $t\le \tau$, we have $\zeta(\psi(t,q^\epsilon))=\zeta(q^\epsilon)-t \ge \eta$, which implies the containment \begin{equation} \{\psi(t,q^\epsilon): t\le \tau, \epsilon>0\} \subset E_\delta\cap J_\eta \cap E_\mathrm{cs} \end{equation} which we know, from item i)., to be bounded: for some $C<\infty$, \begin{equation}
        \{\psi(t,q^\epsilon): t\le \tau, \epsilon>0\} \subset \mathrm{Mat}_n([0,C]).
    \end{equation} By the lemma of leaving compact sets, there exists $s<\tau$ such that, for all $t\in (s,\tau)$, $\psi_t(q)\not \in \mathrm{Mat}_n([0,C]).$ However, if we pick $t\in (s,\tau)$, we have $\psi_t(q^\epsilon) \rightarrow \psi_t(q)$, by the continuity of the dependence in the initial conditions, which is a contradiction. Therefore, $\limsup_{\epsilon\rightarrow 0} \zeta(q^\epsilon)\le \zeta(q)$, which proves the claimed convergence.      
    \item  Let us first establish the claim for $\psi$. Firstly, we note that by ii)., the map $t\mapsto \zeta(q_0(t))$ is continuous on $I$. Therefore, fixing $t\in I$, we may choose choose  $\epsilon, \delta > 0$ such that, if $\abs{t-s} \le \delta$, then $s\in I$ and $s < \min \left(\zeta\left(q_0(s)\right), \zeta\left(q_0(t)\right)\right)-\epsilon$. Now, we observe that, for $s\in [t-\delta, t+\delta],$\begin{equation}
    |\psi(t,q_0(t))-\psi(s,q_0(s))|\le|\psi(t,q_0(t))-\psi(t,q_0(s) )|+|\psi(t,q_0(s))-\psi(s,q_0(s))|.
\end{equation} As $s\rightarrow t$, the first term converges to $0$ by continuity of the solution $\psi(q,t)$ in the initial data $q_0$; it is therefore sufficient to control the second term. By the choice of $\delta$, for all $s\in[t-\delta, t+\delta],$ we have \begin{equation} \zeta(\psi(s,q_0(s)))=\zeta(q_0(s))-s>\epsilon \end{equation} so that $\psi(s,q_0(s))\in J_\epsilon.$ Moreover, by compactness of $[t-\delta,t+\delta]$, there exists some $\eta>0$ such that $q_0(s) \in E_\eta$ for all $s\in [t-\delta, t+\delta]$, and since $q_0(s)\in E_\mathrm{cs}$ and these sets are preserved under the flow, we have $\psi(q_0(s),u)\in E_\eta\cap E_\mathrm{cs}$ for all $0\le u\le \zeta(q_0(s)).$ However, we showed in point i). above that that the intersection of these three regions is compact and so there exists a constant $M=M(\epsilon)$: for all $s\in[t-\delta, t+\delta]$, and for all $u \le t+\delta$, \begin{equation}\label{eq: bbded} u< \zeta(q_0(s));\hspace{1cm} |b(\psi(u,q_0(s))| \le M. \end{equation}
This implies the bound, for all $s\in[t-\delta,t+\delta]$, \begin{equation} |\psi(t,q_0(s))-\psi(s,q_0(s))| \le M|t-s|\end{equation} which implies the claimed continuity. \medskip \\ The case for $\chi(z_0(t), q_0(t), t)$ is similar. Let us fix $t\in I$; following the same argument leading to (\ref{eq: bbded}), there exists $\delta>0, M<\infty$ such that, if $s\in [t-\delta, t+\delta]$ then $s\in I$ and for all $u\le s$, $\psi(u, q_0(u)) \in \mathrm{Mat}_n([0,M])$. The equation (\ref{eq: ODEQ2}) can now be integrated directly to obtain, for $s\in [t-\delta, t+\delta]$, \begin{equation} \chi(s, z_0(s), q_0(s))=\exp\left(\int_0^s w\left(\psi(u, q_0(u))\right)du\right)z_0(s).\end{equation} In particular, it follows that $\chi(s, z_0(s), q_0(s))$ is bounded as $s$ varies in $[t-\delta, t+\delta]$. With this, the argument for $\psi$ can be modified to prove the same result 
\end{enumerate}  \end{proof} We can finally combine the previous lemmas to prove Lemma \ref{lemma: second moment before tgel}. \begin{proof}[Proof of Lemma \ref{lemma: second moment before tgel}] Let us fix $\mu_0$ satisfying Assumption \ref{hyp: A}, and let $\nu_0$ be its pushforward $\nu_0=\pi_\#\mu_0$; let $(\mu_t)_{t\ge 0}$ and $(\nu_t)_{t\ge 0}$ be the solutions to (\ref{eq: E+G}, \ref{eq: mE+G}) with these starting points, respectively. By Lemma \ref{lemma: modified equation}, $\nu_t$ is given by $\nu_t=\pi_\#\mu_t$ and in particular, $\mathcal{E}(t)=\langle \varphi^2, \mu_t\rangle=\langle \varphi^2, \nu_t\rangle$, $\mathcal{Q}_{ij}(t)=\langle \pi_i\pi_j, \mu_t\rangle =\langle \pi_i\pi_j, \nu_t\rangle$ and $t_\mathrm{g}(\nu_0)=t_\mathrm{g}(\mu_0)$. \medskip \\ From Lemma \ref{lemma: explosition}, we know that there exists $t_\mathrm{e}=t_\mathrm{e}(\nu_0)>0$ such that $ \mathcal{E}(t)=\langle\varphi^2,\nu_t\rangle$ is finite, continuous and increasing on $[0,t_\mathrm{e})$, and diverges to infinity as $t\uparrow t_\mathrm{e}.$ Moreover, thanks to the differential equations (\ref{eq: ODE1}), all components of $\mathcal{Q}(t)$ are continuous and increasing on $[0,t_\mathrm{e})$. \medskip \\ Consider next the shifted initial data $\nu^\epsilon_0$ given by (\ref{eq: shifted in data}); thanks to Lemma \ref{lemma: blowup and gelation}, we know that $t_\mathrm{g}(\nu^\epsilon_0)=t_\mathrm{e}(\nu^\epsilon_0).$ By Lemma \ref{lemma: convergence of gelation times}, $t_\mathrm{g}(\nu^\epsilon_0)\rightarrow t_\mathrm{g}(\nu_0)$. For the explosion times, we know from Lemma \ref{lem: link explosion time and existence time} that $t_\mathrm{e}(\nu^\epsilon_0)=\zeta(q^\epsilon_0)$ and $t_\mathrm{e}(\nu_0)=\zeta(q_0)$, where $q^\epsilon_0, q_0\in E$ are the matrixes \begin{equation} (q_0)_{ij}=\langle \pi_i\pi_j, \nu_0\rangle;\hspace{1cm} (q^\epsilon_0)_{ij}=\langle \pi_i\pi_j, \nu^\epsilon_0\rangle. \end{equation} By dominated convergence, $q^\epsilon_0\rightarrow q_0$; by hypothesis (A3.), each $(q_0)_{ii}=\langle \pi_i^2, \nu_0\rangle=\langle \pi_i^2, \mu_0\rangle>0$, so $q_0\in E_\delta$ for some $\delta>0$. Finally, for all $t<\zeta(q_0)=t_\mathrm{e}(\nu_0)$, $\psi(t,q_0)_{ij}=\langle \pi_i\pi_j, \nu_t\rangle $ which certainly satisfies the desired Cauchy-Schwarz inequality $\psi(t,q_0)_{ij}^2\le \psi(t,q_0)_{ii}\psi(t,q_0)_{jj}$, so $q_0\in E_\mathrm{cs}$. A similar argument shows that $q^\epsilon_0 \in E_\mathrm{cs}$ for all $\epsilon>0$, so Lemma \ref{lemma: ODE considerations} shows that $t_\mathrm{e}(\nu^\epsilon_0)=\zeta(q^\epsilon_0)\rightarrow \zeta(q_0)=t_\mathrm{e}(\nu_0)$. Comparing these two limits, $t_\mathrm{g}(\nu_0)=t_\mathrm{e}(\nu_0)$, concluding the proof. \end{proof}

\subsection{\textbf{The Critical Point}} Using the concepts introduced above, we next consider the behaviour at and near the critical time $t_\mathrm{g}$. \begin{lem} \label{lemma: divergence at critical point} In the notation of Lemma \ref{lemma: second moment before tgel}, we have \begin{equation} \mathcal{E}(t_\mathrm{g})=\infty=\lim_{t\rightarrow t_\mathrm{g}} \mathcal{E}(t). \end{equation} \end{lem} 
\begin{proof}We first show that $\mathcal{E}(t_\mathrm{g})=\infty$. Suppose, for a contradiction, that $\mathcal{E}(t_\mathrm{g})<\infty.$ Then, applying \cite[Proposition 2.7]{N00} as in Lemma \ref{lemma: explosition}, we see that, for some positive $\delta>0$, there exists a strong solution $(\nu_t)_{t<\delta}$ to (\ref{eq: E}), starting at $\mu_{t_\mathrm{g}}.$ This solution is conservative, so is an initial segment of the solution $(\nu_t)_{t\ge 0}$ to (\ref{eq: E+G}) starting at $\mu_{ t_\mathrm{g}}$. By uniqueness in Lemma \ref{lemma: E and U}, \begin{equation}
    \nu_t=\mu_{t_\mathrm{g}+t} \hspace{1cm} \text{ for all }t\ge 0.
\end{equation}
By Corollary \ref{corr: gel at tgel}, $\langle \varphi, \mu_{t_\mathrm{g}}\rangle = \langle \varphi, \mu_0\rangle$, and by definition of $t_\mathrm{g}$, \begin{equation} \langle \varphi, \mu_{t_\mathrm{g}+t}\rangle < \langle \varphi, \mu_{0}\rangle = \langle \varphi, \mu_{t_\mathrm{g}}\rangle \text{ for all }t>0. \end{equation}This contradicts the fact that $(\nu_t)_{t<\delta}$ is strong, which therefore shows that $\mathcal{E}(t)=\infty$.

The second point follows, because $t\mapsto \mu_t$ is continuous, and $\mu \mapsto \langle \varphi^2, \mu\rangle$ is lower semicontinuous, when $\mathcal{M}$
is equipped with the vague topology. \end{proof}

\subsection{\textbf{The Supercritical Regime}} We finally turn to the supercritical case; our result is as follows. 
\begin{lem}\label{lemma: second moment finite after tgel} In the notation of Lemma \ref{lemma: second moment before tgel}, the map $t\mapsto \mathcal{E}(t)$ is finite and continuous, and therefore locally bounded, on  $(t_\mathrm{g},\infty)$. \end{lem} 
The proof is based on a \emph{duality argument} following Theorem \ref{thrm: coupling supercritical and subcritical}, which connects the measures in the supercritical regime to an auxiliary process in the subcritical case.  Let $(G^N_t)_{t\geq 0}$ be the random graph processes described in Section  \ref{sec: coupling_to_random_graph} with points $\mathbf{x}_N$ sampled as a Poisson random measure of intensity $N\mu_0$; it is straightforward to see that Assumption \ref{hyp: B} holds. Fix $t>t_\mathrm{g}$, and let $\widetilde{G}^N_{t}$ be the graph $G^N_{t}$ with the giant component deleted. \medskip \\ 
Let $\rho_{t}(x)=\rho(t, x)$ be the survival function defined in Lemmas \ref{lemma: form of rho-t}, \ref{lemma: survival function}, and let $\widehat{\mu}_0^t(dx)=(1-\rho_{t}(x))\mu_0(dx)$. By Lemma \ref{lemma: E and U}, there exists a unique solution $(\widehat{\mu}^t_s)_{s\geq 0}$ to the equation (\ref{eq: E+G}) starting at $\widehat{\mu}^t_0$; write $\widehat{t_\mathrm{g}}(t)$ for its gelation time. \medskip \\ Let $\mathbf{y}_N=(y_i: i\le \widehat{l}^N)$ be an enumeration of the vertexes $x_i$ not belonging to the giant component in $G^N_t$. By Theorem \ref{thrm: coupling supercritical and subcritical}, we can construct a random graph $\widehat{G}^N_t$ on $\{1,...,\widehat{l}^N\}$, 
\medskip \\
In order to appeal to Lemmas \ref{lemma: convergence of random graphs}, \ref{lemma: connect critical times}, we will now verify that the desired Assumptions \ref{hyp: A}, \ref{hyp: B} hold for the vertex space $\widehat{\mathcal{V}}$.
\begin{lem}\label{lemma: conditions B1-3 for duality} Fix $t>0$, and let $\mu_0, G^N_t, \widehat{\mu}^t_0$ and $\widehat{\mathcal{V}}$ be as described above. Then Assumption \ref{hyp: B} hold for $\mathbf{y}_N$ and $\widehat{\mu}^t_0.$ \end{lem} \begin{proof} To ease notation, we write $\widehat{\mu}_0$ for $\widehat{\mu}^t_0$, $\mu^N_0$ for the initial empirical measure of the unmodified process corresponding to $\mathbf{x}_N$, and $\widehat{\mu}^N_0$ for the reduced empirical measure corresponding to $\mathbf{y}_N$: \begin{equation} \widehat{\mu}^N_0 =\frac{1}{N}\sum_{i=1}^{\widehat{l}^N} \delta_{y_i}.\end{equation} It is straightforward to see that $\widehat{\mu}^t_0$ inherits the properties in Assumption \ref{hyp: A} from $\mu_0$, and so it is sufficient to establish Assumption \ref{hyp: B}.\medskip \\ For (B1.), we note that part of the content of Theorem \ref{thrm: coupling supercritical and subcritical} is the weak convergence
\begin{equation}
    \widehat{\mu}^N_0=\frac{1}{N}\sum_{i=1}^{\widehat{l}^N} \delta_{y_i} \rightarrow \widehat{\mu} \hspace{1cm}\text{weakly, in probability}.
\end{equation} Since the vague topology is weaker than the weak topology, we immediately have the vague convergence required. Moreover, by construction, $\text{Supp}(\widehat{\mu}^N_0)\subset\text{Supp}(\mu^N_0)$, so it follows from (B1.) that $\widehat{\mu}^N_0$ is supported on $\{\pi_0=1\}$ as required.

We will now show that (B2.) follows from the previous point, together with the moment estimates for the original initial measure $\mu^N_0$. 

Fix $\xi<\infty$, and let $\chi \in C_c(S)$ be such that $1_{S_\xi} \leq \chi \leq1_{S_{\xi+1}}$. We observe that
\begin{equation} \begin{split}
\abs{\langle \pi, \widehat{\mu}_0^N\rangle - \langle \pi, \widehat{\mu}_0\rangle}  &\leq
\abs{\langle \pi\chi, \widehat{\mu}_0^N-\widehat{\mu}_0\rangle } 
 +\langle \abs{\pi} 1_{S_\xi^\mathrm{c}},\widehat{\mu}_0^N\rangle
 +\langle \abs{\pi} 1_{S_\xi^\mathrm{c}},\widehat{\mu}_0\rangle \\[1ex]
 &\leq
 \abs{\langle \pi\chi, \widehat{\mu}_0^N-\widehat{\mu}_0\rangle} 
 +\frac{C}{\xi}\langle \varphi^2,{\mu}_0^N\rangle
 +\frac{C}{\xi}\langle \varphi^2,{\mu}_0\rangle
\end{split} \end{equation} for some constant $C$, thanks to the bound in part iv) of the definition (\ref{def: BCS}). 
We now fix $\epsilon, \delta>0$. Thanks to (A2., B2.), $\langle \varphi^2, \mu^N_0\rangle$ is bounded in $L^1$ and $\langle \varphi^2, \mu_0\rangle<\infty$, and so we may choose $\xi<\infty$ such that the second and third terms are at most $\epsilon/3$ with probability exceeding $1-\delta/2$, for all $N$. For this choice of $\xi$, the first term vanishes as $N\rightarrow\infty$ by vague convergence in probability, and so is at most $\frac{\epsilon}{3}$ with probability exceeding $1-\delta/2$ for all $N$ large enough. Therefore, for all such $N$, we have \begin{equation} \PP\left(|\langle \pi, \widehat{\mu}^N_0-\widehat{\mu}_0\rangle|>\epsilon\right)\le \delta \end{equation} which proves the desired convergence in probability.  

For the second assertion of (B2.), we note that $\langle \varphi^2, \widehat{\mu}^N_0\rangle \le \langle \varphi^2, \mu^N_0\rangle$ by the construction of $\mathbf{y}_N$, and $\langle \varphi^2, \mu^N_0\rangle$ is uniformly integrable by the hypothesis (B2.).
\end{proof}

We now use this preparatory result to prove Lemma \ref{lemma: second moment finite after tgel}. \begin{proof}[Proof of Lemma \ref{lemma: second moment finite after tgel}] Let $G^N_t, \widetilde{G}^N_t, \widehat{G}^N_t$ be as above. Recalling that we consider equality of graphs to include equality of the vertex data, it follows from Theorem \ref{thrm: coupling supercritical and subcritical} that \begin{equation} \mathbb{P}(\pi_\star(\widehat{G}^N_t)=\pi_\star(\widetilde{G}^N_t))\rightarrow 1. \end{equation}  From Lemmas \ref{lemma: convergence of random graphs}, \ref{lemma: conditions B1-3 for duality}, we obtain the following convergences in probability:
\begin{equation}
    \pi_\star(G^N_{t})\rightarrow \pi_\#{\mu}_{t};\hspace{1cm}
    \pi_\star(\widehat{G}^N_{t})\rightarrow \pi_\#\widehat{\mu}^t_{t}
\end{equation} in the vague topology, in probability.  Moreover, the difference \begin{equation}
    \pi_\star(G^N_{t})-\pi_\star(\widetilde{G}^N_{t})=\frac{1}{N}\delta(\mathcal{C}_1(G^N_{t}))
\end{equation} converges to $0$ in the vague topology in probability, since the support is eventually disjoint from any compact set, with high probability. It follows that \begin{equation}
    \pi_\star(\widetilde{G}^N_{t})\rightarrow \pi_\#\mu_{t}
\end{equation} in the vague topology, in probability, and by uniqueness of limits, we have $\pi_\#\widehat{\mu}^t_{t}=\pi_\#\mu_{t}$. In particular, it follows that \begin{equation}\langle \varphi^2, \mu_t\rangle =\langle \varphi^2, \pi_\# \mu_t\rangle = \langle \varphi^2, \pi_\#\widehat{\mu}^t_t\rangle = \langle \varphi^2, \widehat{\mu}^t_t\rangle.\end{equation} Using assumption ({A2.}), we can see that $t k\in L^2(S\times S, \mu_0\times \mu_0)$, and so it follows from Theorem \ref{thrm: coupling supercritical and subcritical} that the graphs $\widehat{G}^N_{t}$ are subcritical. By Lemma \ref{lemma: connect critical times}, it follows that that  $t<\widehat{t_\mathrm{g}}(t)$, and so by Lemma \ref{lemma: second moment before tgel}, we have \begin{equation}
   \langle \varphi^2, \mu_t\rangle = \langle \varphi^2, \widehat{\mu}^t_{t}\rangle <\infty.
\end{equation} Using Theorem \ref{thrm: continuity of rho} and dominated convergence, the map \begin{equation}\begin{split}
   & t\mapsto q^t_0=\left(\left\langle \pi_i\pi_j, \widehat{\mu}^t_0\right\rangle\right)_{i,j=1}^n=\left(\left\langle (1-\rho_t)\pi_i\pi_j,\mu_0\right\rangle\right)_{i,j=1}^n; \\ & t\mapsto z^t_0=\left(\left\langle \pi_i\pi_0, \widehat{\mu}^t_0\right\rangle\right)_{i=0}^n=\left(\left\langle (1-\rho_t)\pi_i\pi_0,\mu_0\right\rangle\right)_{i=0}^n\end{split}
\end{equation} are continuous, and $q^t_0$ takes values in $E^\circ$. Therefore, by the general ODE considerations in Lemma \ref{lemma: ODE considerations} point iii)., it follows that the maps \begin{equation}
    t\mapsto q^t(t)= \psi(t, q_0^t) =\left(\langle \pi_i\pi_j, \widehat{\mu}^t_t\rangle\right)_{i,j=1}^n; \hspace{1cm} t\mapsto z^t_t=\chi(t, z^t_0, q^t_0)=(\langle \pi_i\pi_0, \widehat{\mu}^t_t\rangle)_{i=0}^n
\end{equation} are finite and continuous on $(t_\mathrm{g}, \infty).$  Since $\pi_\#\widehat{\mu}^t_t=\pi_\#\mu_t$, item iii) of Lemma \ref{lemma: ODE considerations} shows that the maps $t\mapsto \mathcal{Q}(t)_{ij}, 0\le i, j\le n$ are finite and continuous on $(t_\mathrm{g}, \infty)$, which implies that they are bounded on compact subsets.  \end{proof}

\begin{rmk} The same argument also shows that $t\mapsto \widehat{t}_\mathrm{g}(t)$ is continuous. This fact will be used later in the proof of Lemma \ref{lemma: anomalous clusters 2}. \end{rmk}

\section{\textbf{Representation and Dynamics of the Gel}} \label{sec: gel dynamics}
\subsection{\textbf{Representation Formula}}

The duality construction used in the proof of Lemma \ref{lemma: second moment finite after tgel} gives us a natural way to relate the gel data $g_t$ to the survival function $\rho_t$. This is the content of the following lemma. \begin{lem}\label{lemma: representation of M, E} Let $\mu_0$ be an initial data satisfying Assumption \ref{hyp: A}, and let $g_t=(M_t, E_t, 0)$ be the gel data for the corresponding solution to (\ref{eq: E+G}). Let $\rho_t(\cdot)$ be the corresponding survival function defined in Section \ref{sec: coupling_to_random_graph} and Appendix \ref{sec: IRG}. Then we have the equality \begin{equation}\label{eq: formula for M, E}
    g_t=\langle \rho_t \pi, \mu_0\rangle.
\end{equation} In particular, $t\mapsto g_t$ is continuous and if $t>t_\mathrm{g}$ then $M_t>0$, and $E_t>0$ componentwise. \end{lem}  Together with the identification of $\rho_t$ in Lemma \ref{lemma: form of rho-t}, this proves part 3 of Theorem \ref{thrm: Smoluchowski equation}. \begin{proof} We deal with the supercritical and subcritical/critical cases, $t>t_\mathrm{g}, t\le t_\mathrm{g}$ separately. 
\paragraph{1. Supercritical Case $t>t_\mathrm{g}$.}  Let $(\widehat{\mu}^t_s)_{s\geq 0}$ and $\widehat{t_\mathrm{g}}(t)$ be as in the proof of Theorem \ref{lemma: second moment finite after tgel}. Then, since $(\widehat{\mu}^t_s)_{s\geq 0}$ is conservative on $[0, \widehat{t_\mathrm{g}})$, and $t<\widehat{t_\mathrm{g}}(t)$, we have, for all $0\le i \le n+m$, \begin{equation}
    \langle \pi_i, \widehat{\mu}^t_t\rangle =\langle \pi_i, \widehat{\mu}^t_0\rangle = \int_{S} \pi_i(x)(1-\rho(t,x))\mu_0(dx).
\end{equation} As shown in Lemma \ref{lemma: second moment finite after tgel},  $\pi_\#\mu_t=\pi_\#\widehat{\mu}^t_t$, so we have \begin{equation} \begin{split}
    g^i_t:&=\langle \pi_i, \mu_0\rangle -\langle \pi_i, \mu_t\rangle =\langle \pi_i, \mu_0\rangle - \langle \pi_i, \widehat{\mu}^t_t\rangle \\ & =\langle \pi_i\rho_t, \mu_0\rangle \end{split}
\end{equation} as claimed. 

\paragraph{2. Subcritical and Critical Cases $t\le t_\mathrm{g}$.} For $t<t_\mathrm{g}$, the result is immediate: we have $g_t$ by definition of $t_\mathrm{g}$, and $\rho_t=0$ by Theorem \ref{lemma: survival function}. The critical case is identical, recalling from Corollary \ref{corr: gel at tgel} that $g_{t_\mathrm{g}}=0.$ \medskip \\ Continuity follows from Theorem \ref{thrm: continuity of rho} by using dominated convergence. For the final claim, if $t>t_\mathrm{g}$ then $\rho_t(x)>0$ $\mu_0$ - almost everywhere, by Lemma \ref{lemma: survival function}. By hypothesis (A3.), for all $i=1,...,n$, $\pi_i>0$ on a set of positive $\mu_0$ measure. Together, these imply that $\langle \rho_t\pi_i, \mu_0\rangle>0$, as claimed. \end{proof}

\subsection{\textbf{Gel Dynamics Beyond the Critical Time}} We now obtain point 4 of Theorem \ref{thrm: Smoluchowski equation}  as a consequence of the previous results. We have already proven the continuity of $g_t$ on the whole time interval $[0,\infty)$ and the finiteness of the second moments $q_t=(\langle  \pi_i\pi_j, \mu_t\rangle)_{i,j=1}^n$ in the supercritical regime. Therefore, it is sufficient to prove the following result.
\begin{lem}\label{lemma: dynamics after tgel} In the notation of Theorem \ref{thrm: Smoluchowski equation}, let $g_t$ be the data of the gel associated to $(\mu_t)_{t\ge 0}$. Then, for  $t\ge t_\mathrm{g}$, we have
\begin{equation} \label{eq: claimed limit}
    g^i_t=\int_{t_\mathrm{g}}^t \hspace{0.1cm}
    \sum_{j,k=1}^n \langle \pi_i\pi_j, \mu_t\rangle a_{jk}g^k_s  ds.
\end{equation} Thanks to the continuity of the second moments above $t_\mathrm{g}$, this has the differential form, holding in the classical sense, \begin{equation}  \frac{d}{dt} g^i_t=\sum_{j,k=1}^n \langle \pi_i\pi_j, \mu_t\rangle a_{jk}g^k_t.\end{equation} 
\end{lem} \begin{rmk}\label{rmk: continuity of tgelt} In proving Lemma \ref{lemma: dynamics after tgel}, we will split the growth of the gel into two terms $\mathcal{T}_1+\mathcal{T}_2$, where $\mathcal{T}_1$ represents the absorption of particles into the gel, and $\mathcal{T}_2$ represents the coagulation of smaller particles. We will show that $\mathcal{T}_2=0$, giving the claimed result; this may be expected following the relationship between gelation and blowup of the second moment $\mathcal{E}(t)$ in Lemma \ref{lemma: second moment before tgel}, and the finiteness of $\mathcal{E}$ in the supercritical regime. \end{rmk}
\begin{proof} We return to the truncated dynamics (\ref{eq:rE1}, \ref{eq: rE2}) used in the proof of Lemma \ref{lemma: E and U}. We recall that, starting at \begin{equation} \mu^\xi_0 = 1_{S_\xi}\mu_0; \hspace{1cm} g^\xi_0=\int_{x\not\in S_\xi} x\mu_0(dx)\end{equation} the solution $(\mu^\xi_t, g^\xi_t)$ to (\ref{eq:rE1}, \ref{eq: rE2}) exists and is unique, and we have \begin{equation} \label{eq: convergence to E+G}
    \mu^\xi_t= \mu_t1_{S_\xi}; \hspace{1cm} (M^\xi_t, E^\xi_t)\downarrow (M_t, E_t)\text{ as }\xi\uparrow \infty.
\end{equation} where $(\mu_t)_{t\ge 0}$ is the solution to (\ref{eq: E+G}) starting at $\mu_0$, and $(M_t, E_t)$ are the nonzero components of the associated gel data. \medskip \\ Fix $s, t$ such that $t_\mathrm{g}<s<t$. Rewriting  (\ref{eq: rE2}) and using that $P^\xi_t=0$, we have that \begin{equation}\label{eq: truncated gel dynamics}\begin{split}
    g^{\xi,i}_t-g^{\xi,i}_s&=\int_{s}^t 
    \hspace{0.1cm} \sum_{j,k=1}^n \langle \pi_i\pi_j,\mu^\xi_u\rangle a_{jk}g^{\xi,k}_u du\\&\hspace{1cm}+\frac{1}{2}\int_s^t\int_{S_\xi^2}\pi_i(x+y)1[\varphi(x+y)>\xi]\overline{K}(x,y)\mu_u(dx)\mu_u(dy)du.
\end{split}\end{equation} Let us write $\mathcal{T}_1(\xi),\mathcal{T}_2(\xi)$ for the two terms appearing in (\ref{eq: truncated gel dynamics}) for ease of notation. \medskip \\ We first show that $\mathcal{T}_1(\xi)$ converges to the expression analagous to the claimed limit in (\ref{eq: claimed limit}).
By the monotonicity $\mu^\xi_u \le \mu_u$, and local boundedness in Lemma \ref{lemma: second moment finite after tgel}, each $\langle \pi_i\pi_j, \mu^\xi_u\rangle $ is bounded, uniformly in $\xi<\infty$ and $u\in [s,t]$. It is also  straightforward to see that the truncated gel data are bounded by $g^{\xi,i}_u \le \langle \pi_i,\mu_0\rangle$, so the integrand appearing in $\mathcal{T}_1(\xi)$ is bounded. Using (\ref{eq: convergence to E+G}) and bounded convergence, we take the limit $\xi \rightarrow \infty$ to obtain  \begin{equation}
   \mathcal{T}_1(\xi)\rightarrow \int_s^t\sum_{j,k=1}^n\langle \pi_i\pi_j, \mu_u\rangle a_{jk}g^k_u du.
\end{equation} We now deal with the second term $\mathcal{T}_2(\xi)$, which we claim converges to $0$. Expanding the total rate $\overline{K}$, we have \begin{equation}\mathcal{T}_2(\xi)=\int_{s}^t \sum_{j,k=1}^n \int_{S^2}\pi_i(x)\pi_j(x)\pi_k(y)1\left[\varphi(x+y)>\xi\right]\mu_u(dx)\mu_u(dy).
\end{equation} The integrand converges to $0$ pointwise as $\xi\rightarrow\infty$, and is dominated by $\pi_i(x)\pi_j(x)\pi_k(y)$. By Lemma \ref{lemma: second moment finite after tgel}, \begin{equation}
    \sup_{u\in [s,t]}\hspace{0.1cm}\int_{S^2}\pi_i(x)\pi_j(x)\pi_k(y) \mu_u(dx)\mu_u(dy)du<\infty.
\end{equation} Therefore, by dominated convergence, $\mathcal{T}_2(\xi)\rightarrow 0$ as $\xi\rightarrow \infty$, as claimed. Combining this with the analysis of the first term, we have shown that \begin{equation} g^i_t-g^i_s=\int_s^t \hspace{0.1cm}\sum_{j,k=1}^n \langle \pi_i\pi_j, \mu_u\rangle a_{jk}g^k_u du.\end{equation} Taking $s\downarrow t_\mathrm{g}$, and using the continuity $g_s\downarrow 0$ established in Lemma \ref{lemma: representation of M, E}, we obtain the claimed result. \end{proof}
\section{\textbf{Uniform Convergence of the Stochastic Coagulant}} \label{sec: uniform convergence} We now show how previous results, describing the dynamics of $g_t$, imply convergence to their maximum values $\langle \pi_0, \mu_0\rangle$ as $t\rightarrow \infty$. Using this, we will be able to upgrade the previous result, Lemma \ref{lemma: local uniform convergence of stochastic coagulent}, on the convergence of the stochastic coagulant to \emph{uniform} convergence. \begin{lem}\label{lemma: M and E at infinity} Let $\mu_0$ be an initial measure satisfying Assumption \ref{hyp: A},  and let $g_t$ be the gel data for the associated solution $(\mu_t)_{t\ge 0})$ to (\ref{eq: E+G}). As $t\uparrow \infty$, we have \begin{equation}
   g^i_t\rightarrow g^i_\infty=\langle \pi_i, \mu_0\rangle
\end{equation} for $i=0,...,n.$ \end{lem} \begin{proof} Let us fix $1\le i\le n$, and write $g^i_\infty$ for the claimed limit $\langle \pi_i, \mu_0\rangle$; it is immediate that $g^i_t\le g^i_\infty$ for all $t\ge 0$. Choose $t_0>t_\mathrm{g}$ and $1\le j\le n$ such that $a_{ij}>0$. Thanks to Lemma \ref{lemma: representation of M, E}, $\epsilon=a_{ij}g^j_{t_0}>0$, and note also that $g^j_t$ is increasing, so that this bound holds uniformly in $t\ge t_0$. Applying Lemma \ref{lemma: dynamics after tgel} and taking $t\rightarrow \infty$, we obtain the integral inequality \begin{equation}\begin{split} \lim_{t\rightarrow \infty}\left(g^i_t-g^i_{t_0}\right)&\ge \int_{t_0}^\infty \langle \pi_i^2, \mu_s\rangle a_{ij}g^j_s ds  \ge \epsilon \int_{t_0}^\infty \langle \pi_i, \mu_s\rangle^2 ds \\ & \ge \epsilon\int_{t_0}^\infty \left(g^i_\infty-g^i_s\right)^2ds \end{split}\end{equation} where the limit on the left hand side exists since $g^i_t$ is increasing. Recalling that $g^i_t$ is bounded, the integral appearing on the right-hand side must converge, and since the integrand is decreasing in $s$, this is only possible if $(g^i_\infty-g^i_s)^2\rightarrow 0$ as $s\rightarrow \infty$, as desired. \medskip \\ We must deal separately with $\pi_0$, since $\pi_0$ does not appear in the dynamics explicitly and the argument above does not apply. For this case, we note that the monotonicity $\rho_s\le \rho_t$ whenever $s\le t$ implies that $\rho_t$ converges pointwise to a limit $\rho_\infty\le 1$. Using Lemma \ref{lemma: representation of M, E} and dominated convergence, we have, for all $i=1,...,n$ \begin{equation}\langle \pi_i \rho_\infty, \mu_0\rangle=\lim_{t\rightarrow \infty}\langle \pi_i\rho_t, \mu_0\rangle=\lim_{t\rightarrow \infty} g^i_t=\langle \pi_i, \mu_0\rangle. \end{equation} This implies the containment \begin{equation} \{\rho_\infty<1\}\subset \{\pi_i=0\}\cup \mathcal{N}_i \end{equation} for a $\mu_0$-null set $\mathcal{N}_i$, for each $i=1,...n.$ Taking an intersection, and since $\mu_0(\pi_i=0 \text{ for all }i=1,..n)=0$ by irreducibility (A4.), we see that $\rho_\infty=1,$ $\mu_0$-almost everywhere. By Lemma \ref{lemma: representation of M, E} and dominated convergence again, \begin{equation} M_t=g^0_t=\langle \pi_0\rho_t,\mu_0\rangle \rightarrow \langle \pi_0, \mu_0\rangle \end{equation} which is the claimed limit. \end{proof} 
\begin{lem} \label{lemma: uniform convergence of coagulant} Fix a measure $\mu_0$ satisfying Assumption \ref{hyp: A}, and let $(\mu_t)_{t\ge 0}$ be the associated solution to (\ref{eq: E+G}). Let $\mu^N_t$ be the stochastic coagulants, with initial data $\mu^N_0$ satisfying Assumption \ref{hyp: B}. Then we have the \emph{uniform} convergence \begin{equation} \sup_{t\ge 0} \hspace{0.1cm} d(\mu^N_t, \mu_t) \rightarrow 0\end{equation} in probability.  \end{lem} 
\begin{proof} 

From the definition of the vague topology, it is sufficient to prove that, for any $f\in C_c(S)$ with $0\leq f \leq 1$, we have the uniform convergence
$\sup_{t\geq 0}\hspace{0.1cm} \langle f, \mu^N_t-\mu_t \rangle \rightarrow 0$ in probability.

Fix $\epsilon>0$. By Lemma~\ref{lemma: M and E at infinity}, we can find $t_+\in (t_\text{g}, \infty)$ such that $M_{t_+}>\langle \pi_0, \mu_0\rangle-\frac{\epsilon}{3}.$
Let $A^1_N$ be the event \begin{equation} A^1_N=\left\{M^N_{t_+}>\langle \pi_0, \mu_0\rangle-\frac{\epsilon}{3}; \hspace{0.5cm} \langle \pi_0, \mu^N_0\rangle \le \langle \pi_0, \mu_0\rangle + \frac{\epsilon}{3} \right\}.\end{equation} By Lemma~\ref{lemma: WCOG} and condition (B2.), it follows that $\mathbb{P}(A^1_N)\rightarrow 1$.
On this event, we have
\begin{equation} \begin{split}
\sup_{t\geq 0}\hspace{0.1cm} \langle f, \mu^N_t-\mu_t \rangle
&\leq
\sup_{0 \leq t \leq t_+}\hspace{0.1cm} \langle f, \mu^N_t-\mu_t \rangle + \sup_{t >t_+}\hspace{0.1cm} \langle f, \mu^N_t-\mu_t \rangle \\ &
\leq
\sup_{0 \leq t \leq t_+}\hspace{0.1cm} \langle f, \mu^N_t-\mu_t \rangle + \left(\langle \pi_0, \mu^N_0\rangle - M_{t_+}^N\right) + \left(\langle \pi_0, \mu_0\rangle - M_{t_+}\right)
\\ & \leq
\sup_{0 \leq t \leq t_+}\hspace{0.1cm} \langle f, \mu^N_t-\mu_t \rangle + \epsilon
\end{split} \end{equation}
and the first term converges to $0$ in probability by Lemma~\ref{lemma: local uniform convergence of stochastic coagulent}.\end{proof}

\section{\textbf{Behaviour Near the Critical Point}}\label{sec: BNCP} We now prove item 5 of Theorem \ref{thrm: Smoluchowski equation}, concerning the phase transition: we will show that the gel data $g_t=(g^i_t)$ have nonnegative right-derivatives at the gelation time $t_\mathrm{gel}$. We start from the nonlinear fixed point equation (\ref{eq: NLFP 1}), which we rewrite as \begin{equation}\label{eq: NLPF again}c_t=tF(c_t); \hspace{1cm} F\left(c\right)_i=2\int_{S}\left(1-\exp\left(-\sum_{k=1}^n c_k\pi_k(x)\right)\right)\sum_{j=1}^n a_{ij}\pi_j(x)\mu_0(dx). \end{equation} The following proof is a modification of the arguments in \cite[Theorem 3.17]{BJR07}, which itself generalises an analagous, well-known result for the phase transition of Erd\H{o}s-R\'eyni graphs.
\begin{lem}\label{lemma: BNCP} Suppose that $\mu_0$ satisfies Assumption \ref{hyp: A}, and let $c_t$ be as in Lemma \ref{lemma: form of rho-t}.  Then $c_t$ is right-differentiable at $t_\mathrm{g}$, and the right-derivative $c'_{t_\mathrm{g}^+}>0$ is componentwise positive.  \end{lem} \begin{proof} Let us equip $\mathbb{R}^n$ with the inner product \begin{equation} \left(c,c'\right)_{\mu_0}= \sum_{i,j=1}^n c_ic'_j \langle \pi_i\pi_j, \mu_0\rangle   \end{equation} which is the pullback of the $L^2(\mu_0)$ inner product under $c\mapsto \sum_i c_i\pi_i$, and write $|\cdot|_{\mu_0}$ for the associated norm.  Differentiating under the integral sign twice, and using (A2.), we write \begin{equation} F\left(\begin{matrix} c \end{matrix}\right) = \Lambda c-\Sigma(c) + R\left(c\right) \end{equation} where $\Lambda$ is the $n\times n$ matrix found in Lemma \ref{lemma: computation of tcrit}, $\Sigma(\cdot)$ is a quadratic term, and $R$ is a remainder:   \begin{gather} \label{eq: defn of L}
    \Lambda_{ij}=2\sum_{k=1}^n a_{ik}\langle \pi_k\pi_j, \mu_0\rangle;
\\[2ex]
    \label{eq: defn of B}\Sigma\left(c\right)_i= \sum_{j,k,l=1}^n a_{ij}\langle \pi_j\pi_k\pi_l, \mu_0\rangle c_k c_l  \\[2ex]
 \left|R(c)\right|_{\mu_0} = o\left(|c|^2_{\mu_0}\right) \text{as  }|c|\rightarrow 0.\end{gather} The signs here are chosen to guarantee that, if $c>0$, then $\Lambda c, \Sigma(c)>0$, and $\Lambda$ is self-adjoint with respect to $(\cdot, \cdot)_{\mu_0}$. We also recall from Lemma \ref{lemma: computation of tcrit} that the largest eigenvalue of $\Lambda$ is precisely $t_\mathrm{g}^{-1}$, and the corresponding eigenspace is 1-dimensional. Let $\psi$ be an associated eigenvector, scaled so that $|\psi|_{\mu_0}=1$. We note that $\sum_i \psi_i \pi_i$ is an eigenfunction of $T$, and in particular, the sign of $\psi$ can be chosen so that $\sum_i \psi_i\pi_i>0$ is strictly positive $\mu_0$-almost everywhere; using (\ref{eq: expansion of Tf}) it follows that $\psi_i>0$ for all $i=1,...,n.$  From Lemma \ref{lemma: form of rho-t}, Theorem \ref{thrm: RG1} and Theorem \ref{thrm: continuity of rho}, we know that $c_{t_\mathrm{g}}=0$, that $c_{t_\mathrm{g}+\epsilon}\in [0,\infty)^n\setminus 0$ for all $\epsilon>0$, and that $t\mapsto c_t$ is continuous at $t_\mathrm{g}.$ \bigskip \\ Let us write $\psi^\bot$ for the orthogonal compliment of $\text{Span}(\psi)$ with respect to $(\cdot, \cdot)_{\mu_0}$. Since $\text{Span}(\psi)$ is exactly the eigenspace $\text{Ker}(\Lambda-t_\mathrm{g}^{-1}1)$, it follows from the self-adjointness of $\Lambda$ that $\Lambda$ maps $\psi^\bot$ into itself, and that, for $t>t_\mathrm{g}$ small enough, $(t\Lambda-1)|_{\psi^\bot}$ is invertible, and that the operator norm $\|(t\Lambda-1)|_{\psi^\bot}^{-1}\|_{\mu_0\rightarrow \mu_0}$ is bounded as $t\downarrow t_\mathrm{g}$. \bigskip \\ Let $P:\mathbb{R}^n\rightarrow\mathbb{R}^n$ be the orthogonal projection onto $\psi^\bot$ with respect to $(\cdot,\cdot)_{\mu_0}$, and write $c^*_t=Pc_t$ so that we have the orthogonal decomposition \begin{equation}\label{eq: orth dec} c_t=\alpha_t \psi + c^*_t \end{equation} for some $\alpha_t \in \mathbb{R}$. Noting that $\Lambda P=P\Lambda$, it follows from (\ref{eq: NLPF again}, \ref{eq: orth dec}) that \begin{equation} c^\star_t=P(tF(c_t))=t\Lambda c^\star_t + tP\left(-\Sigma(c_t)+R(c_t)\right). \end{equation} The function $-\Sigma(c)+R(c)$ is of quadratic growth as $|c|_{\mu_0}\rightarrow 0$, and using the invertibility of $(t\Lambda-I)|_{\psi^\bot}$ described above, it follows that there exists $\beta>0$ such that $
    |c^*_t|_{\mu_0} \le \beta |c_t|_{\mu_0}^2$
  whenever $|c_t|_{\mu_0}\le 1$. In turn, it follows that $|c_t|_{\mu_0}\sim \alpha_t$ as $t\downarrow t_\mathrm{g}.$ Now, using (\ref{eq: NLPF again}) and the self-adjointness of $\Lambda$, we obtain \begin{equation}
      \begin{split}
          \alpha_t&=(\psi,c_t)_{\mu_0}=(t_\mathrm{g}\Lambda\psi,c_t)_{\mu_0}=t_\mathrm{g}(\psi, \Lambda c_t)_{\mu_0} \\[1ex] &=\frac{t_\mathrm{g}}{t}(\psi, c_t)_{\mu_0}-t_\mathrm{g}\left(\psi,-\Sigma(c_t)+R(c_t)\right)_{\mu_0} \\[2ex]&=\frac{t_\mathrm{g}}{t}\alpha_t-t_\mathrm{g}\left(\psi,-\Sigma(c_t)+R(c_t)\right)_{\mu_0}.
      \end{split}
  \end{equation} We now expand to second order in $\alpha_t$; for clarity, we will number the error terms $\mathcal{T}^i_t.$ Since $|c_t|_{\mu_0}\sim \alpha_t$, it follows that that $|c^*_t|_{\mu_0}=\mathcal{O}(\alpha_t^2)$ and that $R(c_t)=o(\alpha_t^2)$. Expanding $\Sigma(c_t)$ using (\ref{eq: orth dec}),\begin{equation} -\Sigma(c_t)+R(c_t)=-\alpha_t^2\Sigma(\psi)+\mathcal{T}^1_t; \hspace{1cm} |\mathcal{T}^1_t|_{\mu_0}=o(\alpha_t^2). \end{equation} It therefore follows that \begin{equation}
      \begin{split}
          \alpha_t=t_\mathrm{g}\left(\frac{\alpha_t}{t}+\alpha_t^2(\psi, \Sigma(\psi))_{\mu_0}\right)+ \mathcal{T}^2_t; \hspace{1cm} \mathcal{T}^2_t=o(\alpha_t^2).
      \end{split} 
  \end{equation} For $t>t_\mathrm{g}$ small enough, $\alpha_t>0$, and we may rearrange to find \begin{equation} t-t_\mathrm{g}=t\hspace{0.1cm}t_\mathrm{g}\hspace{0.1cm}\alpha_t(\psi,\Sigma(\psi))_{\mu_0}+\mathcal{T}^3_t;\hspace{1cm} \mathcal{T}^3_t=o(\alpha_t) \end{equation}  and in particular $\alpha_t=\Theta(t-t_\mathrm{g})$ as $t\downarrow t_\mathrm{g}$, since \begin{equation} (\psi, \Sigma(\psi))_{\mu_0}=\sum_{i,j,k,l=1}^n a_{ij} \psi_i\psi_k\psi_l \langle \pi_j\pi_k\pi_l, \mu_0\rangle>0.\end{equation}  Finally, we obtain \begin{equation} \frac{\alpha_t}{t-t_\mathrm{g}}\rightarrow\frac{1}{t_\mathrm{g}^2(\psi,\Sigma(\psi))_{\mu_0}}\hspace{1cm} \text{as }t\downarrow t_\mathrm{g}.  \end{equation} The calculations above show that $|c_t-\alpha_t\psi|=\mathcal{O}((t-t_\mathrm{g})^2)$, and the claimed right-differentiability now follows. Finally, since $\psi_i>0$ is strictly positive componentwise and $\alpha'_{t_\mathrm{g}+}>0$, it follows that $c'_{t_\mathrm{g}+}>0$ componentwise. \end{proof} We now show how this implies item 5 of Theorem \ref{thrm: Smoluchowski equation}. From  Lemmas \ref{lemma: form of rho-t}, \ref{lemma: representation of M, E}, we have, for all $i=0,1,...,n$ \begin{equation} g^i_t=\int_{S}\left(1-\exp\left(-\sum_{j=1}^n c^j_t \pi_j(x)\right)\right)\pi_i(x)\mu_0(dx) \end{equation} Differentiating under the integral sign using hypothesis (A2.), we obtain \begin{equation}
   g^i_t=\sum_{j=1}^n c^j_t\langle \pi_i \pi_j, \mu_0\rangle + o(c_t).
\end{equation} In the notation of the previous lemma, we see that for $t>t_\mathrm{g}$, \begin{equation} \begin{split} g^i_t=(t-t_\mathrm{g})\sum_{j=1}^n (c'_{t_\mathrm{g}+})_j \langle \pi_i \pi_j, \mu_0\rangle + o(t-t_\mathrm{g}) \\ =\alpha'_{t_\mathrm{g}+}(t-t_\mathrm{g})\left\langle \sum_{j=1}^n \psi_j \pi_j \pi_i, \mu_0 \right\rangle + o(t-t_\mathrm{g}). \end{split} \end{equation} which proves the desired right-differentiability. For the positivity, since all components of $c'_{t_\mathrm{g}+}$ are strictly positive, we have the lower bound for $i=1,...,n$ \begin{equation} (g'_{t_\mathrm{g}+})_i\ge (c'_{t_\mathrm{g}+})_i \langle \pi_i^2, \mu_0\rangle >0.\end{equation} A similar argument holds for the $0^\text{th}$ component. \medskip \\   Finally, we address the size-biasing effect. We wish to choose a convex combination $\theta_i: i=1,...,n$ such that \begin{equation} \label{eq: SB} \frac{\sum_{i=1}^n \theta_i(g'_{t_\mathrm{g}+})_i}{(g'_{t_\mathrm{g}+})_0}  \ge \frac{\sum_{i=1}^n \theta_i \langle \pi_i, \mu_0\rangle}{\langle \pi_0, \mu_0\rangle}.\end{equation} Thanks to the calculation above, this is equivalent to proving that \begin{equation} \frac{\sum_{i,j=1}^n \theta_i\psi_j \langle \pi_i\pi_j,\mu_0\rangle}{\sum_{k=1}^n \psi_k\langle \pi_k, \mu_0\rangle} \ge \frac{\sum_{i=1}^n \theta_i \langle \pi_i, \mu_0\rangle }{\langle \pi_0, \mu_0\rangle}. \end{equation} If we choose $\theta_i = \psi_i/\sum_j \psi_j$, then these follow from the Cauchy-Schwarz inequality \begin{equation} \begin{split} \label{eq: CS for SB} \left\langle \sum_i \psi_i \pi_i, \mu_0\right\rangle^2 & \le \left\langle(\sum_i \psi_i \pi_i)^2, \mu_0  \right\rangle\langle 1, \mu_0\rangle \\ & =\left\langle(\sum_i \psi_i \pi_i)^2, \mu_0  \right\rangle\langle \pi_0, \mu_0\rangle.\end{split}\end{equation} We recall that the linear combination $f=\sum_i \psi_i\pi_i$ is an eigenfunction of $T$, and so can only be constant $\mu_0$-almost everywhere if $s(x)=(T1)(x)$ is constant. In particular, if $s$ is not constant $\mu_0$-almost everywhere, the inequality (\ref{eq: CS for SB}) is strict, and hence so is (\ref{eq: SB}), as desired.

\section{\textbf{Convergence of the Gel}} \label{sec: COG}

We now prove the remaining part of Theorem \ref{thrm: convergence of stochastic coagulent}, concerning the \emph{uniform} convergence of the stochastic gel, drawing on other results we have proven. We recall that $g^N_t$ are the data of the largest particle in the stochastic coagulant $\mu^N_t$; to conclude the proof of Theorem \ref{thrm: convergence of stochastic coagulent}, we must extend Lemma \ref{lemma: WCOG}, to show uniform convergence in time, in probability.  \medskip \\ Throughout this section, let $\mu_0$ be an initial measure satisfying Assumption \ref{hyp: A}, and $\mu^N_t$ be stochastic coagulants satisfying Assumption \ref{hyp: B} for this choice of $\mu_0$. We will also let $G^N_t$ be random graphs coupled to $\mu^N_t$ as described in Section \ref{sec: coupling_to_random_graph}, so that $g^N_t$ is the data of the largest component in $G^N_t$. \medskip \\ This subsection is structured as follows. We recall that, in the proof of Lemma \ref{lemma: WCOG}, we used the result on mesoscopic clusters from \cite{BJR07}: if $\xi_N\rightarrow \infty$ and $\frac{\xi_N}{N}\rightarrow 0$, then for all $t\ge 0$,  \begin{equation} \frac{1}{N}\sum_{j\ge 2: C_j(G^N_t)\ge \xi_N} C_j(G^N_t)\rightarrow 0 \end{equation} in probability. We will first state a lemma which extends this convergence to be uniform in time. Equipped with this lemma, and previous results, we will show how the proof of Lemma \ref{lemma: WCOG} can be modified to establish uniform convergence, and prove the analagous result when we sum over all clusters exceeding a deterministic size $\xi_N\ll N.$ Finally, we return to prove the preliminary lemma. \bigskip \\ The key lemma which we will require is the following, which generalises the result of Bollob\'as et al. recalled in Lemma \ref{thrm: RG2}.
\begin{lem}\label{lemma: anomalous clusters} Let $G^N_t$ be as above, and let $\xi_N$ be any sequence such that $\xi_N\rightarrow \infty$, $\frac{\xi_N}{N}\rightarrow 0.$ Then we have the uniform estimate \begin{equation} \sup_{t\ge 0}\left[\frac{1}{N}\sum_{j\ge 2: C_j(G^N_t)\ge \xi_N} C_j(G^N_t)\right] \rightarrow 0 \hspace{1cm} \text{in probability}. \end{equation}  \end{lem} The proof of this lemma will be deferred until Subsection \ref{subsec: meso} \subsection{Proof of Theorem \ref{thrm: convergence of stochastic coagulent}} It remains to prove that the convergence of the stochastic approximations $g^N_t$, $\widetilde{g}^N_t$ to the gel, given by the gel data of the largest cluster, and of all clusters exceeding a certain scale $\xi_N$ respectively.
This is the content of the following two lemmas.
\begin{lem}\label{lemma: COG1} In the notation above, we have the uniform convergence \begin{equation} \sup_{t\ge 0}\abs{g^N_t-g_t}\rightarrow 0\hspace{1cm} \text{in probability}.\end{equation}  \end{lem}
\begin{lem}\label{lemma: COG2} Fix a sequence $\xi_N$ such that $\xi_N\rightarrow \infty$ and $\frac{\xi_N}{N}\rightarrow 0$, and let $\widetilde{g}^N_t$ be given by
\begin{equation} \widetilde{g}^N_t=\frac{1}{N}\sum_{j\ge 1: C_j(G^N_t)\ge \xi_N} \pi(\mathcal{C}_j(G^N_t)) = \left(\langle \pi_i 1[\pi_0\ge \xi_N], \mu^N_t\rangle\right)_{i=0}^{n+m}.
\end{equation}
Then \begin{equation} \sup_{t\ge 0}\abs{\widetilde{g}^N_t-g^N_t}\rightarrow 0\hspace{1cm}\text{ in probability}. \end{equation}   \end{lem} We now prove these two lemmas, looking primarily at the $0^\text{th}$ coordinate. The other coordinates follow, with minor modifications which will be discussed later.
\begin{proof}[Proof of Lemma \ref{lemma: COG1}] This is an extension of the proof of Lemma \ref{lemma: WCOG}, from where much of the notation is taken.
We deal first with the $0^\text{th}$ coordinate $M^N_t-M_t$. Let $\eta_r$ be a fast-growing sequence such that $\beta(r, \eta_r)\rightarrow 0$ in the notation of Lemma \ref{lemma: STUI}, and let $S_{(r)}, \widetilde{f}_r, \widetilde{h}_r$ be as in Lemma \ref{lemma:  WCOG}. Let also $\xi_N$ be a sequence, to be constructed later, such that \begin{equation} \xi_N\rightarrow \infty; \hspace{1cm} \frac{\xi_N}{N}\rightarrow 0 \end{equation} and write $f_N=\widetilde{f}_{\xi_N}, h_N=\widetilde{h}_{\xi_N}$. We recall also the decomposition (\ref{eq: decomposition of erorr in WCOG}) \begin{equation}\label{eq: decomposition of error in NCOG} M^N_t-M_t=\sum_{i=1}^5 \mathcal{T}^i_N(t) \end{equation}
where the definitions of the error terms are given in (\ref{eq: decomposition of erorr in WCOG}).
The bounds obtained on $\mathcal{T}^3_N(t), \mathcal{T}^5_N(t)$ in the proof of Lemma \ref{lemma: WCOG} are already uniform in time; we will now show how the previous proof can be modified to estimate the other terms uniformly in time.

\paragraph{1. Estimate on $T^1_N(t)$} $\mathcal{T}^1_N(t)$ is the nonrandom error $\langle \pi_0, \mu_t\rangle -\langle \pi_0 f_N, \mu_t\rangle$. The estimate in Lemma \ref{lemma: WCOG} shows that $\langle \pi_0 f_N, \mu_t\rangle \uparrow \langle \pi_0, \mu_t\rangle$ for each fixed $t\ge 0$.
The maps $t\mapsto \langle \pi_0 f_N, \mu_t\rangle $, $t\mapsto \langle \pi_0, \mu_t\rangle $ are both continuous on $[0, \infty)$, by the definition of the Flory dynamics (\ref{eq: E+G}) and Lemma \ref{lemma: representation of M, E} respectively. Let us extend both of these maps to $[0, \infty]$ by defining both to be $0$ at $t=\infty$; the extensions are continuous, by Lemma \ref{lemma: M and E at infinity}. Therefore, by Dini's theorem, it follows that $\langle \pi_0 f_N, \mu_t\rangle \rightarrow \langle \pi_0, \mu_t\rangle$ uniformly, which implies that $\sup_{t\ge 0}\abs{T^1_N(t)}\rightarrow 0$ as desired.
\paragraph{2. Estimate on $\mathcal{T}^4_N$.}As in (\ref{eq: form of T4n}), we have the equality, for all $t\ge 0$  \begin{equation}
           \begin{split}
                \mathcal{T}^4_N(t)=-M^N_t1\left(M^N_t\le \frac{\xi_N}{N}\right)-\frac{1}{N}\sum_{j\ge 2: C_j(G^N_t)\ge \xi_N}\pi_0(\mathcal{C}_j(G^N_t))
          \end{split} 
       \end{equation} where we have used the coupling of the random graphs $(G^N_t)_{t\geq 0}$ to the stochastic coagulant. Therefore, we have the uniform bound \begin{equation}\label{eq: uniform T4N} \sup_{t\ge 0}\abs{\mathcal{T}^4_N(t)} = \langle \varphi^2, \mu^N_0\rangle^\frac{1}{2}\left(\sup_{t\ge 0}\frac{1}{N}\sum_{j\ge 2:C_j(G^N_t)\ge \xi_N} C_j(G^N_t)\right)^\frac{1}{2}+\frac{\xi_N}{N} \end{equation} which converges to $0$, by Lemma \ref{lemma: anomalous clusters}, (B2.), and because $\xi_N\ll N.$
\paragraph{3. Construction of $\xi_N$, and convergence of $\mathcal{T}^2_N$.} To conclude the proof of the supercritical case, it remains to show how a sequence $\xi_N$ can be constructed such that $\mathcal{T}^2_N \rightarrow 0$ uniformly, in probability. Recalling the definitions of $\tilde{f}_r$ above, let $A_{r,N}$ be the events \begin{equation} \label{eq: definition of A1rn}
       A_{r,N}=\left\{\sup_{t\geq 0} |\langle \pi_0 \widetilde{f}_r, \mu^N_t-\mu_t\rangle|<\frac{1}{r}\right\}.
   \end{equation} Then, as $N\rightarrow \infty$ with $r$ fixed, $\mathbb{P}(A^1_{r,N}) \rightarrow 1$ by Lemma \ref{lemma: uniform convergence of coagulant}. We now define $N_r$ inductively for $r\geq 1$ inductively, as in Lemma \ref{lemma: WCOG}, by setting $N_1=1$ and letting $N_{r+1}$  be the minimal $N>N_r$ such that, for all $N'\ge N$, \begin{equation}
       \label{eq: recursive definition of Nr} N\geq (r+1)^2;\hspace{1cm}  \mathbb{P}(A_{r+1,N'})>\frac{r}{r+1}.
   \end{equation} Now, we set $\xi_N=r$ for $N\in [N_r, N_{r+1}).$ It follows that $\xi_N$ satisfies the requirements above, and \begin{equation}
       \mathbb{P}\left(\sup_{t\geq 0} |\mathcal{T}^2_N| <\frac{1}{\xi_N}\right) \ge \mathbb{P}\left(A^1_{\xi_N,N}\right) > 1-\frac{1}{\xi_N}\rightarrow 1
   \end{equation} Therefore, with this choice of $\xi_N$, $\mathcal{T}^2_N \rightarrow 0$ uniformly in probability on $t\ge 0.$ \bigskip \\ This concludes the proof for the $0^\text{th}$ coordinate $M^N_t$; the $1^\text{st}$ - $n^\text{th}$ coordinates are identical. For the remaining $m$ coordinates, we replace $f_N$ by $\frac{1}{2}(f_N(x)+f_N(Rx))$, which makes $\mathcal{T}^1_N$ identically $0$ by symmetry, and use the bound $\pi_i(x)^2\le c\varphi(x)^2$ in estimating $\mathcal{T}^4_N$. \end{proof} \begin{proof}[Proof of Lemma \ref{lemma: COG2}] We now turn to the case where, instead of considering the largest cluster, we sum over the (possibly empty) set of clusters of size at least $\xi_N$, for a deterministic sequence $\xi_N.$ In this way, we have \begin{equation}\widetilde{g}^N_t=\langle \pi 1[\pi_0 \ge \xi_N], \mu^N_t\rangle. \end{equation} Let us write $h_N(x)=1[\pi_0(x)<\xi_N]$, so that $\widetilde{g}^N_t=\langle \pi, \mu^N_0\rangle -\langle \pi h_N, \mu^N_t\rangle.$ With this notation, \begin{equation} g^N_t-\widetilde{g}^N_t=\langle \pi h_N, \mu^N_t\rangle -\left(\langle \pi,\mu^N_0\rangle -g^N_t\right)  \end{equation} is exactly the term $\mathcal{T}^4_N$ estimated in the proofs of Lemma \ref{lemma: WCOG}, \ref{lemma: COG1}, for the new choice of $\xi_N$. The estimate (\ref{eq: uniform T4N}) therefore applies to bound $\sup_{t\ge 0}\abs{g^N_t-\widetilde{g}^N_t}$, and the hypotheses on $\xi_N$ are sufficient to guarantee that the right-hand side converges to $0$ in probability. \end{proof}

\subsection{\textbf{Proof of Lemma \ref{lemma: anomalous clusters}}}\label{subsec: meso} We now turn to the proof of Lemma \ref{lemma: anomalous clusters}; our strategy is as follows. First, we prove uniform convergence  on compact subsets $I\subset (t_\mathrm{g}, \infty)$ in Lemma \ref{lemma: anomalous clusters 2}. We will then show how this may be extended to the whole interval $[0, \infty)$, by arguing separately for an initial interval $[0, t_-]$ and for large times $[t_+, \infty).$
\begin{lem}\label{lemma: anomalous clusters 2}
       Let $G^N_t$ and $\xi_N$ be as above. Fix a compact subset $I\subset (t_\mathrm{g}, \infty)$. Then we have the convergence \begin{equation} \sup_{t \in I}\left[\frac{1}{N}\sum_{j\geq 2: C_j(G^N_t)\geq \xi_N} C_j(G^N_t)\right] \rightarrow 0 \hspace{1cm}\text{in probability.}\end{equation}
\end{lem}
\begin{proof}[Proof of Lemma \ref{lemma: anomalous clusters 2}]  It is sufficient to show that for every $t>t_\mathrm{g}$ the claim holds for some $I$ of the form $I=(t_-, t_+) \subset (t_\mathrm{g}, \infty)$ containing $t$.
As in Theorem \ref{lemma: second moment finite after tgel}, let $\widehat{\mu}_0^t$ be the measure on $S$ given by $\widehat{\mu}_0^t(dx)=(1-\rho_t(x))\mu_0(dx)$. We also write $\widehat{t_\mathrm{g}}(t)$ for the gelation time of the solution $(\widehat{\mu}^t_s)_{s\ge 0}$ to (\ref{eq: E+G}) starting at $\widehat{\mu}^t_0$. We showed in the proof of Theorem \ref{lemma: second moment finite after tgel} that, for all $t>t_\mathrm{g}$,  $\widehat{t_\mathrm{g}}(t)>t$, and the map $t\mapsto \widehat{t_\mathrm{g}}(t)$ is continuous. Therefore, for any $t>t_\mathrm{g}$, we can choose $t_\pm$ such that
\begin{equation}
    t_\mathrm{g}<t_-<t<t_+<\widehat{t_\mathrm{g}}(t_-).
\end{equation} We form $\widetilde{G}^N_{t_-}$ from $G^N_{t-}$ by deleting all vertexes of the giant component of $C_1(G^N_{t_-})$. We now form a new graph, $\widetilde{G}^N_{t_-,t_+}$ by including all edges between vertexes of $\widetilde{G}^N_{t_-}$ which are present in the graph $G^N_{t_+}$. \medskip \\ From Theorem \ref{thrm: coupling supercritical and subcritical} and Lemma \ref{lemma: conditions B1-3 for duality}, we can construct a sequence $\mathbf{y}_N, N\ge 1$ satisfying Assumption \ref{hyp: B} for $\widehat{\mu}^{t_-}_0$ and random graphs $\widehat{G}^N_{t_-}\sim \mathcal{G}(\mathbf{y}_N, t_-K/N)$, such that \begin{equation}
    \mathbb{P}\left(\widehat{G}^N_{t_-}=\widetilde{G}^N_{t_-}\right)\rightarrow 1.
\end{equation} We now form $\widehat{G}^N_{t_-,t_+}$ from $\widehat{G}^N_{t_-}$ by adding those edges present in $G^N_{t_+}$. By the Markov property of the graph process $(G^N_s)_{t\geq 0}$, these edges are independent of the construction of $\widehat{G}^N_{t_-}$, and so $\widehat{G}^N_{t_-,t_+}\sim \mathcal{G}(\mathbf{y}_N, t_+K/N)$. \medskip \\ Since Assumption \ref{hyp: B} applies to $\mathbf{y}_N$ and $\widehat{\mu}^{t_-}_0$, Lemma \ref{lemma: connect critical times} shows that the critical time for $\mathcal{G}(\mathbf{y}_N, tK/N)$ is exactly the gelation time of $(\widehat{\mu}^{t_-}_s)_{s\ge 0}$, which we have written as $\widehat{t}_\mathrm{g}(t_-)$.  By the choices of $t_\pm$, $t_+<\widehat{t}_\mathrm{g}(t_-)$, and in particular, $\widehat{G}^N_{t_-,t_+}$ is still subcritical. By construction, \begin{equation}\label{eq: coupling compactness}
    \mathbb{P}\left(\widehat{G}^N_{t_-, t_+}=\widetilde{G}^N_{t_-, t_+}\right)\rightarrow 1.
\end{equation} For $s\in [t_-, t_+]$, let $\mathcal{C}_1'(G^N_s)$ be the connected component of $G^N_s$ which contains $\mathcal{C}_1(G^N_{t_-})$, and let $C_1'(G^N_s)$ be its size. By definition, $C'_1(G^N_s)\le C_1(G^N_s)$ and so \begin{equation} \label{eq: C1 vs Cprime1}\sum_{j\ge 2} C_j(G^N_s)1\left[C_j(G^N_s)\ge \xi_N\right]\hspace{0.3cm} \le \hspace{0.3cm}\sum_{j\ge 1} C_j(G^N_s)1\left[C_j(G^N_s)\ge \xi_N, \mathcal{C}_j(G^N_s)\neq \mathcal{C}'_1(G^N_s)\right].  \end{equation} Moreover, the right-hand side is increasing as $s$ runs over $[t_-, t_+]$, since it can be rewritten as \begin{equation} .... = \sum_{i=1}^{l^N} 1\left[\exists j: i \in \mathcal{C}_j(G^N_s),\hspace{0.1cm} C_j(G^N_s)\ge \xi_N,\hspace{0.1cm} i\not \in \mathcal{C}_1(G^N_{t_-})\right] \end{equation} and each summand can only increase in $s$ as the clusters grow. Evaluating at the endpoint $t_+$, the construction of $\widetilde{G}^N_{t_-,t_+}$ gives \begin{equation}\label{eq: evaluate at endpoint} \sum_{j\ge 1} C_j(G^N_{t_+})1\left[C_j(G^N_{t_+}\ge \xi_N, \mathcal{C}_j(G^N_{t_+})\neq \mathcal{C}'_1(G^N_{t_+})\right] = \sum_{j\ge 1} C_j(\widetilde{G}^N_{t_-,t_+})1\left[C_j(\widetilde{G}^N_{t_-,t_+})\ge \xi_N\right]. \end{equation} 
Combining (\ref{eq: coupling compactness}, \ref{eq: C1 vs Cprime1}, \ref{eq: evaluate at endpoint}) we see that, with high probability,
\begin{equation}
    \sup_{s\in [t_-, t_+]} \left[\frac{1}{N}\sum_{j\geq 2: C_j(G^N_s)\geq \xi_N} C_j(G^N_s)\right] \leq \frac{1}{N}C_1(\widehat{G}^N_{t_-, t_+})+\frac{1}{N}\sum_{j\geq 2: C_j(\widehat{G}^N_{t_-,t_+})\geq \xi_N} C_j(\widehat{G}^N_{t_-,t_+}).
\end{equation}
The first term of the right-hand side converges to $0$ in probability  because $\widehat{G}^N_{t_-,t_+}$ is subcritical, and the second term converges to $0$ in probability by Theorem \ref{thrm: RG2}. \end{proof} 
\begin{proof}[Proof of Lemma \ref{lemma: anomalous clusters}] For $\mu_0$ as in the hypothesis, let $M_t$ be the mass of the gel associated to the solution $(\mu_t)_{t\ge 0}$ to (\ref{eq: E+G}). Fix $\epsilon >0$; without loss of generality, assume that $\epsilon<1.$ By continuity from Lemma \ref{lemma: representation of M, E} and Lemma \ref{lemma: M and E at infinity}, we can choose $t_\pm \in (t_\mathrm{g}, \infty)$ such that \begin{equation}
    M_{t_-}<\frac{\epsilon}{3};\hspace{1cm}M_{t_+}>\mu_0(S)-\frac{\epsilon}{3}.
\end{equation}  Consider now the events \begin{equation}
    A^1_N=\left\{\frac{1}{N}C_1(G^N_{t_-})<\frac{2\epsilon}{3}; \hspace{0.6cm}\frac{1}{N} C_1(G^N_{t_+})>\mu_0(S)-\frac{\epsilon}{2}; \hspace{0.6cm} \langle 1, \mu^N_0\rangle <\mu_0(S)+\frac{\epsilon}{2} \right\};
\end{equation} \begin{equation}
    A^2_N=\left\{\frac{1}{N}\sum_{j\geq 2: C_j(G^N_{t_-})\geq \xi_N} C_j(G^N_{t_-})<\frac{\epsilon}{3} \right\}.
\end{equation} Thanks to the coupling described in Section \ref{sec: coupling_to_random_graph}, Lemma \ref{lemma: WCOG} implies that $\mathbb{P}(A^1_N)\rightarrow 1 $, and  $\mathbb{P}(A^2_N)\rightarrow 1$ from Theorem \ref{thrm: RG2}. On the event $A^1_N \cap A^2_N$, we bound as follows. 
\begin{enumerate}[label=\roman{*}).]
    \item For the initial interval $[0, t_-]$, an argument similar to that of Lemma \ref{lemma: anomalous clusters 2} shows that, on this event,
 
    \begin{equation} 
        \sup_{t\in [0, t_-]} \frac{1}{N} \sum_{\substack{j\geq 2: \\ C_j(G^N_t)\geq \xi_N}} C_j(G^N_t) \leq 
        \frac{1}{N}\sum_{\substack{j\geq 1: \\ C_j(G^N_{t_-}) \geq \xi_N}} C_j(G^N_{t_-}) = 
        \frac{1}{N}C_1(G^N_{t_-})+\frac{1}{N}\sum_{\substack{j\geq 2: \\ C_j(G^N_{t_-}) \geq \xi_N}} C_j(G^N_{t_-}) <\epsilon.
   \end{equation}
    \item For late times $t\in [t_+, \infty)$, the largest cluster $\mathcal{C}_1(G^N_t)$ is at least the size of the cluster containing $\mathcal{C}_1(G^N_{t_+})$. Therefore, \begin{equation}
        \inf_{t\geq t_+} \frac{1}{N}C_1(G^N_t)\geq \frac{1}{N}C_1(G^N_{t_+})>\mu_0(S)-\frac{\epsilon}{2}
    \end{equation} and so 
    \begin{equation}
        \sup_{t\geq t_+} \left[\frac{1}{N} \sum_{j\geq 2: C_j(G^N_t)\geq \xi_N} C_j(G^N_t)\right] \leq \sup_{t\geq t_+} \left[ \frac{1}{N} \sum_{j\geq 2} C_j(G^N_t)\right] \le \langle 1, \mu^N_0\rangle -\frac{1}{N}C_1(G^N_{t_+})<\epsilon.
    \end{equation} 
\end{enumerate}
Now, consider the events
\begin{equation}
    A^3_N=\left\{\sup_{t\in [t_-, t_+]}\left[\frac{1}{N}\sum_{j\geq 2: C_j(G^N_{t})\geq \xi_N} C_j(G^N_{t})\right]<\epsilon \right\};\end{equation}
    \begin{equation}
    A_N=A^1_N\cap A^2_N\cap A^3_N.\end{equation} By Lemma \ref{lemma: anomalous clusters 2}, $\mathbb{P}(A^3_N)\rightarrow 1$, and so $\mathbb{P}(A_N) \rightarrow 1$. On the event $A_N$, we have \begin{equation}
        \sup_{t\geq 0}\left[\frac{1}{N}\hspace{0.1cm} \sum_{j\geq 2: C_j(G^N_t)\geq \xi_N} C_j(G^N_t)\right] <\epsilon
    \end{equation} which proves the claimed convergence in probability. 
\end{proof} 
\begin{appendices}

\renewcommand\thesection{\alph{section}}
\section{Weak Formulation of Smoluchowski and Flory Equations}\label{sec: requirements} Throughout, we work with the weak formulation of the Smoluchowski and Flory equations described in the introduction. In order to make sense of every term for a putative solution $(\mu_t)_{t<T}$, we ask for the following conditions to hold. 
\begin{enumerate}[label=\roman{*}).] \item For all Borel sets $A\subset S$, the map $t\mapsto \mu_t(A)$ is measurable; \item For all bounded, measurable functions $f:S\rightarrow \mathbb{R}_+$ of compact support, $\langle f, \mu_0\rangle<\infty$; \item For all compact subsets $S'\subset S$ and all $t<T$, \begin{equation} \int_0^t ds \int_{S'\times S}\overline{K}(x,y)\mu_s(dx)\mu_s(dy)<\infty; \end{equation} \end{enumerate} If these hold, then we say can make sense of the following weak form of the Smoluchowski equation (\ref{eq: E}).\begin{enumerate}[label=\roman*).]\setcounter{enumi}{3} \item For all $f\in C_c(S)$ and $t<T$, \begin{equation}
    \langle f, \mu_t \rangle =  \langle f,\mu_0\rangle +\int_0^t \langle f, L(\mu_s)\rangle ds.
\end{equation} \end{enumerate} 
\section{\textbf{Introduction to Inhomogenous Random Graphs}}\label{sec: IRG}
As discussed in the introduction, the connection between gelation and random graphs is well-understood, and the multiplicative kernel corresponds to the well-known Erd\H{o}s-R\'eyni random graphs \cite{Flo41,ER60,A99}. However, for our purposes, not all particles are equal: particles with large values of $\pi_i(x)$ will undergo more collisions and exhibit quantitatively different behaviour, and so we will need a more sophisticated model of random graphs to accommodate this inhomogeneity. In this section, we will review the theory of \emph{inhomogenous random graphs} developed in \cite{BJR07}, which will play the same r\^ole for our model that the Erd\H{o}s-R\'eyni model does for the multiplicative kernel. We now summarise the key definitions and results from \cite{BJR07} which we use in our work.
\begin{definition} \label{def: Generalised vertex space} A \emph{generalised vertex space} is a triple $\mathcal{V}=(\mathcal{S}, m, (\mathbf{x}_N)_{N\geq 1})$, consisting of \begin{itemize}
    \item A separable metric space $\mathcal{S}$, equipped with its Borel $\sigma$-algebra;
    \item A measure $m$ on $\mathcal{S}$, with $m(\mathcal{S}) \in (0, \infty)$; 
    \item A family of random variables $\mathbf{x}_N=(x^{(N)}_1,...,x^{(N)}_{l^N})$ taking values in $\mathcal{S}$, and of potentially random length $l^N$,  such that the empirical measures \begin{equation}
        m_N=\frac{1}{N}\sum_{k=1}^{l^N} \delta_{x^{(N)}_k} \end{equation} converge to $m$ in the weak topology $\mathcal{F}(C_b(\mathcal{S}))$, in probability.

\end{itemize} In the special case where $m(\mathcal{S})=1$ and $l^N=N$, we say that $(\mathcal{S}, m, (\mathbf{x}_N)_{N\geq 1})$ is a \emph{vertex space}. \end{definition}  
\begin{definition}\label{def: kernel}
    A \emph{kernel} is a symmetric, measurable map $k: \mathcal{S}\times \mathcal{S} \rightarrow [0, \infty).$ We say that $k$ is \emph{irreducible} if, whenever $A\subset S$ is such that $k(x,y)=0$ for all $x\in A$ and $y\in A^\mathrm{c}$, then either $m(A)=0$ or $m(A^\mathrm{c})=0$.
    \end{definition}  \begin{definition}[Inhomogenous random graphs]\label{definition of GN} Given a kernel $k$ and a generalised vertex space $\mathcal{V}$, we let
    $G^N$ be a random graph on $\{1, 2,..,l^N\}$ given as follows. Conditional on the values of $\mathbf{x}_N$, the edge $e=(ij)$ is included with probability \begin{equation}
        p_{ij}=1-\exp\left(-\frac{k(x^{(N)}_i,x^{(N)}_j)}{N} \right)
    \end{equation} and such that the presence of different edges is (conditionally) independent. We write $G^N\sim\mathcal{G}^\mathcal{V}(N,k)$. We also consider the \emph{vertex data}  $\mathbf{x}_N=(x^{(N)}_i)_{i=1}^{l^N}$ to be part of the data of $G^N_t$, so that an equality of random graphs $G=G'$ includes the equality of the vertex data.
\end{definition}\begin{remark} This differs slightly from the main definition in \cite{BJR07}, but is rather one of the alternatives considered in \cite{BJR07}[Remark 2.4] \end{remark}  To treat a general class of kernels $k$, additional regularity is required, to prevent pathologies. This is the content of the following defintion: \begin{definition}[Graphical Kernel] \label{def: graphical kernel}
    We say that a kernel $k$ on a vertex space $\mathcal{V}=(\mathcal{S}, m, (\mathbf{v}_N)_{N\geq 1})$ is \emph{graphical} if the following hold. 
    \begin{enumerate}[label=\roman{*}).]
        \item $k$ is almost everywhere continuous on $\mathcal{S}\times\mathcal{S};$
        \item $k \in L^1(\mathcal{S}\times \mathcal{S}, m \otimes m)$;
        \item If $G^N \sim \mathcal{G}^\mathcal{V}(N,k)$, then
        \begin{equation}
            \frac{1}{N}\mathbb{E}\left[e\left(G^N\right)\right]\rightarrow \frac{1}{2}\int_{\mathcal{S}\times \mathcal{S}} k(v,w)m(dv)m(dw)
        \end{equation} where $e(\cdot)$ denotes the number of edges of the graph.
    \end{enumerate}
\end{definition}  \begin{definition}
 Given a graph $G$, we write $\mathcal{C}_j(G): j=1, 2...$ for the connected components of $G$, in decreasing order of their sizes $\#\mathcal{C}_j(G)=C_j(G)$. If there are fewer than $j$ connected components, then $\mathcal{C}_j(G)=\emptyset$ and $C_j(G)=0$.
\end{definition}The phase transition is given in terms of the convolution operator
\begin{equation}
       (T f)(v)=\int_{\mathbb{R}^d} k(v,w)f(w)m(dw) 
   \end{equation} for functions $f$ such that the right-hand side is defined (i.e., finite or $+\infty$) for $m$-almost all $v$; for instance, if $f\ge 0$ then $Tf$ is well-defined, possibly taking the value $\infty$. We define \begin{equation} \|T\|=\sup\{\|Tf\|_{L^2(m)}: \|f\|_{L^2(m)}\le 1, f\ge 0\}. \end{equation} If $T$ defines a bounded linear map from $L^2(m)$ to itself, then $\|T\|$ is precisely its operator norm in this setting; otherwise, $\|T\|=\infty.$ It is straightforward to show that if $k\in L^2(S\times S, m\otimes m)$ then $T: L^2(m)\rightarrow L^2(m)$ is a Hilbert-Schmidt operator, and that $\|T\|_\text{HS}=\|k\|_{L^2(m)}<\infty$. In this case, $\|T\|$ is certainly finite, and is the operator norm of $T: L^2(m)\rightarrow L^2(m)$. The example of interest to us will fall into this case.
   
   The analysis of the random graphs uses a branching process, similar to that used in the standard analysis of Erd\H{o}s-R\'enyi graphs. Many quantities of the graph can be expressed in terms of the `survival probability' $\rho(k, v)$ when the data $v$ of the first vertex is given. To avoid the unnecessary complication of making this into a precise definition, we use the following characterisation, which is equivalent by \cite[Theorem 6.2]{BJR07}.
   \begin{theorem}
       \label{lemma: survival function}
       Let $k$ be an irreducible kernel on a generalised vertex space $\mathcal{V}$, such that $k \in L^1(\mathcal{S}\times \mathcal{S}, m \times m)$, and such that, for all $x,$ \begin{equation} \label{eq: BJR 51}
           \int_S k(x,y)m(dy)<\infty.
       \end{equation} Consider the nonlinear fixed-point equation 
      \begin{equation} \label{eq: nonlinear fixed point equation2} 
        \forall x \in S,\hspace{1cm}  {\rho}(x)=1-e^{-(T{\rho})(x)}
      \end{equation} where $T$ is the convolution operator (\ref{eq: T}). Then (\ref{eq: nonlinear fixed point equation2}) has a maximal solution $\rho_k(x)=\rho(k;x)$; that is, for any other solution $\tilde{\rho}$, \begin{equation}
          \forall x \in S, \hspace{1cm} \tilde{\rho}(x)\leq \rho(k,x).
      \end{equation} It therefore follows that $0\leq \rho_k(x)\leq 1$ for all $x$. The maximal solution is necessarily unique, and so this uniquely defines $\rho_k.$ Moreover, we have the following dichotomy:
      \begin{enumerate}[label=\roman{*}).]
          \item If $\|T\|\leq 1$, then $\rho(k, x)=0$ for all $x$;
          \item If $\|T\|> 1$, then $\rho(k, x)>0$ for all $m$-almost all $x$.
      \end{enumerate} This can be stated dynamically as follows. Consider the survival function `at time $t$', given by $\rho(tk,x)$, which we will write throughout as $\rho_t(x)$. Then 
      \begin{itemize}
          \item If $t\leq \|T\|^{-1}$, then $\rho_t(x)=0$ for all $x$;
          \item If $t>\|T\|^{-1}$, then $\rho_t(x)>0$ for all $x$.
      \end{itemize}
      
   \end{theorem} We can now state the main results on the phase transition, given by \cite[Theorem 3.1 and Corollary 3.2]{BJR07}.
   \begin{theorem}[Phase Transition]\label{thrm: RG1} Let $k$ be a graphical and irreducible kernel for a vertex space $\mathcal{V}$, with $0<\|T\|< \infty.$ Let $G^N\sim \mathcal{G}^\mathcal{V}(N, k)$ be random graphs on a common probability space. Then we have the convergence \begin{equation}
       \frac{1}{N}C_1(G^N_t)\rightarrow \int_{\mathcal{S}} \rho(tk, v) m(dv) \hspace{1cm} \text{in probability.}
   \end{equation}
   Therefore, if $(G^N_t)_{t\geq 0}$ is a dynamic family of random graphs $
       G^N_t \sim \mathcal{G}^\mathcal{V}(N, tk)$, then we have the following dichotomy:  \begin{enumerate}[label=\roman{*}).]
       \item If $t\leq t_\mathrm{c}=\|T\|^{-1}$, then there is no giant component, in particular \begin{equation}
           \frac{C_1(G^N_t)}{N} \rightarrow 0
       \end{equation} in probability.
       \item If $t>t_\mathrm{c}=\|T\|^{-1}$, then there is a giant component: there exists $c=c(t)>0$ such that
       \begin{equation}
           \mathbb{P}(C_1(G^N_t)>cN)\rightarrow 1.
       \end{equation}
   \end{enumerate}\end{theorem}
   \begin{remark} Following \cite{BJR07}, based on this dichotomy, we say that \begin{enumerate}[label=\roman{*}).]
       \item $G^N$ is \emph{subcritical} if $\|T\|<1;$
       \item $G^N$ is \emph{critical} if $\|T\|=1;$
       \item $G^N$ is \emph{supercritical} if $\|T\|>1.$
   \end{enumerate} \end{remark} The next result characterises $t_\mathrm{g}$ in terms of the point spectrum $\sigma_\mathrm{p}(T)$ as an operator on $L^2(m)$, and appears as \cite[Lemma 5.15]{BJR07} \begin{theorem}[Spectrum of $T$]\label{lemma: spectrum of T} Let $\mathcal{V}$ be a generalised vertex space and $k$ be a graphical, irreducible kernel on $\mathcal{V}$ such that $k\in L^2(S\times S,m\times m)$. Then the operator $T$ defined in (\ref{eq: T}) has an eigenvalue $t_\mathrm{c}^{-1}=\|T\|$ in $L^2(m)$, and the corresponding eigenspace is 1-dimensional. Moreover, there exists an eigenfunction $f$ such that $f>0$ $m$-almost everywhere. \end{theorem}  The third result we will recall is \cite[Theorem 3.6]{BJR07}, which considers clusters of a scale $\xi_N\ll N$, excluding the largest cluster. We term these \emph{mesoscopic} clusters.
   \begin{theorem}\label{thrm: RG2} Let $G^N\sim \mathcal{G}^\mathcal{V}(N, k)$, for a (generalised) vertex space $\mathcal{V}$ and an irreducible graphical kernel $k$. Let $\xi_N$ be a sequence with  \begin{equation}
       \xi_N\rightarrow \infty; \hspace{1cm} \frac{\xi_N}{N}\rightarrow 0.
   \end{equation} Then \begin{equation}
       \frac{1}{N}\sum_{j\geq 2: C_j(G^N)\geq \xi_N}C_j(G^N) \rightarrow 0
   \end{equation} in probability. \end{theorem}   We will also make use of the following monotonicity and continuity properties, from \cite[Theorem 6.4]{BJR07}.
   \begin{theorem}\label{thrm: continuity of rho} Let $k$ be a kernel on a vertex space $\mathcal{V}$, and let $\rho_t(\cdot)=\rho(tk,\cdot)$ be the survival function defined above. Then the map $t\mapsto \rho_t(\cdot)$ is monotonically increasing, in the sense that for all $0\leq s \leq t$ and for all $x$, $\rho_s(x)\le \rho_t(x).$ We also have the following continuity property. Let $t_n\rightarrow t$ be a monotone sequence, either increasing or decreasing. Then \begin{equation}
       \rho_{t_n}(x)\rightarrow \rho_t(x) \hspace{1cm} \text{for $m$- almost all }x, \text{ and}
   \end{equation} \begin{equation}
       \int_{\mathcal{S}}\rho_{t_n}(x)m(dx)\rightarrow \int_{\mathcal{S}}\rho_t(x)m(dx).
   \end{equation} \end{theorem} The final result which we will need is a `duality' result, connecting the supercritical and subcritical behaviours. This is given by \cite[Theorem 12.1]{BJR07}.
   \begin{theorem}\label{thrm: coupling supercritical and subcritical} Let $k$ be an irreducible graphical kernel on a generalised vertex space $\mathcal{V}$, such that $\|T\|>1$. Let $G^N \sim \mathcal{G}^\mathcal{V}(N, k)$, and form $\widetilde{G}^N$ by deleting all vertexes in the largest component $\mathcal{C}_1(G^N).$ Then, defined on the same underlying probability space, there is a generalised vertex space $\widehat{\mathcal{V}}=(\mathcal{S}, \widehat{m}, (\mathbf{y}_N)_{N\geq 1})$ with \begin{equation}
       \widehat{m}(dx)=(1-\rho(k;x))m(dx)
   \end{equation} and such that $\mathbf{y}_N$ is an enumeration of those $x_i$ not belonging to the component $\mathcal{C}_1(G^N)$, and a random graph $\widehat{G}^N \sim \mathcal{G}^{\widehat{\mathcal{V}}}(N,k)$ such that \begin{equation}
       \mathbb{P}(\widetilde{G}^N=\widehat{G}^N)\rightarrow 1.
   \end{equation}
   Furthermore, if $k\in L^2(\mathcal{S}\times \mathcal{S}, m\otimes m)$, then $\widehat{G}^N$ is subcritical.\end{theorem} We emphasise here that we have defined the equality $\widetilde{G}^N=\widehat{G}^N$ to include equality of the values $x_i$ associated to each vertex; this follows from the construction in \cite{BJR07}, since the values $\mathbf{y}_N$ associated to $\widehat{G}^N$ are exactly those $x_i$ not belonging to the giant component. This generalises the standard `duality result' of Bollob\'as \cite{BB84} for Erd\H{o}s-R\'enyi graphs. 
\end{appendices}



\end{document}